\documentclass{mhkr}
\usetikzlibrary{cd}
\usepackage{bussproofs}
\usepackage{turnstile}
\newenvironment{mathprooftree}
  {\varwidth{.9\textwidth}\centering\leavevmode}
  {\DisplayProof\endvarwidth}

\title{Dilator-based analysis of $\KP$}
\author{Hanul Jeon}
\email{ \href{mailto:hanuljeon95@gmail.com}{hanuljeon95@gmail.com}}
\urladdr{ \href{https://hanuljeon95.github.io}{https://hanuljeon95.github.io} }
\address{Institute of Discrete Mathematics and Geometry, Vienna University of Technology. Wiedner Hauptstra\ss e 8–10, 1040 Vienna, Austria} 
\thanks{I would like to thank Alvaro Pintado and Henry Towsner for helpful discussions. Parts of this work were carried out while the author was a PhD candidate at Cornell University. The author used CharGPT 5.6 to organize the introduction, find and fix errors, and write proofs of the lemmas and propositions. The main idea of the paper is due to the author, and the author takes full responsibility for its contents. The research of the author is supported in part by NSF grant DMS--2451350.}

\newcommand{\lag}{\langle}
\newcommand{\rag}{\rangle}

\newcommand{\rank}{\operatorname{rank}}
\newcommand{\field}{\operatorname{field}}

\newcommand{\supp}[1][]{\ifthenelse{\equal{#1}{}}{}{#1\mhyphen}\operatorname{supp}}

\newcommand{\Lin}{\operatorname{Lin}}

\newcommand{\Tr}{\operatorname{Tr}}
\newcommand{\en}{\operatorname{en}}

\newcommand{\trcl}{\operatorname{trcl}}

\newcommand{\rk}{\operatorname{rk}}

\newcommand{\Ord}{\mathrm{Ord}}

\newcommand{\LK}{\mathsf{LK}}

\newcommand{\CK}{\mathsf{CK}}
\newcommand{\KB}{\mathsf{KB}}

\newcommand{\Id}{\mathsf{Id}}

\newcommand{\RFN}[1][]{\ifthenelse{\equal{#1}{}}{}{#1\mhyphen}\mathsf{RFN}}
\newcommand{\smn}[1][]{\ifthenelse{\equal{#1}{}}{}{#1\mhyphen}\mathsf{smn}}

\newcommand{\CA}{\mathsf{CA}}

\newcommand{\RCA}{\mathsf{RCA}}
\newcommand{\ACA}{\mathsf{ACA}}
\newcommand{\ATR}{\mathsf{ATR}}

\newcommand{\Ind}{\mathsf{Ind}}
\newcommand{\KP}{\mathsf{KP}}

\newcommand{\dotin}{\mathrel{\dot{\in}}}
\newcommand{\RS}{\mathsf{RS}}
\newcommand{\ord}{\mathtt{ord}}
\newcommand{\pre}{\mathsf{pre}}

\newcommand{\Prog}{\mathrm{Prog}}

\newcommand{\no}{\operatorname{no}}
\newcommand{\llbr}{\llbracket}
\newcommand{\rrbr}{\rrbracket}

\newcommand{\parsupp}{\operatorname{parsupp}}

\DeclareFontFamily{T1}{txmia}{}
\DeclareFontShape{T1}{txmia}{m}{n}{<-> txmia}{}

\def\GreekFonttxmia{\usefont{T1}{txmia}{m}{n}}

\newcommand{\rmbeta}{{\text{\GreekFonttxmia \char"0C}}}

\DeclareFontFamily{OMS}{CBG-tt}{}
\DeclareFontShape{OMS}{CBG-tt}{m}{n}{<-> gttn1000}{}
\DeclareFontFamily{OMS}{CBG-ttu}{}
\DeclareFontShape{OMS}{CBG-ttu}{m}{n}{<-> gttu1000}{}

\def\GreekFontCBGtt{\usefont{OMS}{CBG-tt}{m}{n}}

\newcommand{\ttOmega}{\text{\GreekFontCBGtt \char"57}}

\begin{document}
\begin{abstract}
    Proof theorists developed various frameworks to analyze impredicative systems like $\Pi^1_1\text{-}\mathsf{CA}_0$ or $\mathsf{KP}$; One is an operator-controlled derivation system, and the other is Girard's dilator-based $\rmbeta$-logic.
    In this paper, we provide a functorial formulation of operator-controlled analysis of $\mathsf{KP}$, thereby unifying the two approaches into a single framework.    
    As an application, a new proof of Girard's boundedness theorem is also provided, which states that every $\Sigma_1$-over-$L_{\omega_1^{\mathsf{CK}}}$-definable function $\omega_1^\mathsf{CK}\to \omega_1^\mathsf{CK}$ is bounded by a recursive dilator. 
\end{abstract}

\maketitle

\section{Introduction}

Gentzen's quest for consistency gave rise to the ordinal analysis, and Sch\"utte, Feferman, and other proof theorists provided an ordinal analysis for predicative theories. In 1967, Takeuti \cite{Takeuti1967ConsistencySubsystemAnalysis} provided an ordinal analysis of $\Pi^1_1\mhyphen\CA_0$ in terms of ordinal diagrams.
Other proof theorists, including Feferman, Pohlers, and Buchholz, examined alternative methods for analyzing impredicative theories. This led to several approaches towards the ordinal analysis of impredicative theories, including operator-controlled proof theory initiated by Buchholz \cite{Buchholz1992LocalPred} and Pohlers \cite{Pohlers2009Prooftheory}, and Girard's $\rmbeta$-logic \cite{Girard1981Dilators, Girard1982Logical2}.
Developed from Sch\"utte's infinitary proof theory \cite{Schutte1977ProofTheory} and Pohlers' local predicativity \cite{Pohlers1981AnalysisIDnu, Pohlers1982AdmissibilityInProofTheory}, operator-controlled derivation systems control infinitary proofs by attaching operators to each infinitary proof, which capture ordinals relevant to the proof. In Girard's $\rmbeta$-logic, a proof is a \emph{functorial} family of proofs over well-orders, meaning that the proofs vary uniformly along order-preserving maps between well-orders.
At first glance, they appear to be completely different frameworks. However, they show some similarities, for example,
\begin{enumerate}
    \item In operator-controlled analysis of $\KP$, an impredicative cut elimination increases the complexity of an operator of a proof.
    \item In Girard's analysis of $\mathsf{ID}_1$, an impredicative cut elimination increases the complexity of a dilator bound of a proof.
\end{enumerate}
Here, roughly speaking, a dilator is a functorial transformation on well-orders, which can be regarded as a uniform way to produce relativized ordinal notations from ordinal parameters.

The present paper investigates the connection between the operator-controlled analysis and Girard's analysis at the level of $\KP$.
Unlike Girard's original formulation of $\rmbeta$-logic, we use finite order templates in place of fixed ordinal parameters to implement functoriality, where the template records only the relative orders of ordinal parameters. 
We show how to interpret operators in operator-controlled derivations as dilators. Under this interpretation, the operator-controlled derivation of $\KP$ is reformulated as a functorial proof system which varies uniformly along order-preserving maps between well-orders. From this perspective, the impredicative collapsing theorem becomes the functorial cut elimination.

This provides an explicit connection between the two approaches.
The significance of this result is not just that one framework can be translated into the other. Operator-controlled derivations and Girard's framework employ different objects to control the complexity of an impredicative proof; The former employs operators, while the latter employs dilators. The present work demonstrates that these two controlling objects at the level of $\KP$ are different incarnations of a common underlying phenomenon. In particular, it indicates that the complexity of a proof of an impredicative theory is measured not only by individual ordinals, but also by transformations on ordinals (in this case, dilators).
This perspective coheres with the recent development of proof theory: Analyzing impredicative theories via higher-order transformations, such as dilators or ptykes. Examples include Freund's characterization of $\Pi^1_1$-comprehension in terms of the type-two well-ordering principle \cite{Freund2019Pi11}, $\Pi^1_2$-analysis of theories by Aguilera-Pakhomov \cite{AguileraPakhomov??Pi12PTA}, and Pakhomov's recent work (Unpublished) on the interplay between functorial proof theory and iterated reflection.

The present paper only deals with the analysis of $\KP$, where dilators are sufficient to capture the relevant proof-theoretic information. One may ask if the case for $\KP$ (and slightly beyond) is the only case where Girard's functorial framework is effective. Indeed, Rathjen stated that dilators and $\rmbeta$-proofs are ``too weak a tool, if one aims beyond admissible proof theory.'' \cite[p. 475]{Rathjen1995AdvancesPi12}
Rathjen's comment raises the question whether the limitations arise from the functorial viewpoint itself, or from its implementations via dilators and $\rmbeta$-logic.
A recent work of Towsner \cite{Towsner2024proofsmodifyproofs, Towsner2025proofsmodifyproofs12} suggests that infinitary proofs can be understood as transformations acting on other infinitary proofs. 
While this point of view does not directly identify the appropriate proof-theoretic invariants, it suggests that higher-order transformations (like dilators and ptykes) may be correct objects to consider. Indeed, Freund's type-two characterization of $\Pi^1_1\mhyphen\CA_0$ foreshadows how higher-order objects beyond dilators appear in the framework. Hence, we consider $\KP$ not as the final target of the higher-order transformation-based approach, but as the simplest initial target whose operator-controlled approach is well understood. 

The current work in this paper does not provide the proof-theoretic dilator of $\KP$ since we do not have an appropriate version of the \emph{functorial boundedness theorem}. Another motivation behind the functorial reformulation of the proof-theoretic analysis of $\KP$ is its application to the theory of dilators and ptykes.
As an example, we provide an alternative proof of Girard's well-known result stating ``Every $\Sigma_1^{L_{\omega_1^\CK}}$-definable function $\omega_1^\CK\to \omega_1^\CK$ is bounded by a recursive dilator.'' 

\section{Preliminaries}

In this section, we review the facts necessary for this paper. We heavily rely on facts about dilators. For basic facts about dilators, see \cite[Ch. B--D]{JeonPhD}. Throughout this paper, we work over Primitive Recursive Set theory unless specified. The only parts where we need a stronger metatheory are when we need $\psi$ to be a 2-ptyx (See \autoref{Subsection: 1-collapsing ptyx} for $\psi$).

\subsection{Normal dilators}
Every flower we will see in this paper is \emph{normal}, so let us recall its definition:
\begin{definition}
    A predilator $D$ is \emph{normal} if it equips a natural transformation $\mu^D\colon \Id\to D$ such that for every linear order $X$ and $x\in X$, $\sigma\in D(X)$,
    \begin{equation*}
        D(X)\vDash \sigma<\mu^D_X(x) \iff \supp^D_X(\sigma)<x.
    \end{equation*}
\end{definition}
We can see that a normal predilator is a preflower. For this reason, we also call a normal predilator a \emph{normal preflower}.
A \emph{normal dilator} is a normal predilator that is also a dilator. 
If $D$ and $E$ are normal predilators, then so is $D\circ E$, with
\begin{equation*}
    \mu^{D\circ E}_X(x)
    =
    \mu^D_{E(X)}\bigl(\mu^E_X(x)\bigr).
\end{equation*}

\begin{remark} \label{Remark: Convention for the normal flower element}
    Whenever $D$ is normal and $x\in X$, we often suppress the
    normality embedding and write simply $x$ for $\mu^D_X(x)\in D(X)$ when the intended meaning is clear.
    For a composite of normal predilators, the convention is iterated; thus an occurrence of $x$ in $D(E(X))$ may denote $\mu^D_{E(X)}\bigl(\mu^E_X(x)\bigr)$.
    
    For example, we write $t$ for $\varepsilon_t$ in
    $\varepsilon(F)$. We will define the 1-collapsing ptyx $\psi$ later, and if $\psi(\sigma)$ is an argument of a normal predilator, we may use the same symbol $\psi(\sigma)$
    for $\mu\bigl(\psi(\sigma)\bigr)$.
\end{remark}

\subsection{Epsilon flower}
In this section, we review the properties of the epsilon flower. We follow the definition in Freund \cite[Definition 3.4]{Freund2020ComputableBH}.
\begin{definition}
    For a given linear order $X$, we define $\varepsilon(X)$ as a linear order given as follows:
    \begin{enumerate}
        \item $0\in \varepsilon(X)$.
        \item $\varepsilon_u\in \varepsilon(X)$ for every $u\in X$.
        \item If $\varepsilon(X)\vDash \alpha_0\ge\cdots\ge \alpha_{m-1}$ and $m>1$, then $\omega^{\alpha_0} + \cdots + \omega^{\alpha_{m-1}}\in \varepsilon(X)$. 
        \item If $\alpha\in\varepsilon(X)$ is not of the form $\varepsilon_u$, then $\omega^\alpha\in \varepsilon(X)$.
    \end{enumerate}

    Then we define $\varepsilon(X)$-order as follows:
    \begin{enumerate}
        \item $\varepsilon(X)\vDash 0 < \varepsilon_u$ for every $u\in X$.
        \item $\varepsilon(X)\vDash 0 < \omega^{\alpha_0} + \cdots + \omega^{\alpha_{m-1}}$.
        \item $\varepsilon(X)\vDash \varepsilon_u < \varepsilon_v$ iff $u<v$.
        \item If $\varepsilon(X)\vDash \alpha_0 < \varepsilon_u$, then $\varepsilon(X)\vDash \omega^{\alpha_0} + \cdots + \omega^{\alpha_{m-1}} < \varepsilon_u$.
        \item If $\varepsilon(X)\vDash \varepsilon_u < \alpha_0$ or $\varepsilon_u=\alpha_0$, then $\varepsilon(X)\vDash \varepsilon_u < \omega^{\alpha_0} + \cdots + \omega^{\alpha_{m-1}}$.
        \item We have $\varepsilon(X)\vDash \omega^{\alpha_0} + \cdots + \omega^{\alpha_{m-1}}< \omega^{\beta_0}+ \cdots + \omega^{\beta_{n-1}}$ iff
        \begin{enumerate}
            \item $m<n$ and $\alpha_i=\beta_i$, for every $i<m$, or
            \item There is $i< \min(m,n)$ such that either $\varepsilon(X)\vDash \alpha_i<\beta_i$ and $\alpha_j=\beta_j$ for every $j<i$.
        \end{enumerate}
    \end{enumerate}
\end{definition}
One can prove in $\RCA_0$ that $\varepsilon(X)$ is a linear order. \cite[Lemma 3.5]{Freund2020ComputableBH}.
To turn $\varepsilon$ into a predilator, let us define the following:
\begin{definition}
    For an increasing $f\colon X\to Y$, we define $\varepsilon(f)\colon \varepsilon(X)\to\varepsilon(Y)$ recursively as follows:
    \begin{enumerate}
        \item $\varepsilon(f)(0) = 0$.
        \item $\varepsilon(f)(\varepsilon_x) = \varepsilon_{f(x)}$.
        \item For $\alpha_0\ge \cdots\ge \alpha_{m-1}$ in $\varepsilon(X)$ with $m>1$, we define $\varepsilon(f)\bigl(\omega^{\alpha_0} + \cdots + \omega^{\alpha_{m-1}}\bigr) = \omega^{\varepsilon(f)(\alpha_0)} + \cdots + \omega^{\varepsilon(f)(\alpha_{m-1})}$.
    \end{enumerate}
    We also define the support transformation $\supp^\varepsilon_X\colon \varepsilon(X)\to [X]^{<\omega}$ recursively as follows:
    \begin{enumerate}
        \item $\supp^\varepsilon_X(0) = \varnothing$.
        \item $\supp^\varepsilon_X(\varepsilon_x) = \{x\}$.
        \item For $\alpha_0\ge \cdots\ge \alpha_{m-1}$ in $\varepsilon(X)$ with $m\ge 1$, $\supp^\varepsilon_X\bigl(\omega^{\alpha_0} + \cdots + \omega^{\alpha_{m-1}}\bigr) = \bigcup_{i<m}\supp^\varepsilon_X(\alpha_i)$.
    \end{enumerate}
\end{definition}

\begin{lemma}
    $\varepsilon$ is a preflower.
\end{lemma}
\begin{proof}
    We can prove that $\varepsilon(f)\colon \varepsilon(X)\to \varepsilon(Y)$ is order-preserving by induction on $\varepsilon(X)$-terms. 
    For the support condition, fix $f\colon X\to Y$ and prove the following by term induction on $\sigma\in \varepsilon(Y)$: 
    If $\supp^\varepsilon_Y(\sigma) \subseteq \ran f$, then $\sigma \in \ran \varepsilon(f)$.
    
    To see $\varepsilon$ is a preflower, observe that the following holds: For $x\in X$ and $\sigma\in \varepsilon(X)$, $X\vDash \supp^\varepsilon(X)(\sigma)<x$ iff $\varepsilon(X)\vDash \sigma<\varepsilon_x$. This also follows from the term induction on $\sigma$. 
    Now let $X$ be an initial segment of $Y$, $\sigma\in \varepsilon(X)$, $\tau\in \varepsilon(Y)$ be such that $\varepsilon(Y)\vDash \tau\le\sigma$. We claim that $\tau\in \varepsilon(X)$: For each $y\in \supp^\varepsilon_Y(\tau)$, we have $\varepsilon(Y)\vDash \varepsilon_y\le \tau\le\sigma$. This implies $Y\vDash y\le \supp^\varepsilon_Y(\sigma)$, meaning that $y$ is dominated by an element of $X$. Since $X$ is an initial segment of $Y$, we have $y\in X$. This holds for every $y\in \supp^\varepsilon_Y(\tau)$, so $\tau \in \varepsilon(X)$.
\end{proof}

$\varepsilon$ is also a dilator, but it is not provable in $\RCA_0$.
In fact, Marcone-Montalban \cite[Theorem 1.7]{MarconeMontalban2011Veblen} proved that $\varepsilon$ being a dilator is equivalent to $\ACA_0^+$ over $\RCA_0$.

Epsilon numbers admit an algebraic structure in the sense that they are closed under addition and multiplication. The same holds for $\varepsilon(X)$, and furthermore, the addition and multiplication are `functorial' in the sense that the addition and multiplication are preserved under $\varepsilon(f)$. Let us formulate the addition and multiplication over $\varepsilon(X)$ as follows:
\begin{definition}
    We define the addition over $\varepsilon(X)$ as follows:
    \begin{enumerate}
        \item $0 + \alpha = \alpha+0 := \alpha$.
        \item For $u,v\in X$,
        \begin{equation*}
            \varepsilon_u + \varepsilon_v =
            \begin{cases}
                \omega^{\varepsilon_u} + \omega^{\varepsilon_v} & \text{If }u\ge v, \\
                \varepsilon_v & \text{If }u<v.
            \end{cases}
        \end{equation*}
        \item For $u\in X$ and $\alpha_0,\cdots,\alpha_{m-1}\in \varepsilon(X)$,
        \begin{equation*}
            \varepsilon_u + (\omega^{\alpha_0}+\cdots+\omega^{\alpha_{m-1}}) = 
            \begin{cases}
                \omega^{\varepsilon_u} + \omega^{\alpha_0}+\cdots+\omega^{\alpha_{m-1}} & \text{If }\varepsilon_u \ge\alpha_0, \\
                \omega^{\alpha_0}+\cdots+\omega^{\alpha_{m-1}} & \text{If }\varepsilon_u < \alpha_0,
            \end{cases}
        \end{equation*}
        and
        \begin{equation*}
            (\omega^{\alpha_0}+\cdots+\omega^{\alpha_{m-1}}) + \varepsilon_u = 
            \begin{cases}
                \omega^{\alpha_0}+\cdots+\omega^{\alpha_{m-1}} + \omega^{\varepsilon_u}& \text{If }\varepsilon_u \le\alpha_{m-1}, \\
                (\omega^{\alpha_0}+\cdots+\omega^{\alpha_{m-2}}) + \varepsilon_u & \text{If }\varepsilon_u>  \alpha_{m-1}.
            \end{cases}
        \end{equation*}

        \item For $\alpha_0,\cdots,\alpha_{m-1},\beta_0,\cdots,\beta_{n-1}\in\varepsilon(X)$,
        we define $(\omega^{\alpha_0}+\cdots+\omega^{\alpha_{m-1}}) + (\omega^{\beta_0}+\cdots+\omega^{\beta_{n-1}})$ by
        \begin{equation*}
            \begin{cases}
                (\omega^{\alpha_0}+\cdots+\omega^{\alpha_{m-2}}) +(\omega^{\beta_0}+\cdots+\omega^{\beta_{n-1}}) & \text{If }\alpha_{m-1}<\beta_0, \\
                \omega^{\alpha_0}+\cdots +\omega^{\alpha_{m-1}}+\omega^{\beta_0}+\cdots+\omega^{\beta_{n-1}} & \text{If }\alpha_{m-1}\ge\beta_0.
            \end{cases}
        \end{equation*}
    \end{enumerate}
    We also define the multiplication over $\varepsilon(X)$ as follows:
    \begin{enumerate}
        \item $0\cdot\alpha = \alpha\cdot 0 := 0$.
        \item For $u,v\in X$, $\varepsilon_u\cdot \varepsilon_v := \omega^{\varepsilon_u + \varepsilon_v}$. (Here we follow the previous definition of the addition.) 
        \item For $u\in X$ and $\alpha_0,\cdots,\alpha_{m-1}\in \varepsilon(X)$,
        \begin{itemize}
            \item $\varepsilon_u \cdot (\omega^{\alpha_0}+\cdots+\omega^{\alpha_{m-1}}) = \omega^{\varepsilon_u + \alpha_0} + \cdots + \omega^{\varepsilon_u + \alpha_{m-1}}$, and
            \item $(\omega^{\alpha_0}+\cdots+\omega^{\alpha_{m-1}})  \cdot\varepsilon_u = \omega^{\alpha_0+\varepsilon_u}$.
        \end{itemize}
        
        \item For $\alpha_0,\cdots,\alpha_{m-1},\beta\in\varepsilon(X)$, we first define
        \begin{equation*}
            (\omega^{\alpha_0}+\cdots+\omega^{\alpha_{m-1}}) \cdot \omega^\beta = 
            \begin{cases}
                \omega^{\alpha_0}+\cdots+\omega^{\alpha_{m-1}} & \beta = 0,\\
                \omega^{\alpha_0+\beta} & \beta\neq 0.
            \end{cases}
        \end{equation*}
        Then for $\sigma = \omega^{\alpha_0}+\cdots+\omega^{\alpha_{m-1}}\in \varepsilon(X)$ and $\beta_0,\cdots,\beta_{n-1}\in\varepsilon(X)$, we define
        \begin{equation*}
            \sigma \cdot (\omega^{\beta_0}+\cdots+\omega^{\beta_{n-1}}) :=  \sigma \cdot \omega^{\beta_0} + \cdots + \sigma \cdot \omega^{\beta_{n-1}}. 
        \end{equation*}
    \end{enumerate}
    Here, we identify the formal expression $\omega^{\varepsilon_u}$ with $\varepsilon_u$.
\end{definition}

\begin{proposition} \label{Proposition: Functoriality of the internal addition and multiplication}
    The addition over the epsilon dilator is functorial; That is, if $f\colon X\to Y$ and $\varepsilon(X)\vDash \sigma+\tau=\upsilon$, then $\varepsilon(Y)\vDash \varepsilon(f)(\sigma) + \varepsilon(f)(\tau) = \varepsilon(f)(\upsilon)$. Similarly, the multiplication over the epsilon dilator is also functorial.
\end{proposition}
\begin{proof}
    The proof follows from the double induction on $\sigma$ and $\tau$, and the details are tedious. However, the main point is that the definition of the addition and multiplication uses case divisions under the term type ($0$, $\varepsilon_x$, or the sum of $\omega^\alpha$s) and the relative order and equality between terms, which are all preserved by $\varepsilon(f)$.
\end{proof}

We can also define the natural sum as follows:
\begin{definition}
    We define the natural sum over $\varepsilon(X)$ as follows:
    \begin{enumerate}
        \item $0\#\alpha = \alpha\#0:=\alpha$.
        
        \item For $u,v\in X$, 
        \begin{equation*}
            \varepsilon_u \# \varepsilon_v =
            \begin{cases}
                \omega^{\varepsilon_u} + \omega^{\varepsilon_v} & \text{If }u\ge v, \\
                 \omega^{\varepsilon_v} + \omega^{\varepsilon_u} & \text{If }u<v.
            \end{cases}
        \end{equation*}
        
        \item Let $u\in X$ and $\alpha_0\ge \cdots\ge\alpha_{m-1}$ are in $\varepsilon(X)$. Let $k< m$ be the least number such that $\varepsilon_u>\alpha_k$ if there is such a number; Otherwise, let $k=m$. We define
        \begin{equation*}
            \varepsilon_u \# (\omega^{\alpha_0}+\cdots+\omega^{\alpha_{m-1}}) 
            = (\omega^{\alpha_0}+\cdots+\omega^{\alpha_{m-1}})  \# \varepsilon_u = \omega^{\alpha_0}+\cdots+ \omega^{\varepsilon_u} + \omega^{\alpha_k}+ \cdots +  \omega^{\alpha_{m-1}}.
        \end{equation*}

        \item For $\alpha_0\ge\cdots\ge\alpha_{m-1}$ and $\beta_0\ge\cdots\ge\beta_{n-1}$ in $\varepsilon(X)$, let $\gamma_0\ge\cdots\ge\gamma_{m+n-1}$ be the unique re-enumeration of the sequence $\alpha_0,\cdots,\alpha_{m-1},\beta_0,\cdots,\beta_{n-1}$.
        We define
        \begin{equation*}
            (\omega^{\alpha_0}+\cdots+\omega^{\alpha_{m-1}}) \# (\omega^{\beta_0}+\cdots+\omega^{\beta_{n-1}}) = \omega^{\gamma_0} + \cdots + \omega^{\gamma_{m+n-1}}.
        \end{equation*}
    \end{enumerate}
\end{definition}

We can also show that the natural sum is functorial. Its proof is similar to that of \autoref{Proposition: Functoriality of the internal addition and multiplication}.

\begin{remark}
    We can also define an internal addition and natural sum operator over the dilator $X\mapsto \omega^X$. Its details are identical to those of the epsilon flower, so we omit the details.
\end{remark}

\subsection{Veblen flower}
In this section, we introduce the Veblen flower. Let $\alpha$ be a fixed linear order.
\begin{definition}
    For a given linear order $X$, we define $\varphi_\alpha(X)$ as a linear order given as follows.
    We simultaneously define principal terms of $\varphi_\alpha(X)$.
    \begin{enumerate}
        \item $0\in \varphi_\alpha(X)$.
        \item $E_x\in \varphi_\alpha(X)$ is a principal term for every $x\in X$.
        \item If $i\in\alpha$ and $\xi\in\varphi_\alpha(X)$, then $\varphi_i(\xi)$ is a principal term, provided that $\xi$ is not a single principal term of one of the following forms:
        \begin{enumerate}
            \item $E_x$ for some $x\in X$.
            \item $\varphi_j(\eta)$ for some $j\in\alpha$ with $i<j$.
        \end{enumerate}
        \item If $p_0,\ldots,p_{m-1}$ are principal terms of $\varphi_\alpha(X)$, $\varphi_\alpha(X)\vDash p_0\ge\cdots\ge p_{m-1}$, and $m>1$, then
        \begin{equation*}
        p_0+\cdots+p_{m-1}\in\varphi_\alpha(X).
        \end{equation*}
    \end{enumerate}

    Then we define $\varphi_\alpha(X)$-order as follows:
    \begin{enumerate}
        \item $\varphi_\alpha(X)\vDash 0<E_x$ for every $x\in X$.
        \item $\varphi_\alpha(X)\vDash 0<\varphi_i(\xi)$ for every $i\in\alpha$ and every $\xi\in\varphi_\alpha(X)$.
        \item $\varphi_\alpha(X)\vDash 0<p_0+\cdots+p_{m-1}$ for every sum term $p_0+\cdots+p_{m-1}$.
        \item $\varphi_\alpha(X)\vDash E_x<E_y$ iff $x<y$.
        \item For $i\in\alpha$, $\xi\in\varphi_\alpha(X)$, and $x\in X$, we have
        \begin{equation*}
        \varphi_\alpha(X)\vDash \varphi_i(\xi)<E_x \iff \varphi_\alpha(X)\vDash \xi<E_x.
        \end{equation*}
        \item For $i\in\alpha$, $\xi\in\varphi_\alpha(X)$, and $x\in X$, we have
        \begin{equation*}
        \varphi_\alpha(X)\vDash E_x<\varphi_i(\xi) \iff 
        \varphi_\alpha(X)\vDash E_x<\xi.
        \end{equation*}
        \item For $i,j\in\alpha$ and $\xi,\eta\in\varphi_\alpha(X)$, we have
        \begin{equation*}
        \varphi_\alpha(X)\vDash \varphi_i(\xi)<\varphi_j(\eta)
        \end{equation*}
        iff one of the following holds:
        \begin{enumerate}
            \item $i<j$ and $\varphi_\alpha(X)\vDash \xi<\varphi_j(\eta)$.
            \item $i=j$ and $\varphi_\alpha(X)\vDash \xi<\eta$.
            \item $j<i$ and $\varphi_\alpha(X)\vDash \varphi_i(\xi)<\eta$.
        \end{enumerate}
        \item We have
        \begin{equation*}
        \varphi_\alpha(X)\vDash p_0+\cdots+p_{m-1}<q_0+\cdots+q_{n-1}
        \end{equation*}
        iff
        \begin{enumerate}
            \item $m<n$ and $p_i=q_i$ for every $i<m$, or
            \item there is $i<\min(m,n)$ such that $\varphi_\alpha(X)\vDash p_i<q_i$ and $p_j=q_j$ for every $j<i$.
        \end{enumerate}
    \end{enumerate}
\end{definition}

To turn $\varphi_\alpha$ into a predilator, let us define the following:
\begin{definition}
    For an increasing $f\colon X\to Y$, we define $\varphi_\alpha(f)\colon\varphi_\alpha(X)\to\varphi_\alpha(Y)$ recursively as follows:
    \begin{enumerate}
        \item $\varphi_\alpha(f)(0)=0$.
        \item $\varphi_\alpha(f)(E_x)=E_{f(x)}$.
        \item $\varphi_\alpha(f)(\varphi_i(\xi))=\varphi_i(\varphi_\alpha(f)(\xi))$.
        \item For principal terms $p_0,\ldots,p_{m-1}$ in $\varphi_\alpha(X)$ with $m>1$, we define
        \begin{equation*}
        \varphi_\alpha(f)(p_0+\cdots+p_{m-1})
        =
        \varphi_\alpha(f)(p_0)+\cdots+\varphi_\alpha(f)(p_{m-1}).
        \end{equation*}
    \end{enumerate}
    We also define the support transformation $\supp^{\varphi_\alpha}_X\colon \varphi_\alpha(X)\to [X]^{<\omega}$ recursively as follows:
    \begin{enumerate}
        \item $\supp^{\varphi_\alpha}_X(0)=\varnothing$.
        \item $\supp^{\varphi_\alpha}_X(E_x)=\{x\}$.
        \item $\supp^{\varphi_\alpha}_X(\varphi_i(\xi))=\supp^{\varphi_\alpha}_X(\xi)$.
        \item For principal terms $p_0,\ldots,p_{m-1}$ in $\varphi_\alpha(X)$ with $m>1$,
        \begin{equation*}
        \supp^{\varphi_\alpha}_X(p_0+\cdots+p_{m-1})
        =
        \bigcup_{i<m}\supp^{\varphi_\alpha}_X(p_i).
        \end{equation*}
    \end{enumerate}
\end{definition}

One can prove that $\varphi_\alpha$ is also a preflower; The following ensures $\varphi_\alpha$ being a preflower:
\begin{equation*}
    \supp^{\varphi_\alpha}_X(\sigma)<x\iff \varphi_\alpha(X)\vDash\sigma< E_x.
\end{equation*}
Moreover, if $\alpha$ is a well-order, then $\varphi_\alpha$ is a flower. $\varphi_\alpha$ being a flower for a well-order $\alpha$ is a theorem of $\ATR_0$.
One can also see that $\varphi_\alpha$ for a non-empty well-order $\alpha$ admits a functorial internal addition, multiplication, base $\omega$ exponentiation, and natural sum. For multiplication, we use that the least Veblen function $\varphi_0$ represents exponentiation to base $\omega$, while the terms $E_x$ and $\varphi_i(\xi)$ for $0<i\in\alpha$ are fixed points of this exponentiation.

\subsection{1-collapsing ptyx} \label{Subsection: 1-collapsing ptyx}
There is an abstract characterization of a 1-collapsing ptyx $\psi(D)$ \cite{Uftring2025InverseGoodstein}. However, we only introduce its canonical representation:
\begin{definition}
    Let $D$ be a dilator. We define a linear order $\psi^+(D)$ inductively as follows: For each $\sigma\in D(\psi^+(D))$, we have $\psi(\sigma)\in \psi^+(D)$. 
    For two $\psi(\sigma),\psi(\tau)\in \psi^+(D)$, we define 
    \begin{equation*}
        \psi^+(D)\vDash \psi(\sigma)<\psi(\tau) \iff D(\psi^+(D))\vDash \sigma<\tau.
    \end{equation*}
\end{definition}

\begin{definition}
    Let us recursively define $\sfK^D\colon \psi^+(D)\to \bigl[D\bigl(\psi^+(D)\bigr)\bigr]^{<\omega}$ as follows:
    \begin{equation*} \textstyle
        \sfK^D\bigl(\psi(\sigma)\bigr) = \{\sigma\}\cup \bigcup \bigl\{ \sfK^D(s)\bigm| s\in \supp^D_{\psi^+ (D)}(\sigma)\bigr\}.
    \end{equation*}
    We also define $\hat\sfK^D\colon D\bigl(\psi^+(D)\bigr)\to \bigl[D\bigl(\psi^+(D)\bigr)\bigr]^{<\omega}$ by
    \begin{equation*}
        \hat\sfK^D(\sigma) = \bigcup \bigl\{ \sfK^D(s)\bigm| s\in \supp^D_{\psi^+(D)}(\sigma)\bigr\}.
    \end{equation*}
\end{definition}

\begin{definition}
    We define $\psi(D)\subseteq \psi^+(D)$ recursively as follows:
    \begin{equation*}
        \psi(\sigma)\in \psi(D)\iff \supp^D_{\psi^+(D)}(\sigma)\subseteq \psi(D) \land D\bigl(\psi^+(D)\bigr)\vDash \hat\sfK^D(\sigma) < \sigma.
    \end{equation*}
    Here $X<\sigma$ denotes ``For every $x\in X$, $x<\sigma$.''
\end{definition}
Hence, the definition of $\psi(D)$ implies that for every $\psi(\sigma)\in \psi(D)$, we have $D(\psi(D))\vDash \sfK^D(\psi(\sigma))\le \sigma$.
The following proposition is easy to prove, but will have focal roles when $D$ is a flower:
\begin{proposition}
    For $\sigma \in D(\psi(D))$, $\max \hat\sfK^D(\sigma) = \max \{\tau\in D(\psi(D)) \mid \psi(\tau)\in \supp^D_{\psi(D)}(\sigma)\}$.
\end{proposition}
\begin{proof}
    Observe that 
    \begin{equation*}
        \hat\sfK^D(\sigma) = \bigcup \bigl\{\{\tau\}\cup \hat\sfK^D(\tau) \bigm| \psi(\tau)\in \supp^D_{\psi(D)}(\sigma)\bigr\}
    \end{equation*}
    and the definition of $\psi(D)$ implies $\hat\sfK^D(\tau) < \tau$ as long as $\psi(\tau)\in \psi(D)$. This shows
    \begin{equation*}
        \max \hat\sfK^D(\sigma) = \max \bigcup \bigl\{\{\tau\} \bigm| \psi(\tau)\in \supp^D_{\psi(D)}(\sigma)\bigr\} = \max \{\tau \mid\psi(\tau)\in \supp^D_{\psi(D)}(\sigma) \} \qedhere 
    \end{equation*}
\end{proof}

\subsection{\texorpdfstring{$\rmbeta$}{Beta}-logic} 
In this subsection, we introduce $\rmbeta$-logic, which will be generalized into $F$-logic for a dilator $F$ in a later section.
An infinitary proof system for $\rmbeta$-logic can be more complicated. Still, we introduce only the simplest form of $\rmbeta$-logic, obtained from first-order logic by adding an ordinal sort and rules for ordinals.

Throughout this paper, we use one-sided Tait-style sequent calculi.
Sequents are identified up to exchange and contraction, and are closed under weakening. Thus, structural rules are not part of the official deduction system. When a proof diagram displays weakening (W), contraction, or exchange, this is only a bookkeeping indication of the corresponding admissible structural operation.
We also understand sequents as sets of formulas. Hence, as an example, $\Gamma, A$ and $\Gamma, A, A$ are identical sequents.

\begin{definition}
    A \emph{$\rmbeta$-language} is a many-sorted first-order language $\calL$ with a distinguished type $\ord$ together with a distinguished binary predicate $\le_\ord$, and the only terms in $\calL$ of type $\ord$ are variables. In particular, we do not allow constant symbols of type $\ord$ or function symbols of value $\ord$. A \emph{$\rmbeta$-theory} is a theory in some $\rmbeta$-language.
    
    A structure $M$ in a $\rmbeta$-language $\calL$ is a \emph{pre-$\rmbeta$-model} if the \emph{ordinal part} $(\ord^M,\le_\ord^M)$ of $M$ is a linear order. If the ordinal part of $M$ is a well-order, we call $M$ a \emph{$\rmbeta$-model}.
    We also define the following semantic implication notions: For a $\rmbeta$-language $\calL$ and a $\rmbeta$-theory $T$ in $\calL$, and a formula $\phi$ in $\calL$,
    \begin{enumerate}
        \item $T\vDash_{\pre\rmbeta} \phi$ if every pre-$\rmbeta$-model of $T$ satisfies the universal closure of $\phi$. We call $\phi$ is \emph{pre-$\rmbeta$-valid in $T$}.
        \item $T\vDash_{\rmbeta} \phi$ if every $\rmbeta$-model of $T$ satisfies the universal closure of $\phi$. We call $\phi$ is \emph{$\rmbeta$-valid in $T$}.
    \end{enumerate}
    We omit $T$ in the previous notations when $T=\varnothing$. 
\end{definition}

The syntactical part of $\rmbeta$-logic comprises two parts: One is each instance $\alpha$-logic for an ordinal $\alpha$, and the other is the functoriality condition unifying each instance by one. We first introduce the $\alpha$-logic:
\begin{definition} \label{Definition: LKalpha}
    Let $\alpha$ be a fixed ordinal and $\calL$ be a $\rmbeta$-language. Now let $\calL_\alpha$ be the language obtained from $\calL$ by adding constant symbols $\overline{\xi}$ for $\xi\in\alpha$. The rules of Tait-style $\LK_\alpha$ are obtained from the rules of first-order logic with the following additional rules:
    \begin{center}
        \vspace{1em}
        \begin{mathprooftree}
            \AxiomC{}
            \RightLabel{$(<_\ord)$ ($\eta<\xi$)}
            \UnaryInfC{$\overline{\eta}<_\ord\overline{\xi}$}
        \end{mathprooftree} %
        \hspace{3em}
        \begin{mathprooftree}
            \AxiomC{}
            \RightLabel{$(\le_\ord)$ ($\eta\le\xi$)}
            \UnaryInfC{$\overline{\eta}\le_\ord\overline{\xi}$}
        \end{mathprooftree}
        
        \vspace{1em}
        \begin{mathprooftree}
            \AxiomC{$\cdots \quad \Gamma, A(\overline{\xi}) \quad\cdots \quad (\text{for all }\xi<\alpha)$}
            \RightLabel{$\forall^\ord$}
            \UnaryInfC{$\Gamma,\forall^\ord x A(x)$}
        \end{mathprooftree} %
        \hspace{3em}
        \begin{mathprooftree}
            \AxiomC{$\Gamma,A(\overline{\xi})$}
            \RightLabel{$\exists^\ord$}
            \UnaryInfC{$\Gamma,\exists^\ord x A(x)$}
        \end{mathprooftree} 
    \end{center}
    We also assume that we can apply structural rules (Weakening, Contraction, Exchange) without explicit application of these rules; These rules can be simultaneously applied to other rules. 
\end{definition}

\begin{example}
    The following proof gives an $\LK_\omega$-proof of $\forall^\ord x \exists^\ord y (x<_\ord y)$:
    \vspace{1em}
    \begin{center}
        \begin{mathprooftree}
            \AxiomC{}
            \RightLabel{$(<_\ord)$}
            \UnaryInfC{$\overline{0} <_\ord \overline{1}$}
            \RightLabel{$(\exists^\ord)$}
            \UnaryInfC{$\exists^\ord y (\overline{0} <_\ord y)$}

            \AxiomC{}
            \RightLabel{$(<_\ord)$}
            \UnaryInfC{$\overline{1} <_\ord \overline{2}$}
            \RightLabel{$(\exists^\ord)$}
            \UnaryInfC{$\exists^\ord y (\overline{1} <_\ord y)$}

            \AxiomC{}
            \RightLabel{$(<_\ord)$}
            \UnaryInfC{$\overline{2} <_\ord \overline{3}$}
            \RightLabel{$(\exists^\ord)$}
            \UnaryInfC{$\exists^\ord y (\overline{2} <_\ord y)$}

            \AxiomC{$\cdots$}
            
            \RightLabel{$(\forall^\ord)$}
            \QuaternaryInfC{$\forall^\ord x \exists^\ord y (x<_\ord y)$}
        \end{mathprooftree}
    \end{center}
\end{example}

\begin{example}
    We can prove transfinite induction over $\LK_\alpha$: For a formula $A(x)$, let us take
    \begin{equation*}
        \Prog_A \equiv \forall^\ord x \bigl[ \exists^\ord y\bigl( y<_\ord x\land \lnot A(y) \bigr)\bigr]\lor A(x).
    \end{equation*}
    Note that 
    \begin{equation*}
        \lnot\Prog_A \iff \exists^\ord x \bigl[ \forall^\ord y\bigl( x\le_\ord y\lor A(y) \bigr)\bigr]\land \lnot A(x).
    \end{equation*}
    Now suppose that we are given an $\LK_\alpha$-proof $\pi_\eta$ for the sequent $\lnot\Prog_A, A(\overline{\eta})$ for every $\eta<\xi$. Now consider the $\LK_\alpha$-proof $\pi_\xi$ given as follows:
    \begin{center} \vspace{1em}
    \begin{mathprooftree}
        \AxiomC{$\pi_\eta$}
        \noLine
        \UnaryInfC{$\vdots$}
        \RightLabel{($\eta<\xi$)}
        \UnaryInfC{$\lnot\Prog_A,A(\overline{\eta})$}
        \RightLabel{($\lor$)}
        \UnaryInfC{$\lnot\Prog_A,\overline{\xi}\le_\ord \overline{\eta}\lor A(\overline{\eta})$}

        \AxiomC{}
        \RightLabel{($\eta\ge\xi$)}
        \UnaryInfC{$\overline{\eta}\ge_\ord\overline{\xi}$}
        \RightLabel{($\lor$)}
        \UnaryInfC{$\lnot\Prog_A,\overline{\xi}\le_\ord \overline{\eta}\lor A(\overline{\eta})$}

        \RightLabel{($\forall^\ord$)}
        \BinaryInfC{$\lnot\Prog_A, \forall^\ord y \bigl(\overline{\xi}\le_\ord y \lor A(y)\bigr)$}

        \RightLabel{(W)}
        \UnaryInfC{$\lnot\Prog_A, \forall^\ord y \bigl(\overline{\xi}\le_\ord y \lor A(y)\bigr), A(\overline{\xi})$}

        \AxiomC{}
        \RightLabel{(Ax)}
        \UnaryInfC{$\lnot A(\overline{\xi}), A(\overline{\xi})$}
        \RightLabel{(W)}
        \UnaryInfC{$\lnot\Prog_A, \lnot A(\overline{\xi}), A(\overline{\xi})$}

        \RightLabel{($\land$)}
        \BinaryInfC{$\lnot\Prog_A, \forall^\ord y \bigl(\overline{\xi}\le_\ord y \lor A(y)\bigr) \land \lnot A(\overline{\xi}), A(\overline{\xi})$}

        \RightLabel{($\exists^\ord$)}
        \UnaryInfC{$\lnot\Prog_A, A(\overline{\xi})$}
    \end{mathprooftree}
    \end{center}

    Hence, we have the following $\alpha$-proof:
    \begin{center}
    \begin{mathprooftree}
        \AxiomC{$\pi_\xi$}
        \noLine
        \UnaryInfC{$\vdots$}
        \noLine
        \UnaryInfC{$\lnot\Prog_A$, $A(\overline{\xi})$ \hspace{1em} $(\xi\in\alpha)$}
        \RightLabel{($\forall^\ord$)}
        \UnaryInfC{$\lnot\Prog_A$, $\forall^\ord x A(x)$}
    \end{mathprooftree}
    \end{center} 
\end{example}

Now, let us define the functoriality part.
\begin{definition}
    Let $\alpha,\beta$ be ordinals and $f\colon \alpha\to\beta$ be an increasing function.
    Suppose that $\Gamma$ is a sequent in $\calL_\beta$ such that every ordinal constant symbol occurring in $\Gamma$ takes the form $\overline{f(\xi)}$ for some $\xi\in\alpha$. 
    We define the new $\calL_\alpha$-sequent $\sfm_f(\Gamma)$ obtained from $\Gamma$ by replacing every constant symbol of the form $\overline{f(\xi)}$ with $\overline{\xi}$.
    
    For an $\LK_\beta$-proof $\pi$ of $\Gamma$, let us define its \emph{mutilation} ${}^f\pi$ obtained from $\pi$ by replacing every sequent $\Delta$ occurring in $\pi$ with $\sfm_f(\Delta)$. If the replacement is impossible, remove the corresponding sequent from the proof accordingly. 
\end{definition}
${}^f\pi$ is not necessarily an $\LK_\alpha$-proof in general, although we will only consider such proofs. Every $\LK_\alpha$-proof we will consider satisfies the following property:
\begin{definition}
    Let $\pi$ be an $\LK_\beta$-proof of a sequent $\Gamma$ with no ordinal constants. We say $\pi$ is \emph{homogeneous} if for every increasing $f,g\colon \alpha\to\beta$, we have ${}^f\pi = {}^g\pi$.
\end{definition}

If $\pi$ is a homogeneous $\LK_\beta$-proof of $\Gamma$ and $\alpha<\beta$, we can uniquely determine the corresponding $\LK_\alpha$-proof. Now let us define the notion of $\rmbeta$-proof:
\begin{definition}
    A \emph{$\rmbeta$-preproof} is a sequence of preproofs $\lag \pi_\alpha\mid \alpha\in\Ord\rag$ such that for every increasing $f\colon\alpha\to\beta$, ${}^f\pi_\beta=\pi_\alpha$. A \emph{$\rmbeta$-proof} is a $\rmbeta$-preproof whose each component is well-founded.
\end{definition}

The following fact is known, although we will provide its proof later in a different fashion:
\begin{proposition}[$\rmbeta$-soundness and completeness, Girard \cite{Girard1982Logical2}]
    Let $\calL$ be a $\rmbeta$-language. An $\calL$-sentence $A$ is $\rmbeta$-provable iff $A$ is $\rmbeta$-valid.
\end{proposition}

\section{Dendrogram-styled \texorpdfstring{$\rmbeta$}{Beta}-logic}
However, the previous dendroid-style $\rmbeta$-proof has a technical flaw: It does not allow the mention of ordinal parameters in the conclusion. In the latter application, however, we need to refer to ordinal parameters in the conclusion to facilitate an (co)inductive argument on $\rmbeta$-preproofs. Once a parameter is represented by an actual ordinal constant, the proof may depend on accidental features of that ordinal, rather than only on its order relation to the other parameters. For example, if a parameter is instantiated as $0$, then the corresponding proof component cannot contain branches for
ordinals below that parameter. This is harmless while the parameter is really $0$, but it becomes problematic after reindexing: the same parameter may be sent to a larger ordinal $\xi$, and then the missing
branches below the parameter should correspond to ordinals below $\xi$. Such information cannot be recovered from the proof component in which the parameter was fixed as $0$.

To avoid this information loss, we replace actual ordinal parameters with \emph{finite linear order templates}. A template records only the relative order among parameters and keeps open all possible ways of inserting new ordinal variables around them.
The resulting $\rmbeta$-(pre)proof is no longer a functorial family of $\alpha$-(pre)proofs, but a single proof tree resembling a cell-decomposition of a (pre)dilator. The resulting (pre)proof also resembles proof trees represented in Towsner's proof system in \cite{Towsner2025proofsmodifyproofs12}.

\begin{definition}
    A \emph{template order} is a finite linear order whose field is a natural number. We denote template orders by $\ttX,\ttY,\ttZ,$ etc. To distinguish members of $\ttX$ from the usual natural numbers, we denote members of $\ttX$ by $\tilde{0}, \tilde{1}, \tilde{2}, $ etc.
\end{definition}
We will think of a template order as a finite set of variables (or ordinal indiscernibles) in a $\rmbeta$-proof system.
Now let us define the deduction system for the dendrogram-styled $\rmbeta$-logic. Every sequent of the deduction system takes the form
\begin{equation*}
    [\ttX]\vdash\Gamma
\end{equation*}
for a template order $\ttX$ and a sequent $\Gamma$ in the language $\calL$ augmented with ordinal constant symbols $\tilde{n}$ for $n\in \ttX$.

\begin{definition}
    Let us define the axioms and deduction rules of $\LK_{\rmbeta}$ as follows: We assume the typical rules for first-order logic, and add the following rules for ordinals:
    \begin{center}
        \vspace{1em}
        \begin{mathprooftree}
            \def\fCenter{\vdash}
            \AxiomC{}
            \RightLabel{$(<_\ord)$ ($\ttX \vDash x<y$)}
            \UnaryInfC{$[\ttX]\fCenter \tilde{x}<_\ord\tilde{y}$}
        \end{mathprooftree} %
        \hspace{3em}
        \begin{mathprooftree}
            \def\fCenter{\vdash}
            \AxiomC{}
            \RightLabel{$(\le_\ord)$ ($\ttX \vDash x\le y$)}
            \UnaryInfC{$[\ttX]\fCenter \tilde{x}\le_\ord\tilde{y}$}
        \end{mathprooftree}
        
        \vspace{1em}
        \begin{mathprooftree}
            \def\fCenter{\vdash}
            \AxiomC{$\cdots\ [\ttX']\fCenter \Gamma, A(\tilde{n}) \quad (n\notin \ttX)\ \cdots $}
            \AxiomC{$\cdots\ [\ttX]\fCenter \Gamma, A(\tilde{n}) \quad (n\in \ttX)\ \cdots$}
            \RightLabel{$(\forall^\ord)$}
            \BinaryInfC{$[\ttX]\fCenter \Gamma,\forall^\ord x A(x)$}
        \end{mathprooftree}
        
        \vspace{1em}
        \begin{mathprooftree}
            \def\fCenter{\vdash}
            \AxiomC{$[\ttX]\fCenter \Gamma,A(\tilde{n})$}
            \RightLabel{$(\exists^\ord)$}
            \UnaryInfC{$[\ttX]\fCenter \Gamma,\exists^\ord x A(x)$}
        \end{mathprooftree}
    \end{center}
    In $\forall^\ord$, the first form of the hypotheses is enumerated for every possible $\ttX'$ for an extension of $\ttX$ of the form $\ttX\cup \{n\}$, where $n$ is the least number not in $\field \ttX$. (The order of $n$ in $\ttX'$ can be arbitrary).

    We form a $\rmbeta$-preproof from a given set of rules. We say a $\rmbeta$-preproof $\pi$ is a \emph{$\rmbeta$-proof} if for every infinite branch $B$ of $\pi$, the union of all template orders occurring in $B$ is ill-founded.
\end{definition}
The definition of a $\rmbeta$-proof is inspired by the cell-decomposition characterization of a dilator: A predilator $D$ is a dilator iff its cell-decomposition $\operatorname{Cell}(D)$ is locally well-founded, and for every infinite branch of $\operatorname{Cell}(D)$, the union of its priority orders (also known as a \emph{thread}) is ill-founded.
See \cite{AguileraFreundPakhomov??Shallow} or \cite[Ch. D]{JeonPhD} for the details.

\begin{example}
    Let us provide a $\rmbeta$-proof for the transfinite induction. We define $\ttX_n$ by
    \begin{equation*}
        \ttX_n = \{\tilde{n}<\widetilde{n-1} < \cdots < \tilde{1}< \tilde{0}\}.
    \end{equation*}
    Then, for each $n$, consider the proof tree $\tilde{\pi}_n$ of $[\ttX_n]\vdash \lnot\Prog_A,A(\tilde{n})$ given as follows, where $m=n+1$:
    First, let us consider the following proof tree $\tilde{\pi}'_n$:
    \begin{center} \vspace{1em}
    \begin{mathprooftree}
        \def\fCenter{\vdash}
        \AxiomC{$\tilde{\pi}_{n+1}$}
        \noLine
        \UnaryInfC{$\vdots$}
        \UnaryInfC{$[\ttX_{n+1}]\fCenter \lnot\Prog_A,A(\tilde{m})$}
        \RightLabel{($\lor$)}
        \UnaryInfC{$[\ttX_{n+1}]\fCenter \lnot\Prog_A,\tilde{n}\le_\ord \tilde{m}\lor A(\tilde{m})$}

        \AxiomC{}
        \RightLabel{$(\le_\ord)$}
        \UnaryInfC{$[\ttX']\fCenter \tilde{m}\ge_\ord\tilde{n}$}
        \RightLabel{($\lor$)}
        \UnaryInfC{$[\ttX']\fCenter \lnot\Prog_A,\tilde{n}\le_\ord \tilde{m}\lor A(\tilde{m})$}

        \RightLabel{($\forall^\ord$)}
        \BinaryInfC{$[\ttX_n]\fCenter\lnot\Prog_A, \forall^\ord y \bigl(\tilde{n}\le_\ord y \lor A(y)\bigr)$}
    \end{mathprooftree}
    \end{center}
    Here $(\le_\ord)$ covers all other possible cases $\ttX'$ ($\ttX'=\ttX$ is possible) on which $\ttX'\vDash n\le m$ holds. Then consider the following:
    
    \begin{center} \vspace{1em}
    \begin{mathprooftree}
        \def\fCenter{\vdash}
        \AxiomC{$\tilde{\pi}_n'$}
        \noLine
        \UnaryInfC{$\vdots$}
        \UnaryInfC{$[\ttX_n]\fCenter\lnot\Prog_A, \forall^\ord y \bigl(\tilde{n}\le_\ord y \lor A(y)\bigr)$}
        \RightLabel{(W)}
        \UnaryInfC{$[\ttX_n]\fCenter\lnot\Prog_A, \forall^\ord y \bigl(\tilde{n}\le_\ord y \lor A(y)\bigr), A(\tilde{n})$}

        \AxiomC{}
        \RightLabel{(Ax)}
        \UnaryInfC{$[\ttX_n]\fCenter \lnot A(\tilde{n}), A(\tilde{n})$}
        \RightLabel{(W)}
        \UnaryInfC{$[\ttX_n]\fCenter \lnot\Prog_A, \lnot A(\tilde{n}), A(\tilde{n})$}

        \RightLabel{($\land$)}
        \BinaryInfC{$[\ttX_n]\fCenter \lnot\Prog_A, \forall^\ord y \bigl(\tilde{n}\le_\ord y \lor A(y)\bigr) \land \lnot A(\tilde{n}), A(\tilde{n})$}

        \RightLabel{($\exists^\ord$)}
        \UnaryInfC{$[\ttX_n]\fCenter \lnot\Prog_A, A(\tilde{n})$}
    \end{mathprooftree}
    \end{center}
    Then the following proof tree gives a $\rmbeta$-preproof for the transfinite induction:
    \begin{center} \vspace{1em}
    \begin{mathprooftree}
        \def\fCenter{\vdash}
        \AxiomC{$\tilde{\pi}_0$}
        \noLine
        \UnaryInfC{$\vdots$}
        \UnaryInfC{$[\ttX_0]\fCenter \lnot\Prog_A, A(\tilde{0})$}
        \RightLabel{$\forall^\ord$}
        \UnaryInfC{$[\varnothing]\fCenter \lnot\Prog_A, \forall^\ord x A(x)$}
    \end{mathprooftree}
    \end{center}
    Now let us argue that it is a $\rmbeta$-proof: We can see that there is only one infinite branch of the previous $\rmbeta$-preproof, namely, a branch such that the sequence of the form $[\ttX_n]\vdash \lnot\Prog_A,A(\tilde{n})$ occurs cofinally often. The conclusion now follows since $\bigcup_n \ttX_n$ has ordertype $\omega^*$.
\end{example}

There is also a downside to the dendrogram-styled proof system. One of the major problems is that we can no longer rank proof trees by ordinals. We can still rank them by dilator terms with parameters from the template order, but it also means we cannot use inductive arguments by ranks of the proof. It becomes problematic when we prove, for example, cut elimination for proofs, since the typical proof uses induction on ranks. The solution is that we will not prove the cut elimination by induction on ranks, but by coinduction on preproofs. Hence, in principle, we can cut-eliminate preproofs that are not even proofs. 
It does not mean ranks are useless, since the rank certifies that a preproof is a proof. 
Such an idea appears in Towsner's work \cite{Towsner2025proofsmodifyproofs12}.

\subsection{Soundness and Completeness, again.} In this subsection, we provide a proof of $\rmbeta$-soundness and $\rmbeta$-completeness in terms of dendrogram-styled $\rmbeta$-logic. For soundness, we need to instantiate a given $\rmbeta$-preproof into an $\LK_\alpha$-preproof. (See \autoref{Definition: LKalpha} for the definition of $\LK_\alpha$.)
\begin{proposition}
    Let $\alpha$ be a linear order.
    Suppose that $\pi$ is an $\alpha$-proof of a sequent $\Gamma$ and $M$ is an $\rmbeta$-model with $\ord^M=\alpha$. Then $M\vDash \bigvee\Gamma$.
\end{proposition}
\begin{proof}
    Easily follows from an induction on $\pi$.
\end{proof}

\begin{definition}
    Let $\pi$ be an $\rmbeta$-preproof of the form
    \begin{equation*}
    \begin{mathprooftree}
        \AxiomC{$\pi'$}
        \noLine
        \UnaryInfC{$\vdots$}
        \RightLabel{($\rmR$)}
        \UnaryInfC{$[\ttX] \vdash \Gamma$}
    \end{mathprooftree}
    \end{equation*}
    $\alpha$ is a linear order, and $f\colon \ttX\to \alpha$ is an increasing function that can be thought of as an assignment of variables. We define the \emph{instantiation} $\pi[\alpha,f]$ corecursively as follows: 
    If the last inference rule $\rmR$ of $\pi$ is other than $\forall^\ord$, we define $\pi[\alpha,f]$ by
    \begin{equation*}
    \begin{mathprooftree}
        \AxiomC{$\pi'[\alpha,f]$}
        \noLine
        \UnaryInfC{$\vdots$}
        \RightLabel{($\rmR$)}
        \UnaryInfC{$\Gamma[f]$}
    \end{mathprooftree}
    \end{equation*}
    Here we define $A[f]$ for an $\LK_\rmbeta$-formula $A$ as a formula obtained from $A$ by replacing an occurrence of $v\in \ttX$ by $f(v)\in \alpha$. Then we define $\Gamma[f]$ as the sequent of $A[f]$ for $A\in\Gamma$.
    If $\rmR$ is $\forall^\ord$, then $\Gamma$ is $\Gamma',\forall^\ord x A(x)$, and $\pi$ would take the form
    \begin{equation} \label{Formula: LKbeta instantiation - forall ord case before instantiation}
    \begin{mathprooftree}
        \AxiomC{$\pi'_{\ttX'}$}
        \noLine
        \UnaryInfC{$\vdots$}
        \noLine
        \UnaryInfC{$\cdots\ [\ttX']\vdash A(w), \Gamma'\ \cdots$}
        
        \AxiomC{$\pi'_{\ttX,v}$}
        \noLine
        \UnaryInfC{$\vdots$}
        \noLine
        \UnaryInfC{$\cdots\ [\ttX]\vdash A(v), \Gamma'\ (v\in \ttX)\ \cdots$}
        
        \RightLabel{($\forall^\ord$)}
        \BinaryInfC{$[\ttX] \vdash \Gamma$}
    \end{mathprooftree}
    \end{equation}
    Here, $\ttX'$ is an one-element extension of $\ttX$, and $w\in \ttX'\setminus \ttX$. Then we define $\pi[\alpha,f]$ as follows:
    \begin{equation*}
    \begin{mathprooftree}
        \AxiomC{$\pi'_{\ttX'}[\alpha,g]$}
        \noLine
        \UnaryInfC{$\vdots$}
        \noLine
        \UnaryInfC{$\cdots\ A[\alpha,f](g(w)), \Gamma'[\alpha,f]\ \cdots$}
        
        \AxiomC{$\pi'_{\ttX,v}[\alpha,f]$}
        \noLine
        \UnaryInfC{$\vdots$}
        \noLine
        \UnaryInfC{$\cdots\ A[\alpha,f](f(v)), \Gamma'[\alpha,f]\ (v\in \ttX)\ \cdots$}
        
        \RightLabel{($\forall^\ord$)}
        \BinaryInfC{$\Gamma[\alpha,f]$}
    \end{mathprooftree}
    \end{equation*}
    Here we enumerate the left-hand side for all possible $g\colon \ttX'\to \alpha$ that extend $f\colon \ttX\to \alpha$.
\end{definition}

The following proposition immediately implies the soundness of $\rmbeta$-logic:
\begin{proposition}
    If $\pi$ is a $\rmbeta$-preproof, then $\pi[\alpha,f]$ is an $\alpha$-preproof. If $\pi$ is a $\rmbeta$-proof and $\alpha$ is a well-order, then $\pi[\alpha,f]$ is an $\alpha$-proof.
\end{proposition}
\begin{proof}
    We prove it by coinduction on $\pi$: The most non-trivial case is when the last inference rule of $\pi$ is $\forall^\ord$. We need to check that in $\pi[\alpha,f]$, the last part $\pi$ takes the form
\begin{equation*}
    \begin{mathprooftree}
        \AxiomC{$\varpi_\xi$}
        \noLine
        \UnaryInfC{$\vdots$}
        \noLine
        \UnaryInfC{$\cdots\ A[\alpha,f](\xi), \Gamma'[\alpha,f]\ \cdots$ $(\xi\in\alpha)$}
        
        \RightLabel{($\forall^\ord$)}
        \UnaryInfC{$\forall^\ord x A(x),\ \Gamma'[\alpha,f]$}
    \end{mathprooftree}
    \end{equation*}
    Suppose that $\xi\in \alpha$. We can extend $f\colon \ttX\to \alpha$ to $g\colon \ttX'\to \alpha$ in a way that $\xi\in \ran g$ and $\lvert\dom g\rvert-\lvert\dom f\rvert\le 1$. 
    If $\xi\notin \ran f$, then let us take $\varpi_\xi = \pi'_{\ttX'}[\alpha,g]$ from \eqref{Formula: LKbeta instantiation - forall ord case before instantiation}.
    If $\xi\in \ran f$, then there is $v\in\ttX$ such that $f(v)=\xi$. Then take $\varpi_\xi$ as $\pi'_{\ttX,v}[\alpha,f]$.

    Now suppose that $\pi$ is a $\rmbeta$-proof and $\alpha$ is a well-order. We claim that $\pi[\alpha,f]$ has no infinite branch.
    Suppose, for the sake of contradiction, that $\pi$ has an infinite branch $B=\lag \Gamma_n \mid n\in\bbN\rag$. By the definition of the instantiation, we can find an infinite branch $\lag [\ttX_n]\vdash \Gamma'_n\mid n\in\bbN\rag$ and $f_n\colon \ttX_n\to \alpha$ such that $\Gamma'_n[\alpha,f_n] = \Gamma_n$ and $f_n\subseteq f_{n+1}$ for every $n$.
    Then $f=\bigcup_n f_n$ is an embedding from $\ttX:=\bigcup_n \ttX_n$ to $\alpha$, so $\ttX$ is a well-order. It contradicts that $\pi$ is a $\rmbeta$-proof.
\end{proof}

$\rmbeta$-completeness follows from a stronger property that we will call the $\rmbeta$-preproof property:\footnote{The proof roughly follows the idea and terminologies in \cite[Theorem 5.19]{Jeon2025PTDandIntermediatePointclasses}. Let me also point out that the proof presented in the first version \cite{Jeon2025PTDandIntermediatePointclasses} was incorrect.}
\begin{theorem} \label{Theorem: Preproof property of LKbeta}
    Suppose that $[\ttX]\vdash \Gamma$ is a sequent of formulas in $\LK_\rmbeta$ with language $\calL$. We can construct an $\calL$-recursive $\rmbeta$-preproof $\pi$ of $[\ttX]\vdash \Gamma$ such that either
    \begin{enumerate}
        \item $\pi$ is a $\rmbeta$-proof, or
        \item If $B$ is an infinite branch of $\pi$ such that the union of its associated template orders is well-founded, then we can find a $B$-recursive $\rmbeta$-model of $\lnot\bigvee \Gamma$. 
    \end{enumerate}
    In particular, if $\calL$ is recursive, then $\pi$ is recursive.
\end{theorem}
\begin{proof}
    Before stating the main proof, the reader is warned that sequents in this proof are represented as lists and not finite sets. This list structure is only bookkeeping, and the underlying sequent calculus remains Tait-style.
    
    Let us fix a recursive enumeration $\epsilon_\calL$ enumerating all $\calL_{\{\tilde{n}\mid n\in\bbN\}}$-formulas, and a given $\calL_{\{\tilde{n}\mid n\in\bbN\}}$-formula occurs infinitely often. We also fix an enumeration $\epsilon_\calL^\tau$ of all $\calL_{\{\tilde{n}\mid n\in\bbN\}}$-terms with variables, with type $\tau$. 
    We also demand that the first $n$ terms on the list $\epsilon^\tau_\calL$ use ordinal variables from $\{\tilde{m}\mid m<n\}$. In particular, the first element of $\epsilon^\tau_\calL$ does not use any ordinal variables. This is possible since $\epsilon^\tau_\calL$ also enumerates variables of type $\tau$.
    
    We recursively build a preproof from the root (the last sequent). For each sequent $[\ttX]\vdash \Gamma$ we build, we assign the daggered formula in a sequent.
    In the initial case $[\ttX]\vdash\Gamma$, we mark the dagger to the first formula in $\Gamma$. Hence, we construct a $\rmbeta$-preproof from
    \begin{equation} \label{Formula: Proproof property dagger initial case}
        [\ttX]\vdash A_0^\dagger, A_1,\cdots,A_{m-1}
    \end{equation}
    when $\Gamma = A_0,\cdots,A_{m-1}$.    
    We call the dagger-marked formula the \emph{leading formula}.
    In each portion of the construction, we move the dagger to the right side. For example, if we start from \eqref{Formula: Proproof property dagger initial case}, and applying the one-step construction yields a portion with the foremost sequent $[\ttX']\vdash A_0,\cdots,A_{m-1},\Delta$, then we put the dagger as follows:
    \begin{equation*}
    \begin{mathprooftree}
        \AxiomC{$[\ttX']\vdash A_0, A_1^\dagger,\cdots,A_{m-1}, \Delta$}
        \noLine
        \UnaryInfC{$\vdots$}
        \noLine
        \UnaryInfC{$[\ttX]\vdash A_0^\dagger, A_1,\cdots,A_{m-1}$}
    \end{mathprooftree}
    \end{equation*}
    For each sequent $[\ttX]\vdash\Gamma$, we also assign the \emph{stage number} $\varsigma_{[\ttX]\vdash\Gamma}\in\bbN$. In the initial case, we take $\varsigma=0$.

    For a sequent $[\ttX]\vdash \Gamma$ in a preproof we construct, the \emph{level} of the sequent $[\ttX]\vdash\Gamma$ is the number of sequents strictly below it in the preproof. 
    Now we introduce the one-step construction:
    \begin{enumerate}
        \item Suppose that $[\ttX]\vdash \Gamma$ is a weakening of an axiom $[\ttX]\vdash\Gamma'$. Then we put the portion
        \begin{equation*}
        \begin{mathprooftree}
            \AxiomC{}
            \RightLabel{(Ax)}
            \UnaryInfC{$[\ttX]\vdash \Gamma'$}
            \noLine
            \UnaryInfC{$\vdots$}
            \RightLabel{(W)}
            \UnaryInfC{$[\ttX]\vdash\Gamma$}
        \end{mathprooftree}
        \end{equation*}
        and the construction terminates. Now consider the remaining cases.
    
        \item Suppose that the leading formula $C$ is atomic. Let $A$ be the $\varsigma$-th formula in the enumeration $\epsilon$. If every ordinal variable occurring in $A$ is in $\ttX$, then the following portion is
        \begin{equation} \label{Formula: Beta Completeness of LK - Atomic case 0}
        \begin{mathprooftree}
            \AxiomC{$[\ttX]\vdash \Gamma, A$}
            \AxiomC{$[\ttX]\vdash \Gamma, \lnot A$}
            \RightLabel{(Cut)}
            \BinaryInfC{$[\ttX]\vdash \Gamma$}
        \end{mathprooftree}
        \end{equation}

        Otherwise, we apply the structural rules to repeat $C$ to the last of the sequent:
        \begin{equation*}
        \begin{mathprooftree}
            \AxiomC{$[\ttX]\vdash \Gamma, C$}
            \noLine
            \UnaryInfC{$\vdots$}
            \UnaryInfC{$[\ttX]\vdash \Gamma$}
        \end{mathprooftree}
        \end{equation*}
        In both cases, we also increase the stage number by 1. (Note that all other cases will not increase the stage number.) For example, both $[\ttX]\vdash\Gamma,A$ and $[\ttX]\vdash\Gamma,\lnot A$ in \eqref{Formula: Beta Completeness of LK - Atomic case 0} has stage number $\varsigma_{[\ttX]\vdash\Gamma}+1$.

        \item Suppose that the leading formula of $\Gamma$ is $A\land B$. The following portion is
        \begin{equation*}
        \begin{mathprooftree}
            \AxiomC{$[\ttX]\vdash \Gamma, A$}
            \AxiomC{$[\ttX]\vdash \Gamma, B$}
            \RightLabel{($\land$)}
            \BinaryInfC{$[\ttX]\vdash \Gamma$}
        \end{mathprooftree}
        \end{equation*}

        \item Suppose that the leading formula of $\Gamma$ is  $A\lor B$.  The following portion is
        \begin{equation*}
        \begin{mathprooftree}
            \def\fCenter{\vdash}
            \AxiomC{$[\ttX]\fCenter\Gamma, A,B$}
            \RightLabel{($\lor$)}
            \UnaryInfC{$[\ttX]\fCenter\Gamma, B$}
            \RightLabel{($\lor$)}
            \UnaryInfC{$[\ttX]\fCenter \Gamma$}
        \end{mathprooftree}
        \end{equation*}
        Here, if the dagger of $\Gamma$ in $[\ttX]\vdash \Gamma$ is not on the rightmost side, the dagger in $[\ttX]\vdash \Gamma, A, B$ is still in $\Gamma$. Otherwise, $A$ takes the dagger in $[\ttX]\vdash \Gamma, A, B$.

        \item Suppose that the leading formula of $\Gamma$ is  $\forall^\tau x A(x)$ for some sort $\tau\neq\ord$. If $x$ is the G\"odel-number-least variable of sort $\tau$ not occurring in $\Gamma$, then the following portion is
        \begin{equation*}
        \begin{mathprooftree}
            \def\fCenter{\vdash}
            \AxiomC{$[\ttX]\fCenter \Gamma,A(x)$}
            \RightLabel{($\forall^\tau$)}
            \UnaryInfC{$[\ttX]\fCenter \Gamma$}
        \end{mathprooftree}
        \end{equation*}

        \item Suppose that the leading formula of $\Gamma$ is $\exists^\tau x A(x)$ for some sort $\tau\neq\ord$. Let $u_0,\cdots,u_{m-1}$ be the terms such that $u_i$ is the $\le\varsigma$-th term on the list $\epsilon^\tau_\calL$, and all ordinal variables occurring in $u_i$ are in $\ttX$. Then the following portion is
        \begin{equation*}
        \begin{mathprooftree}
            \def\fCenter{\vdash}
            \AxiomC{$[\ttX]\fCenter \Gamma,A(u_0),\cdots,A(u_{m-1})$}
            \RightLabel{($\exists^\tau$)}
            \UnaryInfC{$\vdots$}
            \noLine
            \UnaryInfC{$[\ttX]\fCenter \Gamma,A(u_{m-1})$}
            \RightLabel{($\exists^\tau$)}
            \UnaryInfC{$[\ttX]\fCenter \Gamma$}
        \end{mathprooftree}
        \end{equation*}
        Note that $m\neq 0$ since the first member of $\epsilon^\tau_\calL$ does not use any ordinal variables.

        \item Suppose that the leading formula of $\Gamma$ is  $\forall^\ord x A(x)$. The following portion is
        \begin{equation*}
        \begin{mathprooftree}
            \def\fCenter{\vdash}
            \AxiomC{$\cdots\ [\ttX']\fCenter \Gamma,A(w) \quad (w\notin \ttX)\ \cdots $}
            \AxiomC{$\cdots\ [\ttX]\fCenter \Gamma,A(w) \quad (w\in \ttX)\ \cdots$}
            \RightLabel{$(\forall^\ord)$}
            \BinaryInfC{$[\ttX]\fCenter \Gamma$}
        \end{mathprooftree}
        \end{equation*}
        Here $\ttX'$ is a one-element extension of $\ttX$, and we enumerate the hypotheses under all possible $\ttX'$ and $w\in \ttX$.

        \item Suppose that the leading formula of $\Gamma$ is  $\exists^\ord x A(x)$. If $\ttX = \{\tilde{k}\mid k<n\}$, then the following portion is
        \begin{equation*}
        \begin{mathprooftree}
            \def\fCenter{\vdash}
            \AxiomC{$[\ttX]\fCenter \Gamma,A(\tilde{0}),\cdots,A(\widetilde{n-1})$}
            \RightLabel{($\exists^\ord$)}
            \UnaryInfC{$\vdots$}
            \noLine
            \UnaryInfC{$[\ttX]\fCenter \Gamma,A(\widetilde{n-1})$}
            \RightLabel{($\exists^\ord$)}
            \UnaryInfC{$[\ttX]\fCenter \Gamma$}
        \end{mathprooftree}
        \end{equation*}
    \end{enumerate}
    \vspace{1em}
    
    This generates a recursive preproof $\pi$ of $[\ttX]\vdash\Gamma$. Now suppose that $\pi$ is not a $\rmbeta$-proof, and we have an infinite branch
    \begin{equation*}
        \bigl\langle [\ttX_n] \vdash \Gamma_n \bigm| n\in\bbN \bigr\rangle
    \end{equation*}
    with $\ttX_0=\ttX$, $\Gamma_0 = \Gamma$, $\ttX_n\subseteq\ttX_{n+1}$, and $\bigcup_n \ttX_n$ is a well-order. We build a $\rmbeta$-model of $\lnot\bigvee\Gamma$ recursive in the infinite branch. 
    First, let us observe that $\Gamma_n$ is an initial segment (as a list) of $\Gamma_m$ if $n\le m$. Let $\Gamma_\omega = \bigcup_n \Gamma_n$. Clearly, every formula in $\Gamma$ is daggered in some $\Gamma_m$. Furthermore,
    \begin{lemma} \label{Lemma: Preproof property pi - infinitely many atomic formulas}
        In $\Gamma_\omega$, an atomic formula occurs infinitely often. As a corollary, the sequence $\lag \varsigma_{[\ttX_n]\vdash\Gamma_n}\mid n<\omega\rag$ of stage numbers is unbounded. 
    \end{lemma}
    \begin{proof}
        We claim that for each $n$, there is $m>n$ such that the last formula of $\Gamma_m$ is atomic.
        Suppose not, assume that the last formula of $\Gamma_m$ is not atomic for every $m\ge n$. We still have that the last formula $C_0$ of $\Gamma_{n+1}$ is a strict subformula of the leading formula $\Gamma_n$.
        We can find $n_1\ge n+1$ such that the last formula $C_0$ of $\Gamma_{n+1}$ is the leading formula of $\Gamma_{n_1}$. 
        The last formula $C_1$ of $\Gamma_{n_1+1}$ is the strict subformula of $C_0$.
        We can inductively find $n_1<n_2<n_3<\cdots$ and formulas $C_k$ such that $C_k$ is the last formula of $\Gamma_{n_k+1}$, which is a strict subformula of $C_{k-1}$. This is impossible.
    \end{proof}

    The following lemma tells us that the infinite branch assigns the truth to formulas:
    \begin{lemma} \label{Lemma: Preproof property pi - A or not A always occur}
        Let $A$ be a formula in $\calL$ with ordinal parameters from $\ttX_n$. Then there is $m\ge n$ such that either $A$ or $\lnot A$ occurs in $\Gamma_m$.
    \end{lemma}
    \begin{proof}
        Take a large $m>n$ such that the leading formula of $[\ttX_m]\vdash\Gamma_m$ is atomic and the $\varsigma_{[\ttX_m]\vdash\Gamma_m}$-th formula of $\epsilon_\calL$ is $A$. Then the construction guarantees that $\Gamma_{m+1}$ has either $A$ or $\lnot A$ as the last formula. 
    \end{proof}

    Note that both $A$ and $\lnot A$ cannot occur in $\Gamma_n$: Otherwise, $[\ttX_n]\vdash \Gamma_n$ is a weakening of an axiom, so the branch cannot be infinite. Similarly, if $u,v\in \ttX_n$ and $\ttX_n\vDash u<v$, then $u<v$ cannot occur in $[\ttX_n]\vdash \Gamma_n$. 
    Now we define the term model $M$ from an infinite branch as follows:
    \begin{enumerate}
        \item For a sort $\tau$, the set of elements of $M$ of sort $\tau$ is the set of all terms of sort $\tau$.
        \item $\ord^M = \ttX= \bigcup_n \ttX_n$ with the order from $\ttX$. $\ord^M$ is well-founded by the assumption on $\ttX$.
        \item If $f$ is a function symbol, we define $f^M(t_0,\cdots,t_{n-1}) = f(t_0,\cdots,t_{n-1})$.
        \item $M\vDash A$ iff $\lnot A$ occurs in $\Gamma_n$ for some $n$.
    \end{enumerate}
    We need the following lemma to guarantee $M$ is really a model:
    \begin{lemma} For each $n\in\bbN$, let $A$ be the leading formula of $[\ttX_n]\vdash \Gamma_n$. 
        \begin{enumerate}
            \item If $A \equiv B\land C$, then either $B$ or $C$ is in $\Gamma_{n+1}$.
            \item If $A \equiv B\lor C$, then both $B$ and $C$ are in $\Gamma_{n+1}$.
            \item If $A \equiv \forall^\tau x B(x)$ for $\tau\neq \ord$, then there is a free variable $x$ of type $\tau$ such that $B(x)$ is in $\Gamma_{n+1}$.
            \item If $A \equiv \exists^\tau x B(x)$ for $\tau\neq \ord$, then for every term $t$ of type $\tau$, there is $m>n$ such that $B(t)$ occurs in $\Gamma_m$.
            \item If $A\equiv \forall^\ord x B(x)$, then there is $v\in \ttX$ such that $B(v)$ is in $\Gamma_{n+1}$.
            \item If $A\equiv \exists^\ord x B(x)$, then for every $v\in \ttX$ there is $m>n$ such that $B(v)$ occurs in $\Gamma_m$.
        \end{enumerate}
    \end{lemma}
    \begin{proof}
        All but the $\exists$ cases are clear. We first consider the case $A\equiv \exists^\tau xB(x)$ with $\tau\neq \ord$. For a given term $t$, take a large $m_0>n$ such that every ordinal variable used in $t$ occurs in $\ttX_{m_0}$. Then take a large $m>m_0$ such that $\varsigma_{[\ttX_m]\vdash\Gamma_m} = \varsigma$ satisfies
        \begin{enumerate}
            \item $A \equiv \exists^\tau x B(x)$ is the $\varsigma$-th formula of $\epsilon$.
            \item $t$ is the $\le\varsigma$-th element on the list $\epsilon^\tau_\calL$.
        \end{enumerate}
        By the construction, $\Gamma_{m'}$ for some $m'>m$ contains $B(t)$. The case $\tau=\ord$ also follows from a similar argument.
    \end{proof}
    The previous lemma guarantees, for example, $M\vDash B\lor C$ implies $M\vDash B$ or $M\vDash C$: If $M\vDash B\lor C$, then $\lnot B \land \lnot C$ occurs in $\Gamma_n$ for some $n$. By the previous lemma, either $\lnot B$ or $\lnot C$ occurs in $\Gamma_{n+1}$, so either $M\vDash B$ or $M\vDash C$.
    By \autoref{Lemma: Preproof property pi - A or not A always occur}, either $M\vDash A$ or $M\vDash\lnot A$ holds for every formula $A$ with parameters from $M$. Hence $M$ is a desired $\rmbeta$-model.
\end{proof}

\section{A functorial proof system for \texorpdfstring{$\KP$}{KP}, Part I}
In this section, we introduce a functorial proof system for $\KP$. The proof system we introduce in this section corresponds to the completeness theorem of $\rmbeta$-logic, and the proof system does not care about the rank of a given preproof.
Before introducing the proof system, we introduce Freund's functorial constructive hierarchy \cite{Freund2019Pi11}, which will form the terms in the infinitary proof system.
\begin{definition}
    For a linear order $\alpha$ and a transitive set $u$, we define a set $\bfL(u)_\alpha$ and the support function $\supp^{\bfL(u)}_\alpha\colon \bfL(u)_\alpha\to [\alpha]^{<\omega}$ as follows:
    \begin{enumerate}
        \item Each $v\in u$ is a term of $\bfL(u)_\alpha$ and $\supp^{\bfL(u)}_\alpha(v) = \varnothing$.
        \item For each $\xi\in\alpha$, we have a term $\ttL(u)_\xi$ with $\supp^{\bfL(u)}_\alpha(\ttL(u)_\xi) = \{\xi\}$.
        \item Let $\phi(x,y_0,\cdots,y_{n-1})$ be a $\Delta_0$-formula with its all free variables displayed.
        Now for $\xi\in\alpha$ and $a_0,\cdots,a_{n-1}\in \bfL(u)_\alpha$ with $\supp^{\bfL(u)}_\alpha(a_i) < \xi$, we form a term $\{x\in \ttL(u)_\xi \mid\phi(x,a_0,\cdots,a_{n-1})\}$ with
        \begin{equation*} \textstyle 
            \supp^{\bfL(u)}_\alpha\bigl(\{x\in \ttL(u)_\xi \mid\phi(x,a_0,\cdots,a_{n-1})\}\bigr) = \{\xi\}\cup \bigcup_{i<n} \supp^{\bfL(u)}_\alpha(a_i).
        \end{equation*}
    \end{enumerate}
\end{definition}

We can functorially transform terms in $\bfL(u)_\alpha$ as follows: 
\begin{definition} \label{Definition: Functorial boldface L}
    For an increasing $f\colon\alpha\to\beta$, we recursively define
    \begin{itemize}
        \item $\bfL(u)_f(v) = v$ for every $v\in u$.
        \item $\bfL(u)_f(\ttL(u)_\xi) = \ttL(u)_{f(\xi)}$, and
        \item $\bfL(u)_f\bigl(\{x\in \ttL(u)_\xi \mid\phi(x,a_0,\cdots,a_{n-1})\}\bigr) = \{x\in \ttL(u)_{f(\xi)} \mid\phi(x,\bfL(u)_f(a_0),\cdots,\bfL(u)_f(a_{n-1}))\}$.
    \end{itemize}
\end{definition}

We will sometimes equip an ambient linear order with a distinguished element $\ttOmega$ to delimit the allowed set-term indices. This will be the case for the system $\RS_H$, but not for $\sfu\RS$ or $\sfu\RS_\alpha$.

\begin{definition}
    Let $\kappa$ be a linear order. We define the terms in $\calL^u_\kappa$ as members of $\bfL(u)_\kappa$.
    If $\kappa$ is equipped with a distinguished member $\ttOmega\in\kappa$, however, we restrict the terms to $\bfL(u)_{\{\alpha\in \kappa\mid\alpha<\ttOmega\}}$.
    
    We define $\calL^u_\kappa$-formulas as follows:
    \begin{enumerate}
        \item For two terms $s$ and $t$, $s\in t$ and $s\notin t$ are formulas.
        \item We can form new formulas by $\land$, $\lor$, bounded quantifications $\forall x\in s A(x)$ and $\exists x\in s A(x)$, and unbounded quantifications $\forall x A(x)$ and $\exists x A(x)$.
    \end{enumerate}
    We also define $|s| = \max (\supp^{\bfL(u)}_\kappa(s))$ computed in $\kappa$. Note that
    \begin{equation*}
        \lvert \ttL(u)_\xi\rvert = \lvert \{x\in \ttL(u)_\xi \mid \phi(x,a_0,\cdots,a_{n-1})\}\rvert = \xi.
    \end{equation*}
\end{definition}
For $s\in u$, the expression $|s|$ is undefined, but we adopt the convention that $|s|<\xi$ holds for every $\xi\in\kappa$.
Equivalently, $|s|<\xi$ means $\supp^{\bfL(u)}_\kappa(s)<\xi$, which also covers the case of empty support.

Now, let us define the height of formulas (see \cite[Theorem 3.14]{Freund2019Pi11}). 
Before defining the infinitary derivation system, let us define
\begin{equation*}
    s\dotin t \equiv
    \begin{cases}
        \top & \text{If $t = \ttL(u)_\xi$}, \\
        \phi(s,a_0,\cdots,a_{n-1}) & \text{If }t = \{x\in \ttL(u)_\xi \mid \phi(x,a_0,\cdots,a_{n-1})\}.
    \end{cases}
\end{equation*}
We also use the following notations:
\begin{itemize}
    \item $s\subseteq t \equiv \forall x\in s(x\in t)$.
    \item $s=t \equiv s\subseteq t \land t\subseteq s$.
\end{itemize}

We also equip $\bfL(u)$ with an auxiliary functorial linear order. Following Freund \cite{Freund2019Pi11}, we construct a linear order over $\bfL(u)_X$ when we are given a countable enumeration $\{u_i\mid i\in\bbN\}$ of $u$.
We first define
\begin{equation*}
    \Tr\bigl(\bfL(u)\bigr) = \bigl\{ (n,s) \bigm| s\in \bfL(u)_n \land \supp^{\bfL(u)}_n(s) = n\bigr\}.
\end{equation*}
By the support condition for $\bfL(u)$, every $s\in\bfL(u)_X$ admits a unique representation
\begin{equation} \label{Formula: Unique representation of L(u) members}
    s=\bfL(u)_{\en_a}(\hat{s})
\end{equation}
where $a=\supp^{\bfL(u)}_X(s)$, $\en_a\colon |a|\to a$ is the increasing enumeration of $a$, and $(|a|,\hat{s})\in\Tr(\bfL(u))$.
As as in the proof of \cite[Theorem 3.7]{Freund2019Pi11}, we fix a Gödel coding $\mathrm{code}^0_n\colon \bfL(u)_n\to \bbN$ for each $n\in\bbN$. Then for $s\in \bfL(u)_X$ with the unique representation \eqref{Formula: Unique representation of L(u) members}, let us define
\begin{equation*}
    c_X(s) = \mathrm{code}^0_{\left\lvert a\right\rvert} (\hat{s})
\end{equation*}
We note that $c_X(s)$ coincides with Freund's $\mathrm{code}_X(a,s)$ defined in \cite[Theorem 3.7]{Freund2019Pi11}. Hence we have the following functoriality: for $f\colon X\to Y$ and $s\in \bfL(u)_X$, we have $c_Y\bigl(\bfL(u)_f(s)\bigr) = c_X(s)$.

\begin{definition}
    Let $s,t\in \bfL(u)_X$. We define $s <^{\bfL(u)}_X t$ if and only if either
    \begin{enumerate}
        \item The maximal element of the symmetric difference $\supp^{\bfL(u)}_X(s) \mathop{\triangle} \supp^{\bfL(u)}_X(t)$ is in $\supp^{\bfL(u)}_X(t)$; put differently, $\supp^{\bfL(u)}_X(s)$ is smaller than $\supp^{\bfL(u)}_X(t)$ when we compare them in anti-lexicographic order (We identify finite sets with their increasing enumeration), or
        
        \item $\supp^{\bfL(u)}_X(s) = \supp^{\bfL(u)}_X(t)$ and $c_X(s) < c_X(t)$.
    \end{enumerate}
\end{definition}
Then we have the following:
\begin{theorem}[{\cite[Theorem 3.11]{Freund2019Pi11}}] \pushQED{\qed}
    Suppose that $u$ is a countable set with a fixed enumeration. Then $(\bfL(u),<^{\bfL(u)})$ is a dilator. \qedhere 
\end{theorem}
We can also see that $(\bfL(u),<^{\bfL(u)})$ is a normal flower, although this is unnecessary in this paper.

\subsection{System \texorpdfstring{$\sfu\RS$}{uRS} and its soundness}
Now we introduce an infinitary derivation system $\sfu\RS$.\footnote{An acronym of \emph{uncontrolled $\RS$.}} Like in the previous section, every sequent will take the form $[\ttX]\vdash \Gamma$.
\begin{definition}
    Let us define the deduction rules of $\sfu\RS$ as follows, where \emph{in the sequent $[\ttX]\vdash \Gamma$, every set term is a term in $\calL^u_\ttX$.}
    \begin{center}
    \vspace{1em}
    \begin{mathprooftree}
        \def\fCenter{\vdash}
        \AxiomC{}
        \RightLabel{$(\mathrm{Ax}) \quad (x\in y\in u)$}
        \UnaryInfC{$[\ttX]\fCenter \Gamma,\ x\in y$}
    \end{mathprooftree}
    \hspace{3em}%
    \begin{mathprooftree}
        \def\fCenter{\vdash}
        \AxiomC{}
        \RightLabel{$(\mathrm{Ax}) \quad (x\notin y\in u,\ x\in u)$}
        \UnaryInfC{$[\ttX]\fCenter \Gamma,\ x\notin y$}
    \end{mathprooftree}
    
    \vspace{1em}
    \begin{mathprooftree}
        \def\fCenter{\vdash}
        \AxiomC{$[\ttX]\fCenter \Gamma,\ A$}
        \AxiomC{$[\ttX]\fCenter  \Gamma,\ B$}
        \RightLabel{$(\land)$}
        \BinaryInfC{$[\ttX]\fCenter \Gamma,\ A\land B$}
    \end{mathprooftree}
    \hspace{3em}%
    \begin{mathprooftree}
        \def\fCenter{\vdash}
        \AxiomC{$[\ttX]\fCenter \Gamma,\ A$}
        \RightLabel{$(\lor)$}
        \UnaryInfC{$[\ttX]\fCenter \Gamma,\ A\lor B$}
    \end{mathprooftree}
    \hspace{3em}%
    \begin{mathprooftree}
        \def\fCenter{\vdash}
        \AxiomC{$[\ttX]\fCenter \Gamma,\ B$}
        \RightLabel{$(\lor)$}
        \UnaryInfC{$[\ttX]\fCenter \Gamma,\ A\lor B$}
    \end{mathprooftree}

    \vspace{1em}
    \begin{mathprooftree}
        \def\fCenter{\vdash}
        \AxiomC{$[\ttX]\fCenter \Gamma,\ s\dotin t\land r=s$}
        \RightLabel{$(\in)$ \hspace{2em} ($|s|<|t|$, $r\in t$ not basic.)}
        \UnaryInfC{$[\ttX]\fCenter \Gamma,\ r\in t$}
    \end{mathprooftree}
    
    \vspace{1em}
    \begin{mathprooftree}
        \def\fCenter{\vdash}
        \AxiomC{$\cdots \quad [\ttX_s] \fCenter \Gamma,\ s\dotin t\to r\neq s \quad \cdots$ ($\heartsuit_t$)}
        \RightLabel{$(\notin)$}
        \UnaryInfC{$[\ttX]\fCenter \Gamma,\ r\notin t$}
    \end{mathprooftree}
    \hspace{3em}%
    \begin{mathprooftree}
        \def\fCenter{\vdash}
        \AxiomC{$[\ttX]\fCenter \Gamma,\ s\dotin t \land B(s)$}
        \RightLabel{$(\rmb\exists)$ \hspace{2em} ($|s|<|t|$)}
        \UnaryInfC{$[\ttX]\fCenter \Gamma,\ \exists x\in t B(x)$}
    \end{mathprooftree}
    
    \vspace{1em}
    \begin{mathprooftree}
        \def\fCenter{\vdash}
        \AxiomC{$ \cdots \quad [\ttX_s]\fCenter \Gamma,\ s\dotin t\to B(s) \quad \cdots$ ($\heartsuit_t$)}
        \RightLabel{$(\rmb\forall)$}
        \UnaryInfC{$[\ttX]\fCenter \Gamma,\ \forall x\in t B(x)$}
    \end{mathprooftree}

    \vspace{1em}
    \begin{mathprooftree}
        \def\fCenter{\vdash}
        \AxiomC{$[\ttX]\fCenter \Gamma,\ B(s)$}
        \RightLabel{$(\exists)$}
        \UnaryInfC{$[\ttX]\fCenter \Gamma,\ \exists x B(x)$}
    \end{mathprooftree}
    \hspace{3em}%
    \begin{mathprooftree}
        \def\fCenter{\vdash}
        \AxiomC{$ \cdots \quad [\ttX_s]\fCenter \Gamma,\ B(s) \quad \cdots$ ($\heartsuit$)}
        \RightLabel{$(\forall)$}
        \UnaryInfC{$[\ttX]\fCenter \Gamma,\ \forall x B(x)$}
    \end{mathprooftree}
    
    \vspace{1em}
    \begin{mathprooftree}
        \def\fCenter{\vdash}
        \AxiomC{$[\ttX]\fCenter \Gamma,\ A$}
        \RightLabel{(Refl) \hspace{2em} ($A$ is a $\Sigma$-formula)}
        \UnaryInfC{$[\ttX]\fCenter \Gamma,\ \exists z A^z$}
    \end{mathprooftree}
    \hspace{3em}%
    \begin{mathprooftree}
        \def\fCenter{\vdash}
        \AxiomC{$[\ttX]\fCenter \Gamma,\ A$}
        \AxiomC{$[\ttX]\fCenter \Gamma,\ \lnot A$}
        \RightLabel{(Cut)}
        \BinaryInfC{$[\ttX]\fCenter \Gamma$}
    \end{mathprooftree}
    \vspace{1em}
    \end{center}
    Here $(\heartsuit)$ is the following conditions over $s$ and $\ttX_s$: 
    \begin{itemize}
        \item The restriction of the $\ttX_s$-order to $\ttX_s\setminus\ttX$ coincides with the inherited order from the \emph{reverse of} $\bbN$. It is possible that $\ttX_s = \ttX$.
        \item $\ttX_s = \ttX \cup \supp^{\bfL(u)}_{\ttX_s}(s)$
    \end{itemize}
    $(\heartsuit_t)$ is the conjunction of $(\heartsuit)$ and $\ttX_s\vDash |s|<|t|$.
    $(\heartsuit)$- and $(\heartsuit_t)$-labeled rules have hypotheses for each $s$ and $\ttX_s$ ($\alpha_s$ is from $s$) under the constraint $(\heartsuit)$. 
    The second condition of $(\heartsuit)$ makes the extension canonical: since the underlying set of a template order is a finite initial segment of $\bbN$, the extended order is uniquely determined by the original template, the insertion positions, and the number of new variables inserted at each position.
    We choose the reverse natural-number order so that the numbering of newly introduced variables agrees with their antilexicographic priority in $\bfL(u)$.
    
    Every tree formed by the given rule is a $\rmbeta$-preproof. We call a $\rmbeta$-preproof $\pi$ from $\sfu\RS$ a \emph{$\rmbeta$-proof} if for every infinite branch $B$ of $\pi$, the union of all template orders occurring in $B$ is ill-founded.

    We also add the following as an axiom:
    \begin{equation*}
    \begin{mathprooftree}
        \AxiomC{}
        \RightLabel{(Ax)}
        \UnaryInfC{$[\ttX]\vdash \Gamma, A$}
    \end{mathprooftree}
    \end{equation*}
    where $A$ is one of the following statements for $s,t,\vec{t}\in \calL^u_\ttX$.
    \begin{itemize}
        \item Pairing: $
        \exists z
        \left(
          s\in z\wedge t\in z
        \right)$.
        
        \item Union: $
        \exists y\,
        \forall z\in s\,
        \forall w\in z\,
        (w\in y)$.
        
        \item Infinity: $\exists x
        \left[
          \exists y\in x\, (y\in x)
          \wedge
          \forall y\in x\,\exists z\in x\,(y\in z)
        \right]$.
        
        \item $\Delta_0$-Separation:
        For every $\Delta_0$-formula $A(x,\vec a)$,
        \begin{equation*}
        \exists z\,
        \left[
          \forall x\in z\left(
            x\in s\wedge A(x,\vec t)
          \right)
          \wedge
          \forall x\in s\left(
            \neg A(x,\vec t)
            \vee x\in z
          \right)
        \right].
        \end{equation*}
    \end{itemize}
    We add these axioms into $\sfu\RS$ because they are not provable from the rest of the axioms of $\sfu\RS$.
\end{definition}
Unlike $\LK_\rmbeta$, we allow arbitrary finite extensions of $\ttX$ in $(\rmb\forall)$ and $(\notin)$ rules since a set term admits more than one ordinal variable. 
Similar to $\LK_\rmbeta$, we can instantiate the $\rmbeta$-preproof in $\sfu\RS$. We first need to define an instantiated proof system for a given ordinal $\alpha$:
\begin{definition}
    Fix a linear order $\alpha$.
    Let us define the deduction rules of $\sfu\RS_\alpha$ as follows, where every sequent is in the language $\calL^u_\alpha$.
    \begin{center}
    \vspace{1em}
    \begin{mathprooftree}
        \AxiomC{}
        \RightLabel{$(\mathrm{Ax}) \quad (x\in y\in u)$}
        \UnaryInfC{$\Gamma,\ x\in y$}
    \end{mathprooftree}
    \hspace{3em}%
    \begin{mathprooftree}
        \AxiomC{}
        \RightLabel{$(\mathrm{Ax}) \quad (x\notin y\in u,\ y\in u)$}
        \UnaryInfC{$\Gamma,\ x\notin y$}
    \end{mathprooftree}
    
    \vspace{1em}
    \begin{mathprooftree}
        \AxiomC{$\Gamma,\ A$}
        \AxiomC{$\Gamma,\ B$}
        \RightLabel{$(\land)$}
        \BinaryInfC{$\Gamma,\ A\land B$}
    \end{mathprooftree}
    \hspace{3em}%
    \begin{mathprooftree}
        \AxiomC{$\Gamma,\ A$}
        \RightLabel{$(\lor)$}
        \UnaryInfC{$\Gamma,\ A\lor B$}
    \end{mathprooftree}
    \hspace{3em}%
    \begin{mathprooftree}
        \AxiomC{$\Gamma,\ B$}
        \RightLabel{$(\lor)$}
        \UnaryInfC{$\Gamma,\ A\lor B$}
    \end{mathprooftree}

    \vspace{1em}
    \begin{mathprooftree}
        \AxiomC{$\Gamma,\ s\dotin t\land r=s$}
        \RightLabel{$(\in)$ \hspace{2em} ($|s|<|t|$, $r\in t$ not basic.)}
        \UnaryInfC{$\Gamma,\ r\in t$}
    \end{mathprooftree}
    
    \vspace{1em}
    \begin{mathprooftree}
        \def\fCenter{\vdash}
        \AxiomC{$\cdots \quad \Gamma,\ s\dotin t\to r\neq s \quad \cdots$ ($|s|<|t|$)}
        \RightLabel{$(\notin)$}
        \UnaryInfC{$\Gamma,\ r\notin t$}
    \end{mathprooftree}
    \hspace{3em}%
    \begin{mathprooftree}
        \AxiomC{$\Gamma,\ s\dotin t \land B(s)$}
        \RightLabel{$(\rmb\exists)$ \hspace{2em} ($|s|<|t|$)}
        \UnaryInfC{$\Gamma,\ \exists x\in t B(x)$}
    \end{mathprooftree}
    
    \vspace{1em}
    \begin{mathprooftree}
        \AxiomC{$ \cdots \quad \Gamma,\ s\dotin t\to B(s) \quad \cdots$ ($|s|<|t|$)}
        \RightLabel{$(\rmb\forall)$}
        \UnaryInfC{$\Gamma,\ \forall x\in t B(x)$}
    \end{mathprooftree}

    \vspace{1em}

    \begin{mathprooftree}
        \AxiomC{$\Gamma,\  B(s)$}
        \RightLabel{$(\exists)$}
        \UnaryInfC{$\Gamma,\ \exists x B(x)$}
    \end{mathprooftree}
    \hspace{3em}%
    \begin{mathprooftree}
        \AxiomC{$ \cdots \quad \Gamma,B(s) \quad \cdots$}
        \RightLabel{$(\forall)$}
        \UnaryInfC{$\Gamma,\ \forall x B(x)$}
    \end{mathprooftree}
    
    \vspace{1em}
    \begin{mathprooftree}
        \AxiomC{$\Gamma,\ A$}
        \RightLabel{(Refl) \hspace{2em} ($A$ is a $\Sigma$-formula)}
        \UnaryInfC{$\Gamma,\ \exists z A^z$}
    \end{mathprooftree}
    \hspace{3em}%
    \begin{mathprooftree}
        \def\fCenter{\vdash}
        \AxiomC{$\Gamma,\ A$}
        \AxiomC{$\Gamma,\ \lnot A$}
        \RightLabel{(Cut)}
        \BinaryInfC{$\Gamma$}
    \end{mathprooftree}
    \vspace{1em}
    \end{center}
    We also add axioms of the form $\Gamma, A$ for an instance $A$ of Pairing, Union, Infinity, and $\Delta_0$-Separation.
    
    Every tree formed by the given rule is an $\sfu\RS_\alpha$-preproof. We call a $\sfu\RS_\alpha$-preproof $\pi$ a \emph{proof} if $\pi$ has no infinite branch.
\end{definition}
We can interpret a given sequent in $\sfu\RS_\alpha$ as follows, which is due to Freund \cite[Definition 3.2]{Freund2019Pi11}.
\begin{definition}
    For a transitive set $u$ and an ordinal $\alpha$, we define $\llbr s\rrbr$ for an $\calL^u_{\alpha}$-term and an $\calL^u_{\alpha}$-formula recursively as follows:
    \begin{enumerate}
        \item $\llbr v\rrbr = v$ for $v\in u$.
        \item $\llbr \ttL(u)_\xi \rrbr = L(u)_\xi$.
        \item $\llbr \{x\in \ttL(u)_\xi \mid A(x,s_0,\cdots,s_{n-1})\}\rrbr = \{x\in L(u)_\xi \mid A(x,\llbr s_0\rrbr,\cdots,\llbr s_{n-1}\rrbr)\}$, where $A(x,y_0,\cdots,y_{n-1})$ is $\Delta_0$ with parameters from $u\cup\{u\}$ and all free variables displayed.
        \item $\llbr s\in t\rrbr \equiv (\llbr s\rrbr \in \llbr t\rrbr)$, and similar to $\notin$.
        \item $\llbr A \circ B\rrbr \equiv \llbr A\rrbr\circ \llbr B\rrbr$, for $\circ\in \{\land,\lor\}$.
        \item $\llbr\sfQ x\in s A(x) \rrbr\equiv \sfQ x\in \llbr s\rrbr \llbr A\rrbr(x)$, for $\sfQ\in \{\forall,\exists\}$ and a set term $s$.
        \item $\llbr\sfQ x A(x) \rrbr\equiv \sfQ x \llbr A\rrbr(x)$, for $\sfQ\in \{\forall,\exists\}$.
    \end{enumerate}
    For a sequent $\Gamma$, we define $\llbr\Gamma\rrbr$ by $\bigvee\{\llbr A\rrbr \mid A\in \Gamma\}$.
\end{definition}
Note that $\llbr A\rrbr$ is always relativized to $L(u)_\alpha$. Here we warn that we are using a different definition of $L(u)_\alpha$: $L(u)_\alpha$ is the union of 
\begin{enumerate}
    \item $\{u\}\cup \trcl(u)$,
    \item $L(u)_\beta$ for all $\beta<\alpha$, and
    \item The set of all subsets of the form $\{x\in L(u)_\beta\mid A(x,\vec{a})\}$ for some $\beta<\alpha$, a $\Delta_0$-formula $A(x,\vec{y})$, and $\vec{a}\in L(u)_\beta$.
\end{enumerate}

The following easily follows by induction on a proof:
\begin{proposition} \pushQED{\qed} \label{Proposition: Soundness of uRS proof}
    Suppose that $L(u)_\alpha$ is admissible. If there is an $\sfu\RS_\alpha$-proof of a sequent $\Gamma$, then $L(u)_\alpha\vDash \llbr \Gamma\rrbr$.
    If $\Gamma$ has a $\sfu\RS_\alpha$-proof without using $(\mathrm{Refl})$ rule, then $L(u)_\alpha\vDash \Gamma$ for a limit $\alpha>\omega$ regardless of the admissibility of $\alpha$. \qedhere 
\end{proposition}

Also, we can instantiate a $\rmbeta$-proof in $\sfu\RS$ into a $\sfu\RS_\alpha$-proof for $\alpha$ as follows:
\begin{definition}
    Suppose that $\pi$ is a $\rmbeta$-preproof of the form
    \begin{equation*}
    \begin{mathprooftree}
        \AxiomC{$\pi'$}
        \noLine
        \UnaryInfC{$\vdots$}
        \RightLabel{($\rmR$)}
        \UnaryInfC{$[\ttX] \vdash \Gamma$}
    \end{mathprooftree}
    \end{equation*}    
    For a linear order $\alpha$ and an increasing map $f\colon \ttX\to\alpha$, we define $\pi[\alpha,f]$ corecursively as follows:
    If the rule $(\rmR)$ is other than $(\notin)$ or $(\rmb\forall)$, $\pi[\alpha,f]$ is the $\alpha$-preproof
    \begin{equation*}
    \begin{mathprooftree}
        \AxiomC{$\pi'[\alpha,f]$}
        \noLine
        \UnaryInfC{$\vdots$}
        \RightLabel{($\rmR$)}
        \UnaryInfC{$f[\Gamma]$}
    \end{mathprooftree}
    \end{equation*}
    Here, for an $\calL^u_\ttX$-formula $A$, $A[f]$ is an $\calL^u_\alpha$-formula obtained from $A$ by replacing the occurrence of $v\in \ttX$ with $f(v)\in \alpha$, and $f[\Gamma] = \{f(A)\mid A\in \Gamma\}$, where $f(A) = \bfL(u)_{f}(A)$.
    
    Now suppose that $(\rmR)$ is either $(\notin)$, $(\rmb\forall)$, or $(\forall)$: We only consider the case when $(\rmR) = (\rmb\forall)$ since the other cases are analogous. In this case, we have
    \begin{equation} \label{Formula: uRS last is bounded forall}
    \begin{mathprooftree}
        \def\fCenter{\vdash}
        \noLine
        \AxiomC{$\pi'_s$}
        \UnaryInfC{$\vdots$}
        \noLine
        \UnaryInfC{$ \cdots \quad [\ttX_s]\fCenter \Gamma,\ s\dotin t\to B(s) \quad \cdots$ ($\heartsuit_t$)}
        \RightLabel{$(\rmb\forall)$}
        \UnaryInfC{$[\ttX]\fCenter \Gamma,\ \forall x\in t B(x)$}
    \end{mathprooftree}
    \end{equation}
    Then we define $\pi[\alpha,f]$ by
    \begin{equation} \label{Formula: uRS alpha last is bounded forall}
    \begin{mathprooftree}
        \def\fCenter{\vdash}
        \noLine
        \AxiomC{$\pi'_s[\alpha,g_r]$}
        \UnaryInfC{$\vdots$}
        \noLine
        \UnaryInfC{$ \cdots \quad f[\Gamma],\ g_r(s)\dotin f(t)\to g_r\bigl(B(s)\bigr) \quad \cdots$ ($\diamondsuit_{f(t)}$)}
        \RightLabel{$(\rmb\forall)$}
        \UnaryInfC{$f[\Gamma],\ \forall x\in f(t)\ f(B)(x)$}
    \end{mathprooftree}
    \end{equation}
    Here $(\diamondsuit_{f(t)})$ means that the premises are indexed by all $\calL^u_\alpha$-terms $r$ satisfying $\alpha\vDash |r|<|f(t)|$.
    For each such $r$, let $\ttY_r\supseteq\ttX$ and
    $g_r\colon\ttY_r\to\alpha$ be the unique canonical extension of $f$
    such that
    \begin{enumerate}
        \item over $\ttY_r\setminus\ttX$, the $\ttY_r$-order agrees with the reverse natural number order, and
        \item $\ran g_r
            =
            \ran f\cup\supp^{\bfL(u)}_\alpha(r)$.
    \end{enumerate}
    By the support condition for $\bfL(u)$, there is a unique
    $\calL^u_{\ttY_r}$-term $s(r)$ such that $\bfL(u)_{g_r}\bigl(s(r)\bigr)=r$.
    Since $g_r\bigl(|s(r)|\bigr)=|r|<|f(t)|=g_r(|t|)$,  we have $\ttY_r\vDash |s(r)|<|t|$. Hence $(\ttY_r,s(r))$ indexes one of the premises of the source $(\rmb\forall)$-inference.
\end{definition}

\begin{proposition} \label{Proposition: Instantiating uRS proof into a uRS alpha proof}
    Suppose that $\pi$ is a $\rmbeta$-proof in $\sfu\RS$ for the sequent $[\ttX]\vdash \Gamma$.
    For a well-order $\alpha$ and a function $f\colon \ttX\to\alpha$,  $\pi[\alpha,f]$ is a $\sfu\RS_\alpha$-proof.
\end{proposition}
\begin{proof}
    We can easily see that if $\pi$ is a $\rmbeta$-proof for the sequent $[\ttX]\vdash\Gamma$ in $\sfu\RS$, then $\pi[\alpha,f]$ has no infinite branch. Now we verify by induction on the tree structure of $\pi[\alpha,f]$ that $\pi[\alpha,f]$ is a valid $\sfu\RS_\alpha$-preproof. The most non-trivial case is when the inference rule is either $(\rmb\forall)$ or $(\notin)$, and we only consider the case $(\rmb\forall)$.

    Suppose that $\pi$ takes the form \eqref{Formula: uRS last is bounded forall}, so $\pi[\alpha,f]$ takes the form \eqref{Formula: uRS alpha last is bounded forall}. We need to check that sequents of $(\diamondsuit)$ are enumerated by $\calL^u_{\alpha}$-term $r$ for $|r|<|f(t)|$. Put differently, we need to check that if $r$ is an $\calL^u_{\alpha}$-term, then there are template order $\ttY \supseteq \ttX$, an $\calL^u_{\ttY}$-term $s$ (so $\ttY = \ttX_s$) and an extension $g\colon \ttY\to \alpha$ of $f$ such that $r = f(s)$.
    Suppose that we are given an $\calL^u_{\alpha}$-term $r$. We can find a unique $\ttY \supseteq\ttX$ and $g\colon \ttY\to \alpha$ such that
    \begin{enumerate}
        \item Over $\ttY\setminus \ttX$, the $\ttY$-order and the reverse natural number order agree.
        \item The range of $g$ is $\ran f \cup \supp^{\bfL(u)}_\alpha(r)$.
    \end{enumerate}
    Then we can find an $\calL^u_{\ttY}$-term $s$ such that $\bfL(u)_{\alpha}(s) = r$, as desired.
    The other part --- $\ttX_s \vDash |s|<|t|$ implies $\alpha\vDash |g(s)| < |f(t)|$ --- is easy to see.
\end{proof}

\subsection{Embedding \texorpdfstring{$\KP$}{KP} into \texorpdfstring{$\sfu\RS$}{uRS}} \label{Subsection: Embedding KP into uRS}

In this section, we prove that $\sfu\RS$ proves every theorem of $\KP$.
We will not examine the proof of every fact in full detail since the arguments we need are more or less similar to those of \autoref{Section: Embedding}, which are also similar to those in \cite{RathjenCook2016ClassifyingProvTotalsetftnKPKPP, Fernandez2024SetTotalRecKPell}. We only sketch what we need in this subsection.

\begin{proposition}
    If $[\ttX]\vdash \Gamma$ has a $\rmbeta$-proof in $\sfu\RS$, $\ttX'\supseteq \ttX$, and $\Gamma'\supseteq\Gamma$, then $[\ttX']\vdash \Gamma'$ also has a $\rmbeta$-proof in $\sfu\RS$.
\end{proposition}

\begin{definition}
    A sentence in the language of $\RS$ is \emph{axiomatic} if it is one of the following: 
    \begin{enumerate}
        \item Extensionality: $\forall x\,
        \forall y
        \left[
          \neg\left(
            \forall z\in x\, z\in y
            \wedge
            \forall z\in y\, z\in x
          \right)
          \vee x=y
        \right].$
        
        \item Set Induction: $\neg
        \forall x
        \left(
          \neg\forall y\in x\,A(y)\vee A(x)
        \right)
        \vee
        \forall x A(x).$
        
        \item $\Delta_0$-Collection:
        For every $\Delta_0$-formula $A(x,y,\vec a)$,
        \begin{equation*}
        \forall b\,
        \forall \vec a
        \left[
          \neg\forall x\in b\,\exists y A(x,y,\vec a)
          \vee
          \exists c
          \forall x\in b\,\exists y\in c\,A(x,y,\vec a)
        \right].
        \end{equation*}
    \end{enumerate}

    A sequent $\Gamma$ in the language of $\RS$ is \emph{axiomatic} if $\Gamma$ contains one of the following subsequent:
    \begin{enumerate}
        \item An axiomatic sentence.
        \item $A,\lnot A$ for some formula $A$.
        \item $s\notin s$ for some set term $s$.
        \item $s\subseteq s$ for some set term $s$.
        \item $s\neq t, \lnot A(s),A(t)$ for some set term $s$, $t$, and some formula $A$.
    \end{enumerate}
\end{definition}

Then we have the following. Its proof is more or less similar to those occurring in \autoref{Section: Embedding}, so we omit the details.
\begin{theorem} \pushQED{\qed}
    Let $\ttX$ be a template order and $\Gamma$ be a sequent of $\sfu\RS$. If $\Gamma$ is axiomatic, then there is a $\rmbeta$-proof of $[\ttX]\vdash\Gamma$. \qedhere 
\end{theorem}

\subsection{Completeness of \texorpdfstring{$\sfu\RS$}{uRS}}
We are going to prove the completeness theorem for $\sfu\RS$. In this subsection, we assume that $u$ is countable, and identify $u$ with its code $(u,\cup \{u\},\in)$ into the set of natural numbers.
\begin{theorem} \label{Theorem: Preproof property for uRS}
    Suppose that $[\ttX]\vdash\Gamma$ is a sequent of $\sfu\RS$, and let $\epsilon_u$ be a countable enumeration of elements of $u$. Then we can construct an $\epsilon_u$-recursive $\rmbeta$-preproof $\pi$ of $[\ttX]\vdash\Gamma$ such that either
    \begin{enumerate}
        \item $\pi$ is a $\rmbeta$-proof, or
        \item If $B$ is an infinite branch of $\pi$ such that the union of its associated template orders is well-founded, then we can find a $B$-recursive well-founded model of $\KP + \lnot\bigvee\Gamma$.
    \end{enumerate}
\end{theorem}
\begin{proof}
    Its proof is more or less similar to \autoref{Theorem: Preproof property of LKbeta}. Similar to the proof of \autoref{Theorem: Preproof property of LKbeta}, we represent sequents as finite lists of formulas. Again, the underlying sequent calculus is still Tait-style.
    
    Let us fix a recursive enumeration $\epsilon$ of all $\calL^u_{\{\tilde{n}\mid n\in\bbN\}}$-formulas such that every formula occurs infinitely many times.
    We also fix an $\epsilon_u$-recursive enumeration $\epsilon_\mathsf{term}$ of $\calL^u_{\{\tilde{n}\mid n\in\bbN\}}$-terms such that the $m$-th term in the enumeration only uses ordinal variables from $\{\tilde{n}\mid n\le m\}$.
    We also define the leading formula as in the proof of \autoref{Theorem: Preproof property of LKbeta}. However, the one-step construction must be different from that in the proof of \autoref{Theorem: Preproof property of LKbeta} since we have the corresponding rules for atomic formulas.
    Instead, we introduce new formulas by Cut in every stage of the construction, except for the Axiom cases. We also define the stage number $\varsigma$, but we will increase the stage number in every one-step construction since we will always perform Cut in each one-step construction.
    Hereby, we divide the construction into two parts: 
    \begin{enumerate}
        \item If the sequent is axiomatic. In this case, we paste the proof of the leading formula we constructed in \autoref{Subsection: Embedding KP into uRS}. Also, if the sequent follows from one of the Axiom rules, terminate the proof. 

        Now we consider the remaining cases. We divide the cases by the leading formula.

        \item Suppose that the leading formula of $\Gamma$ is $A\land B$. The following portion is
        \begin{equation*}
        \begin{mathprooftree}
            \AxiomC{$[\ttX]\vdash \Gamma, A$}
            \AxiomC{$[\ttX]\vdash \Gamma, B$}
            \RightLabel{($\land$)}
            \BinaryInfC{$[\ttX]\vdash \Gamma$}
        \end{mathprooftree}
        \end{equation*}

        \item Suppose that the leading formula of $\Gamma$ is  $A\lor B$.  The following portion is
        \begin{equation*}
        \begin{mathprooftree}
            \def\fCenter{\vdash}
            \AxiomC{$[\ttX]\fCenter\Gamma, A,B$}
            \RightLabel{($\lor$)}
            \UnaryInfC{$[\ttX]\fCenter\Gamma, B$}
            \RightLabel{($\lor$)}
            \UnaryInfC{$[\ttX]\fCenter \Gamma$}
        \end{mathprooftree}
        \end{equation*}
        
        \item Suppose that the leading formula takes the form $r\notin t$. Then the following portion is
        \begin{equation*}
        \begin{mathprooftree}
            \def\fCenter{\vdash}
            \AxiomC{$\cdots \quad [\ttX_s] \fCenter \Gamma,\ s\dotin t\to r\neq s \quad \cdots$ ($\heartsuit_t$)}
            \RightLabel{$(\notin)$}
            \UnaryInfC{$[\ttX]\fCenter \Gamma$}
        \end{mathprooftree}
        \end{equation*}
        
        \item Suppose that the leading formula takes the form $r\in t$. Let $s_0,\cdots,s_{m-1}$ be the terms such that $s_i$ is the $\le\varsigma$-th term on the list $\epsilon_\mathsf{term}$, all ordinal variables occurring in $s_i$ are in $\ttX$, and $\ttX\vDash |s_i|<|t|$. Then the following portion is
        \begin{equation*}
            \begin{mathprooftree}
            \def\fCenter{\vdash}
            \AxiomC{$[\ttX]\fCenter \Gamma,\ s_0\dotin t\land r=s_0,\cdots,\ s_{m-1}\dotin t\land r=s_{m-1}$}
            \RightLabel{$(\in)$}
            \UnaryInfC{$\vdots$}
            \noLine
            \UnaryInfC{$[\ttX]\fCenter \Gamma,\ s_{m-1}\dotin t\land r=s_{m-1}$}
            \RightLabel{$(\in)$}
            \UnaryInfC{$[\ttX]\fCenter \Gamma$}
        \end{mathprooftree}
        \end{equation*}
        Note that $m=0$ is possible. In this case, we do nothing other than the Cut we will explain later.

        \item Suppose that the leading formula of $\Gamma$ is $\forall x\in t B(x)$. The following portion is
        \begin{equation*}
            \begin{mathprooftree}
            \def\fCenter{\vdash}
            \AxiomC{$ \cdots \quad [\ttX_s]\fCenter \Gamma,\ s\dotin t\to B(s) \quad \cdots$ ($\heartsuit_t$)}
            \RightLabel{$(\rmb\forall)$}
            \UnaryInfC{$[\ttX]\fCenter \Gamma$}
        \end{mathprooftree}
        \end{equation*}
        
        \item Suppose that the leading formula of $\Gamma$ is $\exists x\in t B(x)$.
        Let $s_0,\cdots,s_{m-1}$ be the terms such that $s_i$ is the $\le\varsigma$-th term on the list $\epsilon_\mathsf{term}$, all ordinal variables occurring in $s_i$ are in $\ttX$, and $\ttX\vDash |s_i|<|t|$. Then the following portion is
        \begin{equation*}
        \begin{mathprooftree}
            \def\fCenter{\vdash}
            \AxiomC{$[\ttX]\fCenter \Gamma,s_0\dotin t\land B(s_0),\cdots,s_{m-1}\dotin t\land  B(s_{m-1})$}
            \RightLabel{($\rmb\exists$)}
            \UnaryInfC{$\vdots$}
            \noLine
            \UnaryInfC{$[\ttX]\fCenter \Gamma, s_{m-1}\dotin t\land B(s_{m-1})$}
            \RightLabel{($\rmb\exists$)}
            \UnaryInfC{$[\ttX]\fCenter \Gamma$}
        \end{mathprooftree}
        \end{equation*}

        \item Suppose that the leading formula of $\Gamma$ is $\forall x B(x)$. The following portion is
        \begin{equation*}
            \begin{mathprooftree}
            \def\fCenter{\vdash}
            \AxiomC{$ \cdots \quad [\ttX_s]\fCenter \Gamma,\ B(s) \quad \cdots$ ($\heartsuit$)}
            \RightLabel{$(\forall)$}
            \UnaryInfC{$[\ttX]\fCenter \Gamma$}
        \end{mathprooftree}
        \end{equation*}
        
        \item Suppose that the leading formula of $\Gamma$ is $\exists x B(x)$.
        Let $s_0,\cdots,s_{m-1}$ be the terms such that $s_i$ is the $\le\varsigma$-th term on the list $\epsilon_\mathsf{term}$ and all ordinal variables occurring in $s_i$ are in $\ttX$. Then the following portion is
        \begin{equation*}
        \begin{mathprooftree}
            \def\fCenter{\vdash}
            \AxiomC{$[\ttX]\fCenter \Gamma, B(s_0),\cdots, B(s_{m-1})$}
            \RightLabel{($\exists$)}
            \UnaryInfC{$\vdots$}
            \noLine
            \UnaryInfC{$[\ttX]\fCenter \Gamma, B(s_{m-1})$}
            \RightLabel{($\exists$)}
            \UnaryInfC{$[\ttX]\fCenter \Gamma$}
        \end{mathprooftree}
        \end{equation*}
    \end{enumerate}
    In every stage of the construction except for the Axiom case, we perform Cut to the topmost hypotheses as long as every ordinal variable occurring in the $\varsigma$th formula $A$ of $\epsilon$ occurs in $\ttX$. For example, when the leading formula is $r\notin t$ and every ordinal variable of $A$ occurs in $\ttX$, then the actual one-step construction is
    \begin{equation*}
    \begin{mathprooftree}
        \def\fCenter{\vdash}
        \AxiomC{$[\ttX_s] \fCenter \Gamma,\ s\dotin t\to r\neq s, A$}
        \AxiomC{$[\ttX_s] \fCenter \Gamma,\ s\dotin t\to r\neq s, \lnot A$}
        \RightLabel{(Cut)}
        \BinaryInfC{$\cdots \quad [\ttX_s] \fCenter \Gamma,\ s\dotin t\to r\neq s \quad \cdots$ ($\heartsuit_t$)}
        \RightLabel{$(\notin)$}
        \UnaryInfC{$[\ttX]\fCenter \Gamma$}
    \end{mathprooftree}
    \end{equation*}

    Now suppose that $\pi$ is not a $\rmbeta$-proof, and we have an infinite branch
    \begin{equation*}
        \bigl\langle [\ttX_n] \vdash \Gamma_n \bigm| n\in\bbN \bigr\rangle
    \end{equation*}
    with $\ttX_0=\ttX$, $\Gamma_0 = \Gamma$, $\ttX_n\subseteq\ttX_{n+1}$, and $\bigcup_n \ttX_n$ is a well-order. If $\Gamma_\omega = \bigcup_n \Gamma_n$, then we can see the following:
    \begin{lemma}
    \begin{enumerate}
        \item For each sentence $A$, precisely one of $A$ or $\lnot A$ occurs in $\Gamma_\omega$
        \item If $A$ is axiomatic, then $\lnot A$ occurs in $\Gamma_\omega$.
        \item If $B\land C$ occurs in $\Gamma_\omega$, then one of $B$ or $C$ occurs in $\Gamma_\omega$.
        \item If $B\lor C$ occurs in $\Gamma_\omega$, then both $B$ and $C$ occurs in $\Gamma_\omega$.
        \item If $r\in t$ occurs in $\Gamma_\omega$, then for every set term $s$ with $\ttX\vDash |s|<|t|$, $s\dotin t$ and $r=s$ occurs in $\Gamma_\omega$. 
        \item If $r\notin t$ occurs in $\Gamma_\omega$, then there is a set term $s$ such that $\ttX\vDash |s|<|t|$ and $s\dotin t\to r\neq s$ occurs in $\Gamma_\omega$.
        \item If $\exists x\in t B(x)$ occurs in $\Gamma_\omega$, then for every set term $s$ with $\ttX\vDash |s|<|t|$, $s\dotin t\land B(s)$ occurs in $\Gamma_\omega$.
        \item  If $\forall x\in t B(x)$ occurs in $\Gamma_\omega$, then for some set term $s$ with $\ttX\vDash |s|<|t|$, $s\dotin t\to B(s)$ occurs in $\Gamma_\omega$.
        
        \item If $\exists x B(x)$ occurs in $\Gamma_\omega$, then for every set term $s$, $B(s)$ occurs in $\Gamma_\omega$.
        \item  If $\forall x B(x)$ occurs in $\Gamma_\omega$, then for some set term $s$, $B(s)$ occurs in $\Gamma_\omega$.
    \end{enumerate}
    \end{lemma}
    \begin{proof}
        The only non-trivial cases are the cases for $r\in t$ and $\exists x\in t B(x)$. We only consider the case for $\exists x\in tB(x)$ since the other case is analogous.

        Suppose that $\exists x\in t B(x)$ occurs in $\Gamma_n$ for some $n$.
        For a set term $s$ with $\ttX\vDash |s|<|t|$ let us choose a large $m\ge n$ such that
        \begin{itemize}
            \item Every ordinal parameter occurs in $\ttX_m$,
            \item $s$ is the $<\varsigma$th term on the list $\epsilon_\mathsf{term}$, where $\varsigma$ is the stage number of $\Gamma_m$, and
            \item The $\varsigma$th formula of $\epsilon$ is $\exists x\in t B(x)$.
        \end{itemize}
        By the construction, $\Gamma_{m'}$ for some $m'>m$ contains $s\dotin t\land B(s)$.
    \end{proof}
    Now, let us define $M$ to be the term model with the quotient given by
    \begin{equation*}
        s \sim t \iff \text{$s\neq t$ occurs in $\Gamma_\omega$},
    \end{equation*}
    and $M \vDash A$ iff $\lnot A$ occurs in $\Gamma_\omega$.
    Then $M$ is well-founded: For each term $r$, let us define
    \begin{equation*}
        \|r\| = 
        \begin{cases}
            \rank s, & \text{if } s\in u,\\
            \rank u + \min \{|s| : r\sim s\}, & \text{otherwise.}
        \end{cases}
    \end{equation*}
    
    Now suppose that $M\vDash r\in t$, so $r\notin t$ occurs in $\Gamma_\omega$. By replacing $t$ for $t\notin u$ if necessary, we may assume that $t$ satisfies $|t|=\|t\|$. 
    Hence, there is a set term $s$ such that $M\vDash r=s$ and $\ttX\vDash |s|<|t|$. This shows $\|r\|\le |s| < \|t\|$. Since $\ttX$ is well-ordered, $M$ cannot have an infinite $\in$-descending chain.
\end{proof}

\section{A functorial proof system for \texorpdfstring{$\KP$}{KP}, Part II}

In this section, we introduce a dilator-controlled infinitary proof system $\RS$ and establish its properties. 
First, we introduce the rank of formulas in the language $\calL^u_{\kappa}$. We only consider the very specific choice of $\kappa$, namely, 
\begin{equation} \label{Formula: Which kappa is chosen to define rk}
    \kappa = H(\alpha):= E \bigl(F(\alpha)+ \{\ttOmega\} + G(\alpha)\bigr)
\end{equation}
where $\alpha$ is a linear order, $F,G$ are dilators, and $E$ is either the epsilon flower or the Veblen flower $\varphi_\lambda$ for a non-empty well-order $\lambda$. The specific choice of $E$ is important since we need functorial internal addition, multiplication, and internal constants $1$ and $\omega$ to define the rank properly.

\begin{definition}
    Let $\kappa$ be as given in \eqref{Formula: Which kappa is chosen to define rk}.
    We define $\rk^{\bfL(u)}_\kappa$ as a function from the set of $\calL^u_{\kappa}$-formulas to $\kappa = H(\alpha)$ as follows: 
    \begin{enumerate}
        \item $\rk^{\bfL(u)}_\kappa(v)=0$ for $v\in u$.
        \item $\rk^{\bfL(u)}_\kappa(\ttL(u)_\xi) = \rk^{\bfL(u)}_\kappa\bigl(\{x\in\ttL(u)_\xi\mid \phi(x,a_0,\cdots,a_{n-1})\}\bigr) = \omega\cdot(1+\xi)$.
        \item $\rk^{\bfL(u)}_\kappa(a\in b) = \rk^{\bfL(u)}_\kappa(a\notin b) = \max(\rk^{\bfL(u)}_\kappa(a)+6, \rk^{\bfL(u)}_\kappa(b)+1)$.
        \item $\rk^{\bfL(u)}_\kappa(a=b) = \rk^{\bfL(u)}_\kappa(a\neq b) = \max(\rk^{\bfL(u)}_\kappa(a), \rk^{\bfL(u)}_\kappa(b),5)+4$.
        \item $\rk^{\bfL(u)}_\kappa(A\land B) = \rk^{\bfL(u)}_\kappa(A\lor B) = \max(\rk^{\bfL(u)}_\kappa(A), \rk^{\bfL(u)}_\kappa(B))+1$.
        \item $\rk^{\bfL(u)}_\kappa(\exists x\in a A(x))=\rk^{\bfL(u)}_\kappa(\forall x\in a A(x)) = \max(\rk^{\bfL(u)}_\kappa(a),\rk^{\bfL(u)}_\kappa(A(0))+2)$.
        \item $\rk^{\bfL(u)}_\kappa(\exists x A(x))=\rk^{\bfL(u)}_\kappa(\forall x A(x)) = \max(\ttOmega,\rk^{\bfL(u)}_\kappa(A(0))+1)$.
    \end{enumerate}
\end{definition}

\begin{proposition} \label{Proposition: Functoriality of rank}
    $\rk^{\bfL(u)}$ is functorial in the following sense: If $f\colon \alpha\to\beta$ is increasing, then the following diagram commutes: (See \autoref{Definition: Functorial boldface L} for the definition of $\bfL(u)_{H(f)}$.)
    \begin{equation*}
    \begin{tikzcd}[column sep=large]
        \calL^u_{H(\alpha)} &  \calL^u_{H(\beta)} \\
        H(\alpha) & H(\beta) %
        \arrow[from=1-1, to=1-2, "\bfL(u)_{H(f)}"]
        \arrow[from=1-1, to=2-1, "\rk^{\bfL(u)}_{H(\alpha)}"']
        \arrow[from=1-2, to=2-2, "\rk^{\bfL(u)}_{H(\beta)}"]
        \arrow[from=2-1, to=2-2, "H(f)"']
    \end{tikzcd}
    \end{equation*}
\end{proposition}
\begin{proof}
    We can easily prove that $\rk^{\bfL(u)}_{H(\beta)}\circ \bfL(u)_{H(f)} (s) = H(f)\circ \rk^{\bfL(u)}_{H(\alpha)}(s)$ for every $\calL^u_{H(\alpha)}$-term $s$.
    Then we can prove the equality  $\rk^{\bfL(u)}_{H(\beta)}\circ \bfL(u)_{H(f)} (A) = H(f)\circ \rk^{\bfL(u)}_{H(\alpha)}(A)$ for every $A\in \calL^u_{H(\alpha)}$ by induction on $A$. 
\end{proof}

Sequents in the proof system we consider in this section take the form
\begin{equation*}
    H[\ttX]\sststile{\rho}{\alpha}\Gamma.
\end{equation*}
Here $\alpha,\rho\in H(\ttX)$; We interpret its proof as an $H$-proof bounded by elements of $H$. In every case we consider, $H$ will take the form $E(F+\{\ttOmega\}+G)$ for some dilators $F,G$, and $E$ is either the epsilon flower or the Veblen flower $\varphi_\lambda$ for some non-empty well-order $\lambda$. In later parts of the paper, we write $\{\ttOmega\}$ in the definition of $H$ by $1$, so we denote $H = E\circ (F+1+G)$.

Before defining the rules, we introduce a notation for the parameters required by a set term.
For a linear order $Z$ and a term $s\in\bfL(u)_{H(Z)}$, define the \emph{parameter support} of $s$ by
\begin{equation*}
    \parsupp^H_Z(s)
    =
    \bigcup\bigl\{
        \supp^H_Z(a)
        \bigm|
        a\in\supp^{\bfL(u)}_{H(Z)}(s)
    \bigr\}.
\end{equation*}

\begin{definition}
    Let us define the deduction rules of $\RS_H$ as follows:
    \begin{center}
    \vspace{1em}
    \begin{mathprooftree}
        \def\fCenter#1#2{\sststile{#1}{#2}}
        \AxiomC{}
        \RightLabel{$(\mathrm{Ax}) \quad (x\in y\in u)$}
        \UnaryInfC{$H[\ttX]\fCenter{\rho}{\alpha} \Gamma,\ x\in y$}
    \end{mathprooftree}
    \hspace{3em}%
    \begin{mathprooftree}
        \def\fCenter#1#2{\sststile{#1}{#2}}
        \AxiomC{}
        \RightLabel{$(\mathrm{Ax}) \quad (x\notin y\in u,\ y\in u)$}
        \UnaryInfC{$H[\ttX]\fCenter{\rho}{\alpha} \Gamma,\ x\notin y$}
    \end{mathprooftree}

    \vspace{1em}
    \begin{mathprooftree}
        \def\fCenter#1#2{\sststile{#1}{#2}}
        \AxiomC{$H[\ttX]\fCenter{\rho}{\alpha_0} \Gamma,\ A$}
        \AxiomC{$H[\ttX]\fCenter{\rho}{\alpha_1}  \Gamma,\ B$}
        \RightLabel{$(\land)$ \hspace{2em} ($\alpha_0,\alpha_1<\alpha$)\footnotemark}
        \BinaryInfC{$H[\ttX]\fCenter{\rho}{\alpha} \Gamma,\ A\land B$}
    \end{mathprooftree}
    \footnotetext{We do the comparison over $H(\ttX)$.}
    
    \vspace{1em}
    \begin{mathprooftree}
        \def\fCenter#1#2{\sststile{#1}{#2}}
        \AxiomC{$H[\ttX]\fCenter{\rho}{\alpha_0} \Gamma,\ A$}
        \RightLabel{$(\lor)$}
        \UnaryInfC{$H[\ttX]\fCenter{\rho}{\alpha} \Gamma,\ A\lor B$}
    \end{mathprooftree}
    \hspace{3em}%
    \begin{mathprooftree}
        \def\fCenter#1#2{\sststile{#1}{#2}}
        \AxiomC{$H[\ttX]\fCenter{\rho}{\alpha_0} \Gamma,\ B$}
        \RightLabel{$(\lor)$ \hspace{2em} ($\alpha_0<\alpha$)}
        \UnaryInfC{$H[\ttX]\fCenter{\rho}{\alpha} \Gamma,\ A\lor B$}
    \end{mathprooftree}

    \vspace{1em}
    \begin{mathprooftree}
        \def\fCenter#1#2{\sststile{#1}{#2}}
        \AxiomC{$H[\ttX]\fCenter{\rho}{\alpha_0} \Gamma,\ s\dotin t\land r=s$}
        \RightLabel{$(\in)$ \hspace{2em} ($\alpha_0<\alpha$, $|s|<|t|$, $\rk^{\bfL(u)}_{H(\ttX)}(s)<\alpha$, $r\in t$ not basic.)}
        \UnaryInfC{$H[\ttX]\fCenter{\rho}{\alpha} \Gamma,\ r\in t$}
    \end{mathprooftree}
    
    \vspace{1em}
    \begin{mathprooftree}
        \def\fCenter#1#2{\sststile{#1}{#2}}
        \AxiomC{$\cdots \quad H[\ttX_s] \fCenter{\rho}{\alpha_s} \Gamma,\ s\dotin t\to r\neq s \quad \cdots$ ($\spadesuit_t$)}
        \RightLabel{$(\notin)$ \hspace{2em}($\alpha_s<\alpha$)}
        \UnaryInfC{$H[\ttX]\fCenter{\rho}{\alpha} \Gamma,\ r\notin t$}
    \end{mathprooftree}

    \vspace{1em}
    \begin{mathprooftree}
        \def\fCenter#1#2{\sststile{#1}{#2}}
        \AxiomC{$H[\ttX]\fCenter{\rho}{\alpha_0} \Gamma,\ s\dotin t \land B(s)$}
        \RightLabel{$(\rmb\exists)$ \hspace{2em} ($\alpha_0<\alpha$, $|s|<|t|$, $\rk^{\bfL(u)}_{H(\ttX)}(s)<\alpha$.)}
        \UnaryInfC{$H[\ttX]\fCenter{\rho}{\alpha} \Gamma,\ \exists x\in t B(x)$}
    \end{mathprooftree}
    
    \vspace{1em}
    \begin{mathprooftree}
        \def\fCenter#1#2{\sststile{#1}{#2}}
        \AxiomC{$ \cdots \quad H[\ttX_s]\fCenter{\rho}{\alpha_s} \Gamma,\ s\dotin t\to B(s) \quad \cdots$ ($\spadesuit_t$)}
        \RightLabel{$(\rmb\forall)$ \hspace{2em} ($\alpha_s<\alpha$)}
        \UnaryInfC{$H[\ttX]\fCenter{\rho}{\alpha} \Gamma,\ \forall x\in t B(x)$}
    \end{mathprooftree}
    
    \vspace{1em}
    \begin{mathprooftree}
        \def\fCenter#1#2{\sststile{#1}{#2}}
        \AxiomC{$H[\ttX]\fCenter{\rho}{\alpha_0} \Gamma,\  B(s)$}
        \RightLabel{$(\exists)$ \hspace{2em} ($\alpha_0<\alpha$,\  $\rk^{\bfL(u)}_{H(\ttX)}(s)<\alpha$.)}
        \UnaryInfC{$H[\ttX]\fCenter{\rho}{\alpha} \Gamma,\ \exists x B(x)$}
    \end{mathprooftree}
    
    \vspace{1em}
    \begin{mathprooftree}
        \def\fCenter#1#2{\sststile{#1}{#2}}
        \AxiomC{$ \cdots \quad H[\ttX_s]\fCenter{\rho}{\alpha_s} \Gamma,\ B(s) \quad \cdots$ ($\spadesuit$)}
        \RightLabel{$(\forall)$ \hspace{2em} ($\alpha_s<\alpha$)}
        \UnaryInfC{$H[\ttX]\fCenter{\rho}{\alpha} \Gamma,\ \forall x B(x)$}
    \end{mathprooftree}
    
    \vspace{1em}
    \begin{mathprooftree}
        \def\fCenter#1#2{\sststile{#1}{#2}}
        \AxiomC{$H[\ttX]\fCenter{\rho}{\alpha_0} \Gamma,\ A$}
        \RightLabel{(Refl) \hspace{2em} ($\alpha_0,\ttOmega<\alpha$, $A$ is a $\Sigma$-formula)}
        \UnaryInfC{$H[\ttX]\fCenter{\rho}{\alpha} \Gamma,\ \exists z A^z$}
    \end{mathprooftree}
    
    \vspace{1em}
    \begin{mathprooftree}
        \def\fCenter#1#2{\sststile{#1}{#2}}
        \AxiomC{$H[\ttX]\fCenter{\rho}{\alpha_0} \Gamma,\ A$}
        \AxiomC{$H[\ttX]\fCenter{\rho}{\alpha_1}  \Gamma,\ \lnot A$}
        \RightLabel{(Cut) \hspace{2em} ($\alpha_0,\alpha_1<\alpha$, $\rk^{\bfL(u)}_{H(\ttX)}(A)<\rho$)}
        \BinaryInfC{$H[\ttX]\fCenter{\rho}{\alpha} \Gamma$,}
    \end{mathprooftree}
    \end{center}
    Here $(\spadesuit)$ indicates that the displayed premises are indexed by all pairs $(\ttX_s,s)$ satisfying the following conditions: 
    \begin{itemize}
        \item $\ttX_s=\ttX\cup\parsupp^H_{\ttX_s}(s)$.
        \item Over $\ttX_s\setminus\ttX$, the $\ttX_s$-order agrees with the reverse of the natural number order.
    \end{itemize}
    For each such pair $(\ttX_s,s)$, the rule has one premise of the displayed form, with a corresponding height parameter $\alpha_s<\alpha$. ($\spadesuit_t$) addes the extra condition $H(\ttX_s)\vDash |s|<|t|$ to ($\spadesuit$).
\end{definition}
As in $(\heartsuit)$, these conditions ensure that the extension is determined uniquely by the original template and the numbers of new
variables inserted at each position.
The additional condition in $(\spadesuit_t)$ restricts $s$ to the range quantified over by the bounded formula.

$H$ in $\RS_H$ is called the \emph{domain dilator}; In Girard's terminology, $H$ indicates we only consider $H$-proofs. In semantic terms, $H \vdash \Gamma$ would imply $L(u)_{H(\alpha)}\vDash \bigvee \Gamma$ for every $\alpha$. The template order $\ttX$ allows us to use non-nullary terms of dilators in the proof tree and rank. The presence of the template order in the proof tree makes dilator-based analysis closer to relativized ordinal analysis (e.g., \cite{RathjenCook2016ClassifyingProvTotalsetftnKPKPP, Fernandez2024SetTotalRecKPell}). $H$ will also take the role of operators in operator-controlled ordinal analysis.

Regarding $(\rmb\exists)$ and $(\in)$, they do not `reduce' the size of the template order. That is, the premise template order is no larger than the conclusion template order. It looks too strong, but it is more like a feature. No preproof we will see requires `shrinking a template order from premise to conclusion.'

The parameter support behaves naturally under embeddings and admits a simple description for a composite predilator.
We record these two properties for later use in the substitution argument.
\begin{proposition} \label{Proposition: Parameter support naturality and composition}
    Let $g\colon Z\to W$ be increasing and let $s\in\bfL(u)_{H(Z)}$.
    Then
    \begin{equation*}
        \parsupp^H_W\bigl(\bfL(u)_{H(g)}(s)\bigr)
        =
        g\bigl[\parsupp^H_Z(s)\bigr].
    \end{equation*}
    Moreover, for every predilator $D$ and every $s\in\bfL(u)_{H(D(Z))}$,
    \begin{equation*} \textstyle
        \parsupp^{H\circ D}_Z(s)
        =
        \bigcup\bigl\{
            \supp^D_Z(b)
            \bigm|
            b\in\parsupp^H_{D(Z)}(s)
        \bigr\}.
    \end{equation*}
\end{proposition}
\begin{proof}
    For the first assertion, the naturality of the support transformation of $\bfL(u)$ gives
    \begin{equation*}
        \supp^{\bfL(u)}_{H(W)}
        \bigl(\bfL(u)_{H(g)}(s)\bigr)
        =
        H(g)\bigl[
            \supp^{\bfL(u)}_{H(Z)}(s)
        \bigr].
    \end{equation*}
    Hence, using the naturality of the support transformation of $H$,
    \begin{align*}
        \parsupp^H_W\bigl(\bfL(u)_{H(g)}(s)\bigr)
        &= \textstyle
        \bigcup\bigl\{
            \supp^H_W\bigl(H(g)(a)\bigr)
            \bigm|
            a\in\supp^{\bfL(u)}_{H(Z)}(s)
        \bigr\} \\
        &= \textstyle
        \bigcup\bigl\{
            g\bigl[\supp^H_Z(a)\bigr]
            \bigm|
            a\in\supp^{\bfL(u)}_{H(Z)}(s)
        \bigr\} \\
        &= \textstyle
        g\biggl[
            \bigcup\bigl\{
                \supp^H_Z(a)
                \bigm|
                a\in\supp^{\bfL(u)}_{H(Z)}(s)
            \bigr\}
        \biggr] \\
        &=
        g\bigl[\parsupp^H_Z(s)\bigr].
    \end{align*}

    For the second assertion, recall that the support transformation of the composite predilator $H\circ D$ is given by
    \begin{equation*} \textstyle
        \supp^{H\circ D}_Z(a)
        =
        \bigcup\bigl\{
            \supp^D_Z(b)
            \bigm|
            b\in\supp^H_{D(Z)}(a)
        \bigr\}
    \end{equation*}
    for $a\in H(D(Z))$.
    Therefore,
    \begin{align*}
        \parsupp^{H\circ D}_Z(s)
        &= \textstyle
        \bigcup\bigl\{
            \supp^{H\circ D}_Z(a)
            \bigm|
            a\in\supp^{\bfL(u)}_{H(D(Z))}(s)
        \bigr\} \\
        &= \textstyle
        \bigcup\biggl\{
            \bigcup\bigl\{
                \supp^D_Z(b)
                \bigm|
                b\in\supp^H_{D(Z)}(a)
            \bigr\}
            \biggm|
            a\in\supp^{\bfL(u)}_{H(D(Z))}(s)
        \biggr\} \\
        &= \textstyle
        \bigcup\bigl\{
            \supp^D_Z(b)
            \bigm|
            b\in\parsupp^H_{D(Z)}(s)
        \bigr\}. \qedhere 
    \end{align*}
\end{proof}

The proof of the following lemma illustrates how the functoriality of the proof system $\RS_H$ works and will have a focal role in the impredicative collapsing. The following proof uses coinduction on proof trees in place of induction on the rank of the last sequent since the `rank' is no longer a member of a well-order. A similar idea also appears in \cite{Towsner2025proofsmodifyproofs12}, where a new proof tree is constructed from an old proof tree corecursively. We may think of the proof of the following lemma as constructing an algorithm mapping a proof tree to another proof tree. 
\begin{lemma}[Substitution lemma] \label{Lemma: Substitution lemma}
    Let $H=E(F+1+G)$, and let $D$ be a predilator.
    Suppose that $\pi$ is a preproof of $H[\ttX]\sststile{\rho}{\alpha}\Gamma$ and that $f\colon\ttX\to D(\ttY)$ is increasing.
    Then there is a preproof $\pi[f]$ of
    \begin{equation*}
        H\circ D[\ttY]\sststile{H(f)(\rho)}{H(f)(\alpha)}
        \bfL(u)_{H(f)}[\Gamma].
    \end{equation*}
\end{lemma}
\begin{proof}
    We construct $\pi[f]$ by coinduction on $\pi$, with a case distinction according to the last inference of $\pi$.
    As before, we abbreviate $\bfL(u)_{H(g)}(s)$ and $\bfL(u)_{H(g)}(A)$ by $H(g)(s)$ and $H(g)(A)$, respectively.

    All cases except $(\notin)$, $(\rmb\forall)$, $(\forall)$ are immediate from the coinductive hypothesis and functoriality.
    In particular, inequalities between height bounds are preserved because the relevant maps are increasing, and the side conditions involving ranks are preserved by \autoref{Proposition: Functoriality of rank}.
    The cases $(\notin)$, $(\rmb\forall)$, $(\forall)$ are all analogous, so we only consider $(\rmb\forall)$.

    Suppose that the last inference of $\pi$ is
    \begin{center}
    \begin{mathprooftree}
        \def\fCenter#1#2{\sststile{#1}{#2}}
        \noLine
        \AxiomC{$\pi_s$}
        \noLine
        \UnaryInfC{$\vdots$}
        \noLine
        \UnaryInfC{$\cdots \quad H[\ttX_s]\fCenter{\rho}{\alpha_s}\Gamma,\ s\dotin t\to B(s) \quad \cdots$ $(\spadesuit_t)$}
        \RightLabel{$(\rmb\forall)$ \hspace{2em} ($\alpha_s<\alpha$)}
        \UnaryInfC{$H[\ttX]\fCenter{\rho}{\alpha}\Gamma,\ \forall x\in t B(x)$}
    \end{mathprooftree}
    \end{center}
    We must construct an inference
    \begin{center}
    \begin{mathprooftree}
        \def\fCenter#1#2{\sststile{#1}{#2}}
        \noLine
        \AxiomC{$\varpi_r$}
        \noLine
        \UnaryInfC{$\vdots$}
        \noLine
        \UnaryInfC{$\cdots \quad H\circ D[\ttY_r]\fCenter{H(f)(\rho)}{\beta_r}H(f)[\Gamma],\ r\dotin H(f)(t)\to H(f)(B)(r) \quad \cdots$ $(\clubsuit_{H(f)(t)})$}
        \RightLabel{$(\rmb\forall)$ \hspace{2em} ($\beta_r<H(f)(\alpha)$)}
        \UnaryInfC{$H\circ D[\ttY]\fCenter{H(f)(\rho)}{H(f)(\alpha)}H(f)[\Gamma],\ \forall x\in H(f)(t)H(f)(B)(x)$}
    \end{mathprooftree}
    \end{center}
    where $(\clubsuit_{H(f)(t)}))$ ranges over all pairs $(\ttY_r,r)$ such that
    \begin{itemize}
        \item $\ttY_r=\ttY\cup\parsupp^{H\circ D}_{\ttY_r}(r),$.
        \item The $\ttY_r$-order and the reverse natural number order agree on $\ttY_r\setminus\ttY$.
        \item $H\circ D(\ttY_r)\vDash |r|<H(f)(|t|)$.
    \end{itemize}
    Whenever $\ttY\subseteq\ttY_r$, we identify $f$ with its composite with the canonical embedding $D(\ttY)\to D(\ttY_r)$.

    Fix a pair $(\ttY_r,r)$ satisfying $(\clubsuit_{H(f)(t)}))$.
    Choose the unique template order $\ttZ_r\supseteq\ttX$ and the unique order isomorphism $f_r\colon\ttZ_r\to D(\ttY_r)$ such that
    \begin{itemize}
        \item $\ran f_r = \ran f\cup\parsupp^H_{D(\ttY_r)}(r)$,  
        \item $f_r\restriction\ttX=f$, and
        \item $\ttZ_r$-order and the reverse natural number order agree on $\ttZ_r\setminus\ttX$.
    \end{itemize}
    
    We next pull $r$ back along $H(f_r)$.
    If $a\in\supp^{\bfL(u)}_{H(D(\ttY_r))}(r)$, then
    \begin{equation*}
        \supp^H_{D(\ttY_r)}(a)
        \subseteq
        \parsupp^H_{D(\ttY_r)}(r)
        \subseteq
        \ran f_r.
    \end{equation*}
    By the support conditions for $H$ and $\bfL(u)$, there is therefore a unique term $s(r)\in\bfL(u)_{H(\ttZ_r)}$ such that $H(f_r)(s(r))=r$.
    We claim that the pair $(\ttZ_r,s(r))$ satisfies $(\spadesuit_t)$.

    For the first condition, \autoref{Proposition: Parameter support naturality and composition} gives
    \begin{equation*}
        f_r\bigl[\parsupp^H_{\ttZ_r}(s(r))\bigr]
        =
        \parsupp^H_{D(\ttY_r)}(r).
    \end{equation*}
    Since $f_r$ is injective,
    \begin{align*}
        \ttX\cup\parsupp^H_{\ttZ_r}(s(r))
        &=
        f_r^{-1}[\ran f]\cup
        f_r^{-1}\bigl[\parsupp^H_{D(\ttY_r)}(r)\bigr] \\
        &=
        f_r^{-1}\bigl[\ran f\cup\parsupp^H_{D(\ttY_r)}(r)\bigr] =
        \ttZ_r.
    \end{align*}
    The second condition in $(\spadesuit_t)$ holds by the choice of $\ttZ_r$.

    For the third condition, functoriality of the support of $\bfL(u)$ gives $H(f_r)(|s(r)|)=|r|$.
    Since $f_r$ extends $f$ and $t$ only uses parameters from $\ttX$, we also have $H(f_r)(|t|)=H(f)(|t|)$.
    Hence
    \begin{equation*}
        H(f_r)(|s(r)|)
        =
        |r|
        <
        H(f)(|t|)
        =
        H(f_r)(|t|).
    \end{equation*}
    Since $H(f_r)$ is increasing, $|s(r)|<|t|$.
    Thus $(\ttZ_r,s(r))$ satisfies $(\spadesuit_t)$.
    We may therefore write $\ttX_{s(r)}=\ttZ_r$, and $\pi_{s(r)}$ is one of the premises of the original inference.

    Applying the coinductive construction to $\pi_{s(r)}$ and $f_r$ yields a preproof $\pi_{s(r)}[f_r]$ of
    \begin{equation*}
        H\circ D[\ttY_r]
        \sststile{H(f_r)(\rho)}{H(f_r)(\alpha_{s(r)})}
        H(f_r)[\Gamma],\
        r\dotin H(f)(t)\to H(f)(B)(r).
    \end{equation*}
    Since $\rho$, $\Gamma$, $t$, and $B$ only use parameters from $\ttX$ and $f_r\restriction\ttX=f$, we have $H(f_r)(\rho)=H(f)(\rho)$ and $ H(f_r)[\Gamma]=H(f)[\Gamma]$.
    Moreover,
    \begin{equation*}
        H(f_r)(\alpha_{s(r)})
        <
        H(f_r)(\alpha)
        =
        H(f)(\alpha).
    \end{equation*}
    Thus we can take $\varpi_r=\pi_{s(r)}[f_r]$ and $\beta_r=H(f_r)(\alpha_{s(r)})$.
    Since $(\ttY_r,r)$ was an arbitrary index satisfying $(\clubsuit_{H(f)(t)})$, these preproofs supply all premises required for the target $(\rmb\forall)$-inference.

    The case $(\notin)$ is identical, with $s\dotin t\to B(s)$ replaced by $s\dotin t\to a\neq s$.
\end{proof}

\begin{remark}
    \label{Remark: Parameter control in substitution}
    The construction in the proof of \autoref{Lemma: Substitution lemma} has an additional parameter-control property.
    Consider a premise indexed by $(\ttY_r,r)$ arising in either the $(\rmb\forall)$ or $(\notin)$ case.
    As in the proof, let $f_r\colon\ttX_{s(r)}\to D(\ttY_r)$ be the corresponding extension of $f$, so that
    \begin{equation*}
        \ran f_r
        =
        \ran f\cup\parsupp^H_{D(\ttY_r)}(r).
    \end{equation*}

    Suppose additionally that $F$ is a flower and that $t$ is not a constant from $u$.
    If $I$ is an initial segment of $D(\ttY_r)$ containing $\ran f$, then $\parsupp^H_{D(\ttY_r)}(r)\subseteq I$ and hence $\ran f_r\subseteq I$.
    Thus the substitution used for this premise introduces no arguments of $H$ outside $I$.

    To see this, put $Z=D(\ttY_r)$ and identify the part of $H(Z)$ below $\ttOmega$ with $(E\circ F)(Z)$.
    Since $\ran f\subseteq I$, the map $f$ factors through $I$, and hence $H(f)(|t|)$ belongs to the image of $(E\circ F)(I)$ in $(E\circ F)(Z)$.
    Since $F$ and $E$ are flowers, this image is an initial segment.

    If $r$ is a constant from $u$, then $\parsupp^H_Z(r)=\varnothing$, and there is nothing to prove.
    Otherwise, for every $b\in\supp^{\bfL(u)}_{H(Z)}(r)$, we have
    \begin{equation*}
        b\le |r|<H(f)(|t|).
    \end{equation*}
    Hence $b$ belongs to the image of $(E\circ F)(I)$ in $(E\circ F)(Z)$.
    By the support condition for $H$, $\supp^H_Z(b)\subseteq I$.
    Taking the union over all such $b$ yields $\parsupp^H_Z(r)\subseteq I$.
    The conclusion $\ran f_r\subseteq I$ now follows from $\ran f_r=\ran f\cup\parsupp^H_Z(r)$.
\end{remark}

\begin{lemma}[Rank extension lemma]
    Suppose that $H=E(F+1+G)$, $H'=E(F+1+G')$, and $G\subseteq G'$. If $H[\ttX]\sststile{\rho}{\alpha} \Gamma$, then $H'[\ttX]\sststile{\rho}{\alpha} \Gamma$.
\end{lemma}
\begin{proof}
    Again, we prove it by coinduction on preproofs. Its proof is easier than that of the Substitution lemma, so we omit it.
\end{proof}

\section{Predicative results}
In this section, we analyze properties of $\RS_H$.

\begin{lemma} \label{Lemma: Rathjen-Cook Lemma 3.15} Let $H=E\circ (F+1+G)$ and $\ttX$ be a template order.
    \begin{enumerate}
        \item If $H(\ttX)\vDash \alpha\le\alpha'\land \rho\le\rho'$, $H[\ttX]\sststile{\rho}{\alpha} \Gamma$, then $H[\ttX]\sststile{\rho'}{\alpha'} \Gamma,\Delta$.
        \item If $C$ is a true basic formula and $H[\ttX]\sststile{\rho}{\alpha} \Gamma,\lnot C$, then $H[\ttX]\sststile{\rho}{\alpha} \Gamma$.
        \item If $H\sststile{\rho}{\alpha} \Gamma,A\lor B$, then $H\sststile{\rho}{\alpha} \Gamma,A, B$.
    \end{enumerate}
\end{lemma}
\begin{proof}
    Its proof is similar to that of \cite[Lemma 3.15]{RathjenCook2016ClassifyingProvTotalsetftnKPKPP}. We only provide details for the first item.
    Unlike \cite[Lemma 3.15]{RathjenCook2016ClassifyingProvTotalsetftnKPKPP} where the proof relies on induction on $\alpha$, we use coinduction on a proof $\pi$ of $H[\ttX]\sststile{\rho}{\alpha}\Gamma$.
    
    If $H[\ttX]\sststile{\rho}{\alpha}\Gamma$ is an axiom, then $H[\ttX]\sststile{\rho'}{\alpha'}\Gamma,\Delta$ is also an axiom, so we are done. Now, let us consider the case when $H[\ttX]\sststile{\rho}{\alpha}\Gamma$ is the result of a deduction. We divide the case for the last rule of inference. The argument is very straightforward, so let us examine the following case only:
    Suppose that the last rule of inference is $(\land)$, so we have
    \begin{equation*}
    \begin{mathprooftree}
        \def\fCenter#1#2{\sststile{#1}{#2}}
        \AxiomC{$\pi_0$}
        \noLine
        \UnaryInfC{$\vdots$}
        \noLine
        \UnaryInfC{$H[\ttX]\fCenter{\rho}{\alpha_0} \Gamma,\ A$}
        \AxiomC{$\pi_1$}
        \noLine
        \UnaryInfC{$\vdots$}
        \noLine
        \UnaryInfC{$H[\ttX]\fCenter{\rho}{\alpha_1}  \Gamma,\ B$}
        \RightLabel{$(\land)$}
        \BinaryInfC{$H[\ttX]\fCenter{\rho}{\alpha} \Gamma,\ A\land B$}
    \end{mathprooftree}
    \end{equation*}
    By the coinductive hypothesis, we have new proofs
    \begin{equation*}
    \begin{mathprooftree}
        \def\fCenter#1#2{\sststile{#1}{#2}}
        \AxiomC{$\pi'_0$}
        \noLine
        \UnaryInfC{$\vdots$}
        \noLine
        \UnaryInfC{$H[\ttX]\fCenter{\rho'}{\alpha_0} \Gamma,\ A,\ \Delta$}
        \AxiomC{$\pi'_1$}
        \noLine
        \UnaryInfC{$\vdots$}
        \noLine
        \UnaryInfC{$H[\ttX]\fCenter{\rho'}{\alpha_1}  \Gamma,\ B,\ \Delta$}
        \noLine
        \BinaryInfC{ }
    \end{mathprooftree}
    \end{equation*}
    Then apply $(\land)$ to form a new proof of $H[\ttX]\sststile{\rho'}{\alpha'}\Gamma,\ A\land B,\ \Delta$.
\end{proof}

\begin{lemma} \label{Lemma: Reverting conjunctive formulas}
    Let $H=E\circ (F+1+G)$ and $\ttX$ be a template order.
    \begin{enumerate}
        \item If $H[\ttX]\sststile{\rho}{\alpha} \Gamma,A\land B$, then $H[\ttX]\sststile{\rho}{\alpha} \Gamma,A$ and $H[\ttX]\sststile{\rho}{\alpha} \Gamma, B$.
        \item If $H[\ttX]\sststile{\rho}{\alpha}\Gamma, r\notin t$, then for every $\ttY\supseteq\ttX$ and $s \in \bfL(u)_{H(\ttY)}$ satisfying $(\spadesuit_t)$, then we have
        \begin{equation*}
            H[\ttY]\sststile{\rho}{\alpha} \Gamma,s\dotin t\to r\neq s.
        \end{equation*}
        \item If $H[\ttX]\sststile{\rho}{\alpha}\Gamma, \forall x\in t B(x)$, then for every $\ttY\supseteq\ttX$ and $s \in \bfL(u)_{H(\ttY)}$ satisfying $(\spadesuit_t)$, then we have
        \begin{equation*}
            H[\ttY]\sststile{\rho}{\alpha} \Gamma,s\dotin t\to B(s).
        \end{equation*}
        \item If $H[\ttX]\sststile{\rho}{\alpha}\Gamma, \forall x B(x)$, then for every $\ttY\supseteq\ttX$ and $s \in \bfL(u)_{H(\ttY)}$ satisfying $(\spadesuit)$, then we have
        \begin{equation*}
            H[\ttY]\sststile{\rho}{\alpha} \Gamma,B(s).
        \end{equation*}
    \end{enumerate}
\end{lemma}
\begin{proof}
    We prove it by coinduction on a proof tree. We only prove the third item.
    
    Suppose that we are given $H[\ttX]\sststile{\rho}{\alpha} \Gamma,\forall x\in t\ B(x)$ and $s$ and $\ttY$ with the given hypotheses. If $\forall x\in t\ B(x)$ is not the principal formula, we apply the coinductive hypothesis to the premise. 
    If $H[\ttX]\sststile{\rho}{\alpha} \Gamma,\forall x\in t\ B(x)$ is the principal formula of $(\rmb\forall)$, then we have $\alpha_s<\alpha$ such that $H[\ttX]\sststile{\rho}{\alpha_s} \Gamma,s\dotin t\to B(s)$. Then apply the first item of \autoref{Lemma: Rathjen-Cook Lemma 3.15}.
\end{proof}

\begin{lemma}[Reduction]
    Suppose that $\rk^{\bfL(u)}_{H(\ttX)}(C) \le \rho$ and $\rk^{\bfL(u)}_{H(\ttX)}(C) \neq \ttOmega$.
    Then
    \begin{equation*}
        H[\ttX]\sststile{\rho}{\alpha} \Gamma,\lnot C \text{ and } H[\ttX]\sststile{\rho}{\beta} \Delta, C \implies H[\ttX]\sststile{\rho}{\alpha+\beta} \Gamma,\Delta.
    \end{equation*}
\end{lemma}
\begin{proof}
    By symmetry, we may assume that $C$ is existential, so $C$ is either $s\in t$, $\exists x\in t B(x)$, or $\exists x B(x)$. We prove it by coinduction on the preproof of $H[\ttX]\sststile{\rho}{\beta} \Delta, C$. We only consider the latter case $C\equiv \exists x\in t B(x)$ since the remainder is simpler. 
    
    If $C$ is not the principal formula in $H[\ttX]\sststile{\rho}{\beta} \Delta, C$, then we apply the coinductive hypothesis to the premise.
    Now suppose that $C$ is a principal formula of the last inference of $H[\ttX]\sststile{\rho}{\beta} \Delta, C$.
    If $C$ is the principal formula of $\mathrm{(Refl)}$, then $C$ takes the form $\exists x B(x)$ for some $\Delta_0$-formula $B$, so $\rk^{\bfL(u)}_{H(\ttX)}(C) = \ttOmega$, contrary with the assumption $\rk^{\bfL(u)}_{H(\ttX)}(C)\neq\ttOmega$.
    Hence, the last inference must be $(\rmb\exists)$ or $(\exists)$.

    We first consider the case when the last inference is $(\rmb\exists)$, and $C$ takes the form $\exists x\in t B(x)$.
    This implies that there are $s$ with $\rk^{\bfL(u)}_{H(\ttX)}(s) < \beta$ and $\beta_0<\beta$ such that
    \begin{equation} \label{Formula: Reduction formula 00}
        H[\ttX]\sststile{\rho}{\beta_0} \Delta, s\mathrel{\dotin} t\land B(s).
    \end{equation}
    By \autoref{Lemma: Reverting conjunctive formulas} applied to $H[\ttX]\sststile{\rho}{\alpha} \Gamma,\lnot C$, we have
    \begin{equation} \label{Formula: Reduction formula 01}
        H[\ttX]\sststile{\rho}{\alpha} \Gamma,s\dotin t\to \lnot B(s).
    \end{equation}
    The rank definitions and $|s|<|t|$ give
    \begin{equation*}
        \rk^{\bfL(u)}_{H(\ttX)}(s\mathrel{\dotin}t\land B(s))
        <
        \rk^{\bfL(u)}_{H(\ttX)}(C)
        \le\rho.
    \end{equation*}
    Hence we apply $(\mathrm{Cut})$ directly to \eqref{Formula: Reduction formula 00} and \eqref{Formula: Reduction formula 01}.
    The resulting conclusion may be assigned height $\alpha+\beta$, since $\alpha<\alpha+\beta$ and $\beta_0<\beta\le\alpha+\beta$.

    Now consider the case when the last inference is $(\exists)$ and $C$ takes the form $\exists x B(x)$
    Then there are a set term $s$ and $\beta_0<\beta$ such that
    \begin{equation} \label{Formula: Reduction unbounded exists 00}
        H[\ttX]\sststile{\rho}{\beta_0} \Delta,B(s).
    \end{equation}
    By applying \autoref{Lemma: Reverting conjunctive formulas} to $H[\ttX]\sststile{\rho}{\alpha} \Gamma, \forall x\lnot B(x)$, we have
    \begin{equation} \label{Formula: Reduction unbounded exists 01}
        H[\ttX]\sststile{\rho}{\alpha} \Gamma,\lnot B(s).
    \end{equation}
    It remains to check that $B(s)$ has rank below $\rho$. Put $\delta=\rk^{\bfL(u)}_{H(\ttX)}(B(0))$.
    Since
    \begin{equation*}
        \rk^{\bfL(u)}_{H(\ttX)}(C)
        =
        \max\{\ttOmega,\delta+1\}
        \neq\ttOmega,
    \end{equation*}
    we have $\delta\ge\ttOmega$. On the other hand, since $s$ is a
    genuine set term, $\rk^{\bfL(u)}_{H(\ttX)}(s)<\ttOmega$.
    Hence, the usual substitution estimate for ranks gives
    \begin{equation*}
        \rk^{\bfL(u)}_{H(\ttX)}(B(s))
        \le\delta
        <
        \rk^{\bfL(u)}_{H(\ttX)}(C)
        \le\rho.
    \end{equation*}
    Hence we can apply Cut to \eqref{Formula: Reduction unbounded exists 00} and \eqref{Formula: Reduction unbounded exists 01}, obtaining $H[\ttX]\sststile{\rho}{\alpha+\beta}\Gamma,\Delta$.
\end{proof}

\begin{theorem}[First Cut elimination] \label{Theorem: First Cut Elimination Case 0}
    Let $m\ge 1$. If 
    $H[\ttX]\sststile{\ttOmega+m+1}{\alpha}\Gamma$, then $H[\ttX]\sststile{\ttOmega+m}{\omega^\alpha}\Gamma$.
\end{theorem}
\begin{proof}
    We proceed by coinduction on the proof tree. The only interesting case is when the last inference is the cut:
    \begin{equation*}
    \begin{mathprooftree}
        \def\fCenter#1#2{\sststile{#1}{#2}}
        \AxiomC{$H[\ttX]\fCenter{\ttOmega+m+1}{\alpha_0} \Gamma,\ A$}
        \AxiomC{$H[\ttX]\fCenter{\ttOmega+m+1}{\alpha_1}  \Gamma,\ \lnot A$}
        \RightLabel{(Cut)}
        \BinaryInfC{$H[\ttX]\fCenter{\ttOmega+m+1}{\alpha} \Gamma$}
    \end{mathprooftree}
    \end{equation*}
    with $\alpha_0,\alpha_1<\alpha$ and $\rk^{\bfL(u)}_{H(\ttX)}(A)<\ttOmega+m+1$. By applying the coinductive hypothesis to the premises, we get
    $H[\ttX]\sststile{\ttOmega+m}{\omega^{\alpha_0}} \Gamma,\ A$ and $H[\ttX]\sststile{\ttOmega+m}{\omega^{\alpha_1}} \Gamma,\ \lnot A$.
    We are done when $\rk^{\bfL(u)}_{H(\ttX)}(A)<\ttOmega+m$, so let us consider the case $\rk^{\bfL(u)}_{H(\ttX)}(A)=\ttOmega+m$. In this case, Reduction yields
    \begin{equation*}
        H[\ttX]\sststile{\ttOmega+m}{\omega^{\alpha_0} + \omega^{\alpha_1}} \Gamma, 
    \end{equation*}
    and we get the desired result since $H(\ttX)\vDash \omega^{\alpha_0} + \omega^{\alpha_1} \le \omega^\alpha$.
\end{proof}

The following cut elimination theorem will be necessary for later applications. Yet, we do not need it in the proof-theoretic analysis of $\KP$.
\begin{proposition} \label{Proposition: First Cut Elimination Case 1}
    Let $E=\varphi_2$ and $H=E\circ (F+1+G)$. If $H[\ttX]\sststile{\ttOmega+\omega}{\alpha}\Gamma$, then $H[\ttX]\sststile{\ttOmega+1}{\varphi_1(\alpha)}\Gamma$.
\end{proposition}
\begin{proof}
    The proof of \autoref{Theorem: First Cut Elimination Case 0} yields the following:
    For $m\ge 1$, $H[\ttX]\sststile{\ttOmega+m+1}{\alpha}\Gamma$ implies $H[\ttX]\sststile{\ttOmega+m}{\varphi_0(\alpha)}\Gamma$.
    Now we prove the main proposition by coinduction on a proof of $H[\ttX]\sststile{\ttOmega+\omega}{\alpha}\Gamma$.
    The only non-trivial case is when the last inference is Cut, so the proof takes the form
    \begin{equation*}
    \begin{mathprooftree}
        \def\fCenter#1#2{\sststile{#1}{#2}}
        \AxiomC{$H[\ttX]\fCenter{\ttOmega+\omega}{\alpha_0} \Gamma,\ A$}
        \AxiomC{$H[\ttX]\fCenter{\ttOmega+\omega}{\alpha_1}  \Gamma,\ \lnot A$}
        \RightLabel{(Cut)}
        \BinaryInfC{$H[\ttX]\fCenter{\ttOmega+\omega}{\alpha} \Gamma$}
    \end{mathprooftree}
    \end{equation*}
    Suppose that $\rk^{\bfL(u)}_{H(\ttX)}(A) = \ttOmega + m$ for some $m<\omega$. By the coinductive hypothesis, we have $H[\ttX]\sststile{\ttOmega+1}{\varphi_1(\alpha_0)} \Gamma,\ A$ and $H[\ttX]\sststile{\ttOmega+1}{\varphi_1(\alpha_1)} \Gamma,\ \lnot A$.
    Without loss of generality, assume that $\alpha_0=\alpha_1<\alpha$. Then we can form the following proof:
    \begin{equation*}
    \begin{mathprooftree}
        \def\fCenter#1#2{\sststile{#1}{#2}}
        \AxiomC{$H[\ttX]\fCenter{\ttOmega+m+1}{\varphi_1(\alpha_0)} \Gamma,\ A$}
        \AxiomC{$H[\ttX]\fCenter{\ttOmega+m+1}{\varphi_1(\alpha_0)}  \Gamma,\ \lnot A$}
        \RightLabel{(Cut)}
        \BinaryInfC{$H[\ttX]\fCenter{\ttOmega+m+1}{\varphi_1(\alpha_0)+1} \Gamma$}
    \end{mathprooftree}
    \end{equation*}
    This implies $H[\ttX]\sststile{\ttOmega+1}{\beta} \Gamma$, where $\beta = \underbrace{\varphi_0\circ \cdots \circ \varphi_0}_{m \text{ times}}(\varphi_1(\alpha_0)+1)$. From $\beta<\varphi_1(\alpha)$, we have $H[\ttX]\sststile{\ttOmega+1}{\varphi_1(\alpha)} \Gamma$.
\end{proof}

\begin{lemma}[Bounding lemma]
    Let $A$ be a $\Sigma$-formula and $H(\ttX)\vDash \alpha\le\beta<\ttOmega$.
    Then
    \begin{equation*}
        H[\ttX] \sststile{\rho}{\alpha} \Gamma, A \implies H[\ttX]\sststile{\rho}{\alpha} \Gamma, A^{\ttL(u)_\beta}.
    \end{equation*}
\end{lemma}
\begin{proof}
    The proof uses coinduction on a preproof, which will largely follow \cite[Lemma 3.18]{RathjenCook2016ClassifyingProvTotalsetftnKPKPP}.
    The most interesting case is when $A$ is a principal formula of the last inference. Note that $\alpha<\ttOmega$ forces that the last inference cannot be $\mathrm{(Refl)}$.
    We only consider the cases when the last inference is either $(\rmb \exists)$, $(\exists)$, or $(\rmb \forall)$.
    \begin{enumerate}
        \item Suppose that the last inference is $(\rmb\forall)$, so $A$ takes the form $\forall x\in s\ B(x)$ for some $\Sigma$-formula $B(x)$ and a term $s$ in $\bfL(u)_{H(\ttX)}$. The hypotheses of the last inference take the form
        \begin{equation*}
            H[\ttY_t] \sststile{\rho}{\alpha_t} \Gamma,\ t\mathrel{\dot\notin}s\lor B(t) 
        \end{equation*}
        for each $t$, $\ttY_t$, and $\alpha_t$ satisfying $(\spadesuit_s)$ and $\alpha_t<\alpha$. In particular, we have $H(\ttY_t)\vDash \alpha_t \le \alpha<\ttOmega$. By the coinductive hypothesis, we have 
        \begin{equation*}
            H[\ttY_t] \sststile{\rho}{\alpha_t} \Gamma,\ (t\mathrel{\dot\notin}s)^{\ttL(u)_\beta}\lor B(t)^{\ttL(u)_\beta}. 
        \end{equation*}
        The formula $t\mathrel{\dot\notin}s$ is $\Delta_0$, so $(t\mathrel{\dot\notin}s)^{\ttL(u)_\beta} \equiv (t\mathrel{\dot\notin}s)$. From the given new set of hypotheses, we can derive
        \begin{equation*}
            H[\ttX]\sststile{\rho}{\alpha} \Gamma,\ \forall t\in s B(t)^{\ttL(u)_\beta}.
        \end{equation*}
        
        \item Now suppose that the last inference is $(\rmb\exists)$, so $A$ takes the form $\exists x\in s B(x)$. The hypothesis takes the form
        \begin{equation*}
            H[\ttX]\sststile{\rho}{\alpha_0} \Gamma,\ t\dotin s\land B(t)
        \end{equation*}
        for some $\alpha\in H(\ttX)$ and a term $t\in \bfL(u)_{H(\ttX)}$ such that $\alpha_0<\alpha$, $|t|<|s|$, and $\rk^{\bfL(u)}_{H(\ttX)}(t)<\alpha$. Now, let us apply the coinductive hypothesis to get
        \begin{equation*}
            H[\ttX]\sststile{\rho}{\alpha_0} \Gamma,\ t\dotin s\land B(t)^{\ttL(u)_\beta}.
        \end{equation*}
        We can apply $(\rmb\exists)$ since $\rk^{\bfL(u)}_{H(\ttX)}(t) < \alpha$ and $|t|<|s|$. This implies 
        \begin{equation*}
            H[\ttX]\sststile{\rho}{\alpha} \Gamma,\ \exists x\in s B(x)^{\ttL(u)_\beta}.
        \end{equation*}
    \end{enumerate}
    The case for $(\exists)$ is similar to the $(\rmb\exists)$ case, so we omit the details.
\end{proof}

\begin{lemma}[Persistence Lemma]
    Let $\gamma,\delta\in H(\ttX)$.
    \begin{enumerate}
        \item If $\gamma<\delta<\ttOmega$, then
        \begin{equation*}
            H[\ttX]\sststile{\rho}{\alpha}
            \Gamma,\forall x\in\ttL(u)_\delta B(x)
            \implies
            H[\ttX]\sststile{\rho}{\alpha}
            \Gamma,\forall x\in\ttL(u)_\gamma B(x).
        \end{equation*}

        \item If $\gamma<\ttOmega$, then
        \begin{equation*}
            H[\ttX]\sststile{\rho}{\alpha}
            \Gamma,\forall x\,B(x)
            \implies
            H[\ttX]\sststile{\rho}{\alpha}
            \Gamma,\forall x\in\ttL(u)_\gamma B(x).
        \end{equation*}
    \end{enumerate}
\end{lemma}
\begin{proof}
    Both assertions are proved by coinduction on the given preproof.
    If the displayed universal formula is not principal in the last
    inference, apply the coinductive hypothesis to the premises and
    reconstruct the last inference.

    For the first assertion, suppose that the displayed formula is
    principal in a bounded $(\rmb\forall)$-inference. To construct the
    target bounded universal inference, fix a premise index
    $(\ttX_s,s)$ satisfying $H(\ttX_s)\vDash |s|<\gamma$.
    Since $\gamma<\delta$, the same pair also satisfies
    $|s|<\delta$, so the corresponding premise already occurs in the
    source inference. Keeping precisely these premises yields the
    required $(\rmb\forall)$-inference.

    For the second assertion, suppose that the displayed formula is
    principal in an unbounded $(\forall)$-inference. Again fix a
    target premise index $(\ttX_s,s)$ with
    $|s|<\gamma$. The source unbounded $(\forall)$-inference has a
    premise indexed by every pair satisfying $(\spadesuit)$, and hence
    in particular by $(\ttX_s,s)$. These premises give the required
    bounded $(\rmb\forall)$-inference.
\end{proof}

\section{Embedding} \label{Section: Embedding}
Let us observe that we can define an internal natural sum over the epsilon flower. For a formula $A$, we define $\no(A) = \omega^{\rk(A)}$. For a sequent $\Gamma = \{A_0,\cdots,A_{m-1}\}$, we define $\no(\Gamma) = \no(A_0)\#\cdots\# \no(A_{m-1})$.
We write
\begin{equation*}
    H[\ttX]\dststile{\rho}{\xi}\Gamma \iff H[\ttX] \sststile{\rho}{\no(\Gamma)\# \xi}\Gamma.
\end{equation*}
We omit $\xi$ or $\rho$ if they are zero.

The purpose of using the natural sum is twofold: First, a sequent is a multiset of formulas, and we do not want to care about the order of formulas in a sequent. Second, most existing embedding proofs use the natural sum, and we want to `reuse' the known proofs in the current context.
\begin{lemma}
    $H[\ttX]\Vdash \Gamma, A,\lnot A$.
\end{lemma}
\begin{proof}
    We coinductively construct a preproof for $\Gamma, A,\lnot A$.
    It has several subcases.
    Suppose that $A$ is $r\in t$. We consider the following preproof:
    \begin{center}
    \begin{mathprooftree}
        \def\fCenter#1#2{\sststile{#1}{#2}}
        \AxiomC{$H[\ttY] \fCenter{}{\no(\Gamma)\#\omega^{\gamma_0}\#\omega^{\gamma_1}} \Gamma,\ (s\mathrel{\dot\in} t\land r=s),\ (s\mathrel{\dot\notin} t\lor r\neq s)$}
        \RightLabel{$(\in)$}
        \UnaryInfC{$\cdots\ H[\ttY]\fCenter{}{\no(\Gamma)\# \no(r\in t)\# \omega^{\gamma_1}} \Gamma,\ r\in t,\ (s\mathrel{\dot\notin} t\lor r\neq s)\ \cdots$}
        \RightLabel{$(\notin)$}
        \UnaryInfC{$H[\ttX]\fCenter{}{\no(\Gamma)\# \no(r\in t)\cdot 2} \Gamma,\ r\in t,\ r\notin t$}
    \end{mathprooftree}
    \end{center}
    Here,
    \begin{itemize}
        \item $\gamma_0 = \rk(s\mathrel{\dot\in} t\land r=s)$, and
        \item $\gamma_1 = \rk(s\mathrel{\dot\notin} t\lor r\neq s)$.
    \end{itemize}
    We can see that $\rk(s\mathrel{\dot\in} t), \rk(s) < \rk(t)$, so
    \begin{equation*}
        \gamma_0 = \max\bigl(\rk (s\mathrel{\dot\in}t), \rk(r=s)\bigr)+1 < \max\bigl(\rk(t), \rk(r)+4,9\bigr)+1 \le \rk(r\in t)
    \end{equation*}
    and similarly we have $\gamma_1 < \rk(r\notin t)$.

    Now suppose that $A$ is $B\land C$. In this case, consider the following preproof:
    \begin{center}
    \begin{mathprooftree}
        \def\fCenter#1#2{\sststile{#1}{#2}}
        \AxiomC{$H[\ttX] \fCenter{}{\no(\Gamma)\#\no(B)\cdot 2} \Gamma,\ B,\ \lnot B$}
        \RightLabel{$(\lor)$}
        \UnaryInfC{$H[\ttX]\fCenter{}{\no(\Gamma)\# \no(B)\#\no(B\land C)} \Gamma,\ B\ \lnot B\lor\lnot C$}

        \AxiomC{$H[\ttX] \fCenter{}{\no(\Gamma)\#\no(C)\cdot 2} \Gamma,\ C,\ \lnot C$}
        \RightLabel{$(\lor)$}
        \UnaryInfC{$H[\ttX]\fCenter{}{\no(\Gamma)\# \no(B\land C)\#\no(C)} \Gamma,\ C,\ \lnot B\lor\lnot C$}
        \RightLabel{$(\land)$}
        \BinaryInfC{$H[\ttX]\fCenter{}{\no(\Gamma)\# \no(B\land C)} \Gamma,\ B\land C,\ \lnot B\lor\lnot C$}
    \end{mathprooftree}
    \end{center}

    Lastly, suppose that $A$ is $\exists x\in t B(x)$. Similar to Case $\in$, we get the following with an appropriate choice of $\gamma_0$ and $\gamma_1$:
    \begin{center}
    \begin{mathprooftree}
        \def\fCenter#1#2{\sststile{#1}{#2}}
        \AxiomC{$H[\ttY] \fCenter{}{\no(\Gamma)\#\omega^{\gamma_0}\#\omega^{\gamma_1}} \Gamma,\ (s\mathrel{\dot\in} t\land B(s)),\ (s\mathrel{\dot\notin} t\lor \lnot B(s))$}
        \RightLabel{$(\rmb\exists)$}
        \UnaryInfC{$\cdots\ H[\ttY]\fCenter{}{\no(\Gamma)\# \no(\exists x\in t B(x))\# \omega^{\gamma_1}} \Gamma,\ \exists x\in t B(x),\ (s\mathrel{\dot\notin} t\lor \lnot B(s))\ \cdots$}
        \RightLabel{$(\rmb\forall)$}
        \UnaryInfC{$H[\ttX]\fCenter{}{\no(\Gamma)\# \no(\exists x\in t B(x))\cdot 2} \Gamma,\ \exists x\in t B(x),\ \forall x\in t \lnot B(x)$}
    \end{mathprooftree}
    \end{center}
    The case when $A$ is $\exists x B(x)$ is similar to the bounded case.  This finishes the proof.
\end{proof}
The previous proof has the exact details of \cite[Lemma 5.4]{Fernandez2024SetTotalRecKPell} or \cite[Lemma 4.3(i)]{RathjenCook2016ClassifyingProvTotalsetftnKPKPP} except one point: The proof in the cited theorems constructs a proof by induction on the rank of the formula $A$. In the previous proof, we do not rely on induction; Instead, we construct a preproof coinductively. We will \emph{not} (and cannot) rely on recursive construction for the proof tree, which is replaced by corecursive construction. Again, the dilator bound confirms the constructed preproof is `valid' and `well-founded.'

\begin{lemma} \label{Lemma: Preproof exists for basic facts}
    \begin{enumerate}
        \item $H[\ttX]\Vdash s\notin s$.
        \item $H[\ttX]\Vdash s\mathrel{\dot\notin} t,\  s\in t$, if $H(\ttX)\vDash |s|<|t|$.
        \item $H[\ttX]\Vdash s\subseteq s$.
        \item $H[\ttX]\Vdash s=s$.
        \item $H[\ttX]\Vdash s\in \ttL(u)_\beta$, if $H(\ttX) \vDash |s| < \beta$.
    \end{enumerate}
\end{lemma}
\begin{proof}
    Its proof follows from mimicking the proof of \cite[Lemma 5.5]{Fernandez2024SetTotalRecKPell}, so we omit most of its details. Again, we build preproofs corecursively and not recursively.
\end{proof}

\begin{lemma}
    Let $\ttX$ be a template order and $s$, $t$ be terms. For every formula $A(x)$ in the language of set theory, we have
    \begin{equation*}
        H[\ttX]\Vdash s\neq t, \lnot A(s),  A(t).
    \end{equation*}
\end{lemma}
\begin{proof}
    See \cite[Lemma 5.7]{Fernandez2024SetTotalRecKPell} or \cite[Lemma 4.4]{RathjenCook2016ClassifyingProvTotalsetftnKPKPP}.
\end{proof}

\begin{lemma}[Set Induction]
    Let $A(x)$ be a formula and let 
    \begin{equation*}
        \Ind_A :\equiv \forall x \bigl[\bigl(\forall y\in x A(y)\bigr)\to A(x)\bigr].
    \end{equation*}
    Then we have $H[\ttX]\dststile{}{\omega^{\rk(\Ind_A)}}\Ind_A \to \forall x A(x)$.
\end{lemma}
\begin{proof}
    See \cite[Lemma 5.8]{Fernandez2024SetTotalRecKPell} or \cite[Lemma 4.5]{RathjenCook2016ClassifyingProvTotalsetftnKPKPP}.
\end{proof}

\begin{lemma}[Infinity]
    $H[\varnothing]\Vdash \exists x \bigl[\exists z\in x(z\in x)\land \forall y\in x \exists z\in x(y\in z)\bigr]$.
\end{lemma}
\begin{proof}
    We follow the proof of \cite[Lemma 4.6]{RathjenCook2016ClassifyingProvTotalsetftnKPKPP}.
\end{proof}

\begin{lemma}[Pairing]
    $H[\ttX]\Vdash \exists z (s\in z\land t\in z)$.
\end{lemma}
\begin{proof}
    We follow the proof of \cite[Lemma 4.8]{RathjenCook2016ClassifyingProvTotalsetftnKPKPP}.
    Let $\alpha = \max(|s|,|t|)$ computes in $H(\ttX)$. Then consider
    \begin{equation*}
        \begin{mathprooftree}
            \AxiomC{ \text{\autoref{Lemma: Preproof exists for basic facts} (5)} }
            \noLine
            \UnaryInfC{$\vdots$}
            \noLine
            \UnaryInfC{$H[\ttX] \Vdash s\in \ttL(u)_{\alpha+1}$}

            \AxiomC{ \text{\autoref{Lemma: Preproof exists for basic facts} (5)} }
            \noLine
            \UnaryInfC{$\vdots$}
            \noLine
            \UnaryInfC{$H[\ttX] \Vdash t\in \ttL(u)_{\alpha+1}$}

            \RightLabel{$(\land)$}
            \BinaryInfC{$H[\ttX] \Vdash s\in \ttL(u)_{\alpha+1} \land t\in \ttL(u)_{\alpha+1}$}
            \RightLabel{$(\exists)$}
            \UnaryInfC{$H[\ttX] \Vdash \exists z [s\in z\land t\in z]$}
        \end{mathprooftree}
        \qedhere 
    \end{equation*}
\end{proof}

\begin{lemma}[Union]
    $H[\ttX] \Vdash \exists z \forall y\in s \forall x\in y (x\in z)$.
\end{lemma}
\begin{proof}
    See \cite[Lemma 4.8]{RathjenCook2016ClassifyingProvTotalsetftnKPKPP}.
\end{proof}

\begin{lemma}[$\Delta_0$-Separation]
    If $A(x,y_0,\cdots,y_{n-1})$ is a $\Delta_0$-formula with all free variables expressed, and $s,t_0,\cdots,t_{n-1}\in \bfL(u)_\ttX$, then 
    \begin{equation*}
        H[\ttX]\Vdash \exists y \bigl[ \forall x\in y \bigl(x\in s\land A(x,t_0,\cdots,t_{n-1})\bigr) \land \forall x\in s \bigl(\lnot A(x,t_0,\cdots,t_{n-1})\lor x\in y\bigr)\bigr].
    \end{equation*}
\end{lemma}
\begin{proof}
    See \cite[Lemma 5.9]{Fernandez2024SetTotalRecKPell} or \cite[Lemma 4.7]{RathjenCook2016ClassifyingProvTotalsetftnKPKPP}.
\end{proof}

\begin{lemma}[$\Delta_0$-Collection]
    Suppose that $A(x,y)$ is a $\Delta_0$-formula in the language of set theory. Then
    \begin{equation*}
        H[\ttX]\Vdash [\forall x\in s \exists y A(x,y)] \to [\exists z \forall x\in s\exists y\in z A(x,y)].
    \end{equation*}
\end{lemma}
\begin{proof}
    See \cite[Lemma 5.16]{Fernandez2024SetTotalRecKPell} or \cite[Lemma 4.9]{RathjenCook2016ClassifyingProvTotalsetftnKPKPP}.
\end{proof}

Hence, we get the following:
\begin{theorem}
    Suppose that $\Gamma(x_0,\cdots,x_{n-1})$ is a finite set of formulas with all free variables displayed such that $\KP\vdash \bigvee \Gamma(x_0,\cdots,x_{n-1})$. 
    Then there is $m<\omega$ such that for every template order $\ttX$ and terms $s_0,\cdots,s_{n-1}\in \bfL(u)_{H(\ttX)}$, we have $H[\ttX]\sststile{\ttOmega+m}{\ttOmega\cdot \omega^m} \Gamma(s_0,\cdots,s_{n-1})$.
\end{theorem}

\section{Impredicative collapsing theorem}

Now, let us perform the functorial cut elimination. Similar to the impredicative cut elimination in the operator-controlled derivation, we only consider the case when the last sequent consists only of $\Sigma\cup \Pi$-formulas. Unlike the operator-controlled derivation system, where we can apply the inductive proof over the rank of a sequent, we must rely on the coinduction on a preproof. It becomes problematic when we eliminate a cut of rank $\ttOmega$ for the following reason: In a traditional operator-based derivation system, we start with
\begin{equation*}
    \calH_\eta \sststile{\ttOmega+1}{\alpha_0} \Gamma,A\text{ and }\calH_\eta \sststile{\ttOmega+1}{\alpha_0} \Gamma,\lnot A
\end{equation*}
for a $\Sigma$-formula $A$. We first apply the (co)inductive hypothesis to $\calH_\eta \sststile{\ttOmega+1}{\alpha_0} \Gamma,A$ to get $\calH_{\hat{\alpha}_0} \sststile{\psi\hat{\alpha}_0}{\psi\hat\alpha_0} \Gamma,A$ for $\hat\alpha_0 = \eta+\omega^{\ttOmega+\alpha_0}$. Then we apply the Boundedness theorem to get $\calH_{\hat{\alpha}_0} \sststile{\psi\hat{\alpha}_0}{\psi\hat\alpha_0} \Gamma,A^{L_{\psi\hat{\alpha}_0}}$.
In $\calH_\eta \sststile{\ttOmega+1}{\alpha_0}\Gamma,\lnot A$ side, $\lnot A$ is a $\Pi$-formula and $\eta<\hat{\alpha}_0$. Hence, we have $\calH_{\hat\alpha_0} \sststile{\ttOmega+1}{\alpha_0}\Gamma,\lnot A^{L_{\psi\hat\alpha_0}}$ by Persistence. Then we apply the inductive hypothesis on $\alpha_0$ to collapse the proof.

This strategy does not work for two reasons: First, in our current setting, there are no $\psi$-numbers in $H$. We will get $\psi$-numbers just after the collapsing, but restricting $\lnot A$ to $(\lnot A)^{L_{\psi \hat{\alpha}_0}}$ must happen before collapsing. Second, collapsing $\calH_{\hat\alpha_0} \sststile{\ttOmega+1}{\alpha_0}\Gamma,\lnot A^{L_{\psi\hat\alpha_0}}$ relies on the induction on $\alpha_0$. However, we cannot use sequent rank to certify the (co)induction. We have only used coinductive reasoning over a preproof. To resolve both issues, we do both the persistence argument and the collapsing simultaneously. Moreover, the functorial nature of the proof system also requires a collapsing step for substitution. This forces sequents to allow $\Pi$-formulas (unlike the operator-controlled derivation systems, where the collapsing happens only for sequents with $\Sigma$-formulas only), and handling the $(\rmb\forall)$ case requires the substitution lemma.

In the following theorem, we use the normal-term convention in \autoref{Remark: Convention for the normal flower element}.
In particular, when $\psi(\xi)\in \psi_\Omega H(\ttY+\Omega)$ is regarded as an element of the part of $H\bigl(\ttY +  \psi_\Omega H(\ttY+\Omega)\bigr)$ below $\ttOmega$, the notation $\psi(\xi)$ denotes $\mu^{E\circ F}_{\ttY +  \psi_\Omega H(\ttY+\Omega)}\bigl(\psi(\xi)\bigr)$; when $E = \varepsilon$, this denotes $\varepsilon_{\psi(\xi)}$.

\begin{theorem}[Functorial Collapsing]
    Suppose that $H = E(F+1+G)$ for $E=\varepsilon$ or $\varphi_\lambda$ for a non-empty well-order $\lambda$, a \emph{normal flower} $F$, and $\Gamma$ be a sequent of $\forall \Sigma$-formulas in the language of $\RS_H$.
    Suppose we are given
    \begin{enumerate}
        \item a preproof of $H[\ttX]\sststile{\ttOmega+1}{\alpha}\Gamma$, and
        \item an order-preserving map $f\colon\ttX\to\ttY+\psi_\Omega H(\ttY+\Omega)$.
    \end{enumerate}
    Let $\eta\in H\bigl(\ttY+\psi_\Omega H(\ttY+\Omega)\bigr)$
    satisfy
    \begin{enumerate}
        \item $\hat\sfK(\eta)<\eta$, and 
        \item $\psi^{-1}(f(v))<\eta$ for every $v\in\ttX$ such that $f(v)\in\psi_\Omega H(\ttY+\Omega)$.
    \end{enumerate}   
    If we define
    \begin{equation*}
        \hat\alpha[\eta]
        =
        \eta+\omega^{\ttOmega+H(f)(\alpha)}
        \in
        H\bigl(\ttY+\psi_\Omega H(\ttY+\Omega)\bigr),
    \end{equation*}
    then
    \begin{equation*}
        H\bigl(\Id+\psi_\Omega H(\Id+\Omega)\bigr)[\ttY]
        \sststile{\psi(\hat\alpha[\eta])}{\psi(\hat\alpha[\eta])}
        \bfL(u)_{H(f)}[\Gamma]^\eta,
    \end{equation*}
    where $A^\eta$ is the formula obtained from $A$ by replacing every occurrence of
    $\forall x$ by
    $\forall x\in\ttL(u)_{\psi(\eta)}$.
\end{theorem}

\begin{proof}
    We first record that all collapsing terms occurring in the coinductive construction are well-defined.
    More generally, for $\beta\in H(\ttX)$, put $\hat\beta[\eta] = \eta+\omega^{\ttOmega+H(f)(\beta)}$.
    We claim that $\hat\sfK\bigl(\hat\beta[\eta]\bigr)<\eta$.
    By the support properties of the internal ordinal operations,
    \begin{equation*}
        \hat\sfK\bigl(\hat\beta[\eta]\bigr)
        \subseteq
        \hat\sfK(\eta)\cup
        \hat\sfK\bigl(H(f)(\beta)\bigr).
    \end{equation*}
    It therefore suffices to show $\hat\sfK\bigl(H(f)(\beta)\bigr)<\eta$.

    Let $\xi\in\hat\sfK(H(f)(\beta))$.
    Every $\psi$-term contributing to $\hat\sfK(H(f)(\beta))$ has the form $f(v)$ for some $v\in\supp^H_{\ttX}(\beta)$ such that $f(v)\in\psi_\Omega H(\ttY+\Omega)$.
    Writing $f(v)=\psi(\zeta)$, the definition of $\hat\sfK$ gives $\xi\le\zeta=\psi^{-1}(f(v))<\eta$.
    Hence $\hat\sfK\bigl(H(f)(\beta)\bigr)<\eta$.
    Together with $\hat\sfK(\eta)<\eta$, this proves $\hat\sfK\bigl(\hat\beta[\eta]\bigr)<\eta<\hat\beta[\eta]$.
    Consequently, $\psi\bigl(\hat\beta[\eta]\bigr)\in\psi_\Omega H(\ttY+\Omega)$ for every $\beta\in H(\ttX)$.

    We now prove the theorem by coinduction on the given preproof, with a case distinction according to its last inference.
    The construction takes a preproof $\pi$, an order-preserving map $f$, and a collapsing parameter $\eta$ satisfying the preceding conditions, and produces a preproof $\pi[f,\eta]$.

    For the remainder of the proof, put
    \begin{equation*}
        D=\Id+\psi_\Omega H(\Id+\Omega).
    \end{equation*}
    Whenever $\ttY\subseteq Z$, we identify terms and maps over $\ttY$ with their images under the corresponding canonical embeddings.
    Define
    \begin{equation*}
        I^\eta_Z
        =
        Z+
        \bigl\{
            q\in\psi_\Omega H(Z+\Omega)
            \bigm|
            \psi^{-1}(q)<\eta
        \bigr\}
        \subseteq D(Z).
    \end{equation*}
    Here $\eta$ is identified with its image over $Z$.
    Since $\psi$ is increasing, $I^\eta_Z$ is an initial segment of $D(Z)$.
    By the hypothesis on $f$, $\ran f\subseteq I^\eta_Z$  whenever $Z\supseteq\ttY$.

    We only examine the non-trivial cases. Case $(\notin)$ is similar to that of $(\rmb\forall)$, and Cases $(\in)$ is similar to that of $(\rmb\exists)$.
    \begin{enumerate}
        \item Suppose that the last inference of $\pi$ is a bounded $(\rmb\forall)$-inference with principal formula $\forall x\in t A(x)$:
        \begin{center}
        \begin{mathprooftree}
            \def\fCenter#1#2{\sststile{#1}{#2}}
            \noLine
            \AxiomC{$\pi_s$}
            \noLine
            \UnaryInfC{$\vdots$}
            \noLine
            \UnaryInfC{$\cdots\quad
            H[\ttX_s]\fCenter{\ttOmega+1}{\alpha_s}
            \Gamma',\ s\dotin t\to A(s)
            \quad\cdots$ $(\spadesuit_t)$}
            \RightLabel{$(\rmb\forall)$ \hspace{2em} ($\alpha_s<\alpha$)}
            \UnaryInfC{$H[\ttX]\fCenter{\ttOmega+1}{\alpha}
            \Gamma',\ \forall x\in t A(x)$}
        \end{mathprooftree}
        \end{center}
    
        We have to construct the corresponding bounded $(\rmb\forall)$-inference over $H\circ D$.
        Fix a target premise index $(\ttY_r,r)$ satisfying $(\clubsuit_{H(f)(t)})$.
        By the construction in the proof of \autoref{Lemma: Substitution lemma}, there are a source premise index $(\ttX_{s(r)},s(r))$ satisfying $(\spadesuit_t)$ and an increasing map $f_r\colon\ttX_{s(r)}\to D(\ttY_r)$ such that
        \begin{itemize}
            \item $\ran f_r= \ran f\cup \parsupp^H_{D(\ttY_r)}(r)$,
            \item $H(f_r)(s(r))=r$,
            \item $f_r\restriction\ttX=f$.
        \end{itemize}
    
        We first check that the coinductive hypothesis can be applied with the same collapsing parameter $\eta$.
        Since $\ran f\subseteq I^\eta_{\ttY_r}$,
        \autoref{Remark: Parameter control in substitution} yields $\ran f_r\subseteq I^\eta_{\ttY_r}$.
        Hence $\psi^{-1}(f_r(v))<\eta$ whenever $v\in\ttX_{s(r)}$ and $f_r(v)\in\psi_\Omega H(\ttY_r+\Omega)$.
        Thus $f_r$ and $\eta$ satisfy the hypotheses of the collapsing construction.
    
        Put
        \begin{equation*}
            \hat\alpha_{s(r)}[\eta]
            =
            \eta+
            \omega^{\ttOmega+H(f_r)(\alpha_{s(r)})}.
        \end{equation*}
        Applying the coinductive hypothesis to $\pi_{s(r)}$, $f_r$, and $\eta$ gives
        \begin{equation*}
            H\circ D[\ttY_r]
            \sststile{\psi(\hat\alpha_{s(r)}[\eta])}
            {\psi(\hat\alpha_{s(r)}[\eta])}
            H(f)[\Gamma']^\eta,\
            r\dotin H(f)(t)\to H(f)(A)^\eta(r).
        \end{equation*}
        Here we have used $f_r\restriction\ttX=f$ and
        $H(f_r)(s(r))=r$.
        Since $\alpha_{s(r)}<\alpha$ and $H(f_r)$ is increasing,
        \begin{equation*}
            H(f_r)(\alpha_{s(r)})
            <
            H(f_r)(\alpha)
            =
            H(f)(\alpha).
        \end{equation*}
        Therefore $\hat\alpha_{s(r)}[\eta] < \hat\alpha[\eta]$, and hence $\psi(\hat\alpha_{s(r)}[\eta])<\psi(\hat\alpha[\eta])$.
        By monotonicity, we may increase the cut rank of the preceding preproof to $\psi(\hat\alpha[\eta])$ while retaining its height $\psi(\hat\alpha_{s(r)}[\eta])$.
    
        Since $(\ttY_r,r)$ was an arbitrary index satisfying $(\clubsuit)$, the resulting preproofs provide all premises required for $(\rmb\forall)$.
        We therefore obtain
        \begin{equation*}
            H\circ D[\ttY]
            \sststile{\psi(\hat\alpha[\eta])}
            {\psi(\hat\alpha[\eta])}
            H(f)[\Gamma']^\eta,\
            \forall x\in H(f)(t)\,H(f)(A)^\eta(x).
        \end{equation*}

        \item Suppose that the last inference is the unbounded universal rule $(\forall)$, with principal formula $\forall x A(x)$.
        Thus its premises are indexed by all appropriate genuine set terms $s$ over finite extensions of $\ttX$, and have the form
        \begin{equation*}
            H[\ttX_s]
            \sststile{\ttOmega+1}{\alpha_s}
            \Gamma',\,A(s),
            \qquad
            \alpha_s<\alpha.
        \end{equation*}
        
        After collapsing, the principal formula becomes
        \begin{equation*}
            \forall x\in\ttL(u)_{\psi(\eta)}
            H(f)(A)^\eta(x).
        \end{equation*}
        Hence the target inference is an ordinary bounded $(\rmb\forall)$-inference.
        
        Fix a target premise index $(\ttY_r,r)$ satisfying $(\clubsuit)$ for this bounded universal inference. In particular, for a non-constant term $r$ we have $|r|<\psi(\eta)$.
        By the same pullback construction as in the proof of
        \autoref{Lemma: Substitution lemma}, there are a source premise index
        $(\ttX_{s(r)},s(r))$ and an increasing map $ f_r\colon\ttX_{s(r)}\to D(\ttY_r)$ such that $f_r\restriction\ttX=f$, $H(f_r)(s(r))=r$, and $\ran f_r=\ran f\cup\parsupp^H_{D(\ttY_r)}(r)$.
        Unlike the bounded universal case, there is no source bound $t$ whose height has to be checked, so the pullback term $s(r)$ automatically indexes one of the premises of the unbounded universal inference.
        
        We claim that $\parsupp^H_{D(\ttY_r)}(r)\subseteq I^\eta_{\ttY_r}$.
        If $r$ is a constant from $u$, then its parameter support is empty, and there is nothing to prove.
        Otherwise, let $b\in\supp^{\bfL(u)}_{H(D(\ttY_r))}(r)$.
        Then $b\le |r|<\psi(\eta)$.
        Below $\ttOmega$, the order $H(D(\ttY_r))$ is identified with the
        $E\circ F$-part of $H(D(\ttY_r))$.
        The inequality
        $b<\psi(\eta)$ implies $\supp^H_{D(\ttY_r)}(b)\subseteq I^\eta_{\ttY_r}$.
        Taking the union over all
        $b\in\supp^{\bfL(u)}_{H(D(\ttY_r))}(r)$ gives $\parsupp^H_{D(\ttY_r)}(r)\subseteq I^\eta_{\ttY_r}$.
        Since the original assumption on $f$ gives $\ran f\subseteq I^\eta_{\ttY_r}$,
        we conclude that $\ran f_r\subseteq I^\eta_{\ttY_r}$.
        Consequently, $\psi^{-1}(f_r(v))<\eta$ whenever
        $v\in\ttX_{s(r)}$ and
        $f_r(v)\in\psi_\Omega H(\ttY_r+\Omega)$.
        Thus the coinductive hypothesis applies to the source premise indexed by $s(r)$ with the same collapsing parameter $\eta$.
        
        Put $\hat\alpha_{s(r)}[\eta]= \eta+\omega^{\ttOmega+H(f_r)(\alpha_{s(r)})}$.
        The coinductive hypothesis yields
        \begin{equation*}
            H\circ D[\ttY_r]
            \sststile{\psi(\hat\alpha_{s(r)}[\eta])}
            {\psi(\hat\alpha_{s(r)}[\eta])}
            H(f)[\Gamma']^\eta,\,
            H(f)(A)^\eta(r).
        \end{equation*}
        
        Moreover,
        \begin{equation*}
            H(f_r)(\alpha_{s(r)})
            <
            H(f_r)(\alpha)
            =
            H(f)(\alpha),
        \end{equation*}
        and therefore $\hat\alpha_{s(r)}[\eta] <\hat\alpha[\eta]$.
        Hence $\psi(\hat\alpha_{s(r)}[\eta])<\psi(\hat\alpha[\eta])$.
        
        Thus, after increasing the cut rank to
        $\psi(\hat\alpha[\eta])$, we obtain the premise required for the
        bounded universal inference with bound
        $\ttL(u)_{\psi(\eta)}$.
        Since $(\ttY_r,r)$ was arbitrary, we conclude
        \begin{equation*}
            H\circ D[\ttY]
            \sststile{\psi(\hat\alpha[\eta])}
            {\psi(\hat\alpha[\eta])}
            H(f)[\Gamma']^\eta,\,
            \forall x\in\ttL(u)_{\psi(\eta)}
            H(f)(A)^\eta(x). 
        \end{equation*}

    \end{enumerate}

    Before examining the existential case, let us make the following observation:
    \begin{equation} \label{Formula: Bounding H f rho}
        \rho < \ttOmega \implies H(f)(\rho) < \psi(\eta).
    \end{equation}
    Indeed, $H(f)(\rho)$ belongs to the part of
    $H(\ttY+\psi_\Omega H(\ttY+\Omega))$ below $\ttOmega$, and all of its parameters belong to $I^\eta_{\ttY}$. Since $E\circ F$ is a normal flower, it follows that
    \begin{equation*}
        H(f)(\rho) < \mu^{E\circ F}_{D(\ttY)}(\psi(\eta)).
    \end{equation*}
    In particular, if $s$ is a  set term over $\ttX$, then
    \begin{equation*}
        \rk^{\bfL(u)}_{H\circ D(\ttY)}
        \bigl(\bfL(u)_{H(f)}(s)\bigr)
        =
        H(f)\bigl(\rk^{\bfL(u)}_{H(\ttX)}(s)\bigr)
        <
        \psi(\eta).
    \end{equation*}

    \begin{enumerate}[start=3]
        \item Suppose that $\Gamma=\Gamma',\exists x\in t\,A(x)$ is obtained by $(\rmb\exists)$ with principal formula $\exists x\in t\,A(x)$.
        Then there are a set term $s$ and $\alpha_0<\alpha$ such that $|s|<|t|$, $\rk^{\bfL(u)}_{H(\ttX)}(s)<\alpha$, and
        \begin{equation*}
            H[\ttX]
            \sststile{\ttOmega+1}{\alpha_0}
            \Gamma',\,s\mathrel{\dotin}t\land A(s).
        \end{equation*}
        
        By the coinductive hypothesis, we obtain
        \begin{equation*}
            H\circ D[\ttY]
            \sststile{\psi(\hat\alpha_0[\eta])}
            {\psi(\hat\alpha_0[\eta])}
            H(f)[\Gamma']^\eta,\,
            H(f)(s)\mathrel{\dotin}H(f)(t)
            \land
            \bigl(H(f)(A(s))\bigr)^\eta.
        \end{equation*}
        The preliminary calculation shows that
        $\hat\alpha_0[\eta]$ is well-defined. Since $\alpha_0<\alpha$, we have $\hat\alpha_0[\eta]<\hat\alpha[\eta]$, and hence $\psi(\hat\alpha_0[\eta]) < \psi(\hat\alpha[\eta])$.
        
        It remains to verify the side conditions for $(\rmb\exists)$.
        Functoriality of term heights gives
        \begin{equation*}
            |H(f)(s)|
            =
            H(f)(|s|)
            <
            H(f)(|t|)
            =
            |H(f)(t)|.
        \end{equation*}
        Hence the preceding observation and functoriality of ranks yield
        \begin{equation*}
            \rk^{\bfL(u)}_
            {H\circ D(\ttY)}
            \bigl(H(f)(s)\bigr)
            =
            H(f)\bigl(
                \rk^{\bfL(u)}_{H(\ttX)}(s)
            \bigr)
            <
            \psi(\eta)<
            \psi(\hat\alpha[\eta]).
        \end{equation*}
        
        We may therefore increase the cut rank of the preceding preproof
        from $\psi(\hat\alpha_0[\eta])$ to
        $\psi(\hat\alpha[\eta])$, while retaining its height
        $\psi(\hat\alpha_0[\eta])$.
        Applying $(\rmb\exists)$ now gives
        \begin{equation*}
            H\circ D[\ttY]
            \sststile{\psi(\hat\alpha[\eta])}
            {\psi(\hat\alpha[\eta])}
            H(f)[\Gamma']^\eta,\,
            \exists x\in H(f)(t)\,
            H(f)(A)^\eta(x).
        \end{equation*}

        \item Suppose that
        $\Gamma=\Gamma',\exists x\,A(x)$ is obtained by $(\exists)$.
        Then there are a set term $s$ and $\alpha_0<\alpha$ such that
        \begin{equation*}
            H[\ttX]\sststile{\ttOmega+1}{\alpha_0}
            \Gamma',A(s),
        \end{equation*}
        with $\rk^{\bfL(u)}_{H(\ttX)}(s)<\alpha$.
        By the coinductive hypothesis,
        \begin{equation*}
            H\circ D[\ttY]
            \sststile{\psi(\hat\alpha_0[\eta])}
            {\psi(\hat\alpha_0[\eta])}
            H(f)[\Gamma']^\eta,\,
            H(f)(A)^\eta(H(f)(s)).
        \end{equation*}
        Since $s$ is a set term, the preceding observation gives
        \begin{equation*}
            \rk^{\bfL(u)}_{H\circ D(\ttY)}(H(f)(s))
            <\psi(\eta)
            <\psi(\hat\alpha[\eta]).
        \end{equation*}
        Moreover,
        $\psi(\hat\alpha_0[\eta])<\psi(\hat\alpha[\eta])$.
        Thus we may increase the cut bound to
        $\psi(\hat\alpha[\eta])$ and apply $(\exists)$, obtaining
        \begin{equation*}
            H\circ D[\ttY]
            \sststile{\psi(\hat\alpha[\eta])}
            {\psi(\hat\alpha[\eta])}
            H(f)[\Gamma']^\eta,\,
            \exists x\,H(f)(A)^\eta(x).
        \end{equation*}

        \item Suppose that $\Gamma = \Gamma', \exists x A^x$ is obtained by $(\mathrm{Refl})$ with principal formula $\exists x A^x$ for a $\Sigma$-formula $A$. Then we have $\alpha_0<\alpha$ such that
        \begin{equation*}
            H[\ttX] \sststile{\ttOmega+1}{\alpha_0} \Gamma', A.
        \end{equation*}
        Applying the coinductive hypothesis gives
        \begin{equation*}
            H\circ D[\ttY] \sststile{\psi(\hat{\alpha}_0[\eta])}{\psi(\hat{\alpha}_0[\eta])} H(f)(\Gamma')^\eta, H(f)(A).
        \end{equation*}
        By applying the Bounding Lemma, we have
        \begin{equation*}
            H\circ D[\ttY] \sststile{\psi(\hat{\alpha}_0[\eta])}{\psi(\hat{\alpha}_0[\eta])} H(f)(\Gamma')^\eta, H(f)(A)^{\ttL(u)_{\psi(\hat{\alpha}_0[\eta])}}.
        \end{equation*}
        Hence by applying $(\exists)$, we have
        \begin{equation*}
            H\circ D[\ttY] \sststile{\psi(\hat{\alpha}_0[\eta])}{\psi(\hat{\alpha}_0[\eta])+1} H(f)(\Gamma')^\eta, \exists z H(f)(A)^z.
        \end{equation*}
        We get a desired result since $H\circ D(\ttY)\vDash \psi(\hat{\alpha}_0[\eta])+1 < {\psi(\hat{\alpha}[\eta])}$.
        
        \item Finally, suppose that $\Gamma$ is obtained by Cut from the following two sequents:
        \begin{equation*}
             H[\ttX] \sststile{\ttOmega+1}{\alpha_0} \Gamma,A \qquad\text{and}\qquad
             H[\ttX] \sststile{\ttOmega+1}{\alpha_1} \Gamma,\lnot A .
        \end{equation*}
        
        Here we may assume that $\alpha_0 = \alpha_1$ by increasing the smaller proof rank to the larger one if necessary.
        If $\rk^{\bfL(u)}_{H(\ttX)}(A) < \ttOmega$, then every quantifier in $A$ must be bounded.
        Applying the coinductive hypotheses to each of the sequents yields
        \begin{itemize}
            \item $H\circ D[\ttY]  \sststile{\psi(\hat{\alpha}_0[\eta])}{\psi(\hat{\alpha}_0[\eta])} H(f)(\Gamma)^\eta, H(f)(A)$, and 
            \item $H\circ D[\ttY]  \sststile{\psi(\hat{\alpha}_0[\eta])}{\psi(\hat{\alpha}_0[\eta])}  H(f)(\Gamma)^\eta, H(f)(\lnot A)$.
        \end{itemize}
        From $\psi(\hat{\alpha}_0[\eta]) < \psi(\hat{\alpha}[\eta])$, we can increase the cut rank of each sequents by $\psi(\hat{\alpha}[\eta])$. Then we can perform (Cut) to eliminate $H(f)(A)$ and $H(f)(\lnot A)$ since
        \begin{equation*}
            \rk^{\bfL(u)}_{H\circ D(\ttY)} \bigl(H(f)(A)\bigr) = H(f)\bigl(\rk^{\bfL(u)}_{H(\ttX)}(A)\bigr) < \psi(\eta) < \psi(\hat{\alpha}[\eta]).
        \end{equation*}
        The second inequality holds by $\rk^{\bfL(u)}_{H(\ttX)}(A) < \ttOmega$ and \eqref{Formula: Bounding H f rho}.
        
        Now, let us consider the case $\rk^{\bfL(u)}_{H(\ttX)}(A) = \ttOmega$. By symmetry, we assume that $A$ is a $\Sigma$-formula. The coinductive hypothesis yields
        \begin{equation*}
            H\circ D[\ttY]  \sststile{\psi(\hat{\alpha}_0[\eta])}{\psi(\hat{\alpha}_0[\eta])}  H(f)(\Gamma)^\eta, H(f)(A).
        \end{equation*}
        By the Bounding Lemma, we have
        \begin{equation} \label{Formula: Collapsing Cut rank Omega, Positive cut 00}
            H\circ D[\ttY]  \sststile{\psi(\hat{\alpha}_0[\eta])}{\psi(\hat{\alpha}_0[\eta])} H(f)(\Gamma)^\eta, H(f)(A)^{\ttL(u)_{\psi(\hat{\alpha}_0[\eta])}}.
        \end{equation}

        Now we want to apply the coinductive hypothesis to the other sequent, using $\hat{\alpha}_0[\eta]$ instead of $\eta$ (but use the same $f$). It would yield
        \begin{equation} \label{Formula: Collapsing Cut rank Omega, Negative cut 00}
            H\circ D[\ttY]  \sststile{\psi(\hat{\alpha}_0[\hat{\alpha}_0[\eta]])}{\psi(\hat{\alpha}_0[\hat{\alpha}_0[\eta]])} H(f)(\Gamma)^{\hat{\alpha}_0[\eta]}, H(f)(\lnot A)^{\ttL(u)_{\psi(\hat{\alpha}_0[\eta])}}.
        \end{equation}
        Note that $\hat\sfK(\hat\alpha_0[\eta])<\hat\alpha_0[\eta]$ since
        \begin{equation*}
            \hat\sfK(\hat\alpha_0[\eta]) \subseteq \hat{\sfK}(\eta) \cup \hat{\sfK}\bigl(H(f)(\alpha_0)\bigr) < \eta <\hat{\alpha}_0[\eta],
        \end{equation*}
        where the first inequality follows from the assumption on $f$ saying $\psi^{-1}(f(v))<\eta$ for every $v\in \ttX$.
        This justifies the previous application of the coinductive hypothesis with $\hat\alpha_0[\eta]$. The other condition of the coinductive hypothesis follows immediately.
        By an easy consequence of the Persistence lemma (obtained by iterating it over the outer universal quantifiers) applied to $H(f)(\Gamma)^{\hat{\alpha}_0[\eta]}$ in \eqref{Formula: Collapsing Cut rank Omega, Negative cut 00}, we have
        \begin{equation} \label{Formula: Collapsing Cut rank Omega, Negative cut 01}
            H(\Id+\psi_\Omega H(\Id+\Omega))[\ttY]  \sststile{\psi(\hat{\alpha}_0[\hat{\alpha}_0[\eta]])}{\psi(\hat{\alpha}_0[\hat{\alpha}_0[\eta]])} H(f)(\Gamma)^\eta, H(f)(\lnot A)^{\ttL(u)_{\psi(\hat{\alpha}_0[\eta])}}.
        \end{equation}
        
        Now observe that $\hat{\alpha}_0[\eta] < \hat{\alpha}_0[\hat{\alpha}_0[\eta]]< \hat{\alpha}[\eta]$; The last inequality follows from
        \begin{equation*}
            \hat{\alpha}_0[\hat{\alpha}_0[\eta]] = \eta + \omega^{\ttOmega + H(f)(\alpha_0)}\cdot 2  < \eta+ \omega^{\ttOmega+H(f)(\alpha)} = \hat{\alpha}[\eta].
        \end{equation*}
        Hence, we can increase the cut rank both in \eqref{Formula: Collapsing Cut rank Omega, Positive cut 00} and \eqref{Formula: Collapsing Cut rank Omega, Negative cut 01} to $\psi(\hat{\alpha}[\eta])$. 
        
        Then we perform (Cut) to have
        \begin{equation*}
            H(\Id+\psi_\Omega H(\Id+\Omega)) [\ttY] \sststile{\psi(\hat{\alpha}[\eta])}{\psi(\hat{\alpha}[\eta])} H(f)(\Gamma)^\eta.
        \end{equation*}
        We can apply the (Cut) since all original parameters occurring in $H(f)(A)$ have levels below $\psi(\eta)$, while the relativization introduces no level above $\psi(\hat\alpha_0[\eta])$. Hence, for some $n<\omega$,
        \begin{equation*}
            \rk^{\bfL(u)}_{H\circ D(\ttY)}
            \bigl(H(f)(A)^{\ttL(u)_{\psi(\hat\alpha_0[\eta])}}\bigr)
            \le
            \psi(\hat\alpha_0[\eta])+n
            <
            \psi(\hat\alpha[\eta]).\qedhere 
        \end{equation*}
    \end{enumerate}
\end{proof}

As a special case, we have the following:
\begin{corollary}
    Suppose that $H=E\circ (F+1+G)$ for a normal flower $F$, and we are given $H[\ttX]\sststile{\ttOmega+1}{\alpha} \Gamma$ for a sequent of $\Sigma$-formulas. Then we have $H\bigl(\Id+\psi_\Omega H(\Id+\Omega)\bigr)[\ttX] \sststile{\psi(\omega^{\ttOmega+\alpha})}{\psi(\omega^{\ttOmega+\alpha})} \Gamma$.
\end{corollary}
\begin{proof}
    Apply the Functorial Collapsing to $H[\ttX]\sststile{\ttOmega+1}{\alpha} \Gamma$, the inclusion map $\iota\colon \ttX\to \ttX + \psi_\Omega H(\ttX+\Omega)$ with $\eta=0$. Note that $\hat{\alpha}[0] = \omega^{\ttOmega+\alpha}$.
\end{proof}

\section{Consequences of the functorial collapsing}

Most applications of the results in this paper use both the completeness of $\sfu\RS$ and the cut-elimination of $\RS$. We need the completeness of $\sfu\RS$ to get a dilator bounding a $\rmbeta$-proof of a target sequent; Then we move to $\RS$ to perform cut-elimination. We assume that ``For every dilator $D$, $\psi D$ is well-founded'' holds over the metatheory.
First, let us observe the following soundness result:
\begin{proposition}
    Suppose that we have a $\rmbeta$-proof of $H[\ttX]\sststile{\rho}{\beta} \Gamma$ for $\rho<\ttOmega$. Then for every ordinal $\alpha$ and $f\colon \ttX\to\alpha$, we have
    $L_{H(\alpha)}\vDash \bigl[\mkern-5mu\bigl[ \Gamma[H(f)] \bigr]\mkern-5mu\bigr]$.
\end{proposition}
\begin{proof}
    Since $\rho<\ttOmega$, the $\rmbeta$-proof does not use $\mathrm{(Refl)}$ rule. From a given $\RS$-proof of the sequent $H[\ttX]\sststile{\rho}{\alpha} \Gamma$, we can produce a $\sfu\RS_{H(\alpha)}$-proof of $\Gamma[H(f)]$ by \autoref{Proposition: Instantiating uRS proof into a uRS alpha proof}, and the resulting proof also does not use $\mathrm{(Refl)}$ rule.
    Then we get the desired result by \autoref{Proposition: Soundness of uRS proof}.
\end{proof}

We need to turn a $\sfu\RS$-proof into an $\RS$-proof. This requires bounding sequents in a $\sfu\RS$-proof into an appropriate dilator. 
This dilator is given by turning a $\sfu\RS$-proof into a dilator.
\begin{definition}
    \label{Definition: Linearization of a uRS proof}
    Let $\pi$ be a $\rmbeta$-preproof in $\sfu\RS$ whose last sequent takes the form $[\varnothing]\vdash\Gamma$, and let $T_\pi$ be its tree of sequent occurrences. The root $\langle\rangle$ corresponds to the last sequent of $\pi$ and has template order $\ttX_{\langle\rangle}=\varnothing$. Each node $p\in T_\pi$ is identified with the sequence of premise labels along the path from the root to $p$. A finitary premise is labeled by $0$ or $1$, as appropriate, whereas a premise of $(\notin)$, $(\rmb\forall)$, or $(\forall)$ is labeled by its index $(\ttY,s)$. If the sequent at $p$ is $[\ttX]\vdash\Gamma_p$, we put $\ttX_p=\ttX$.

    Now we define $\Lin(\pi)$ as follows:
    \begin{equation*}
        \Lin(\pi)(A)
        =
        \bigl\{
            f[p] \bigm|
            p\in T_\pi,\ 
            f\colon\ttX_p\to A
            \text{ increasing}
        \bigr\},
    \end{equation*}
    ordered by the Kleene--Brouwer order over $2+\bfL(u)_A$.
    Here $f[\lag s_0,\cdots,s_{n-1}\rag] = \bigl\langle 2+\bfL(u)_f(s_i) \bigm| i<n\bigr\rangle$. We also define, for an increasing $h\colon A\to B$,
    \begin{equation*}
        \Lin(\pi)(h) \bigl(f[p]\bigr) = (h\circ f)[p].
    \end{equation*}
    We also define $\supp^{\Lin(\pi)}_A(f[p]) = \ran f$.
\end{definition}
We can see that $\supp^{\Lin(\pi)}$ is natural and satisfies the support condition. Moreover, the map
\begin{equation*}
    (n, f[p]) \mapsto (p,\ttX_p)
\end{equation*}
is a bijection from $\Tr(\Lin(\pi))$ to the set of $\{(p,\ttX_p)\mid p\in T_\pi\}$. Hence, we use $(p,\ttX_p)$ to represent terms of $(p,\ttX_p)$. For notational convenience, we simply write $p$ in place of $(p,\ttX_p)$.
Moreover, we have
\begin{equation*}
    \Lin(\pi)(\alpha)
    \cong (T_{\pi[\alpha,\varnothing]},<_{\KB}).
\end{equation*}
Consequently, if $\pi$ is a $\rmbeta$-proof in $\sfu\RS$, then $\Lin(\pi)$ is a dilator. Moreover, if $\pi$ is recursive, then $\Lin(\pi)$ is also recursive.

\begin{lemma} \label{Lemma: Embedding a linearized uRS proof}
    Let $\pi$ be a $\rmbeta$-proof in $\sfu\RS$ whose last sequent takes the form $[\varnothing]\vdash \Gamma$.
    Put $G = \Lin(\pi)\circ \varphi_2$ and $H = \varphi_2\circ (\Id+1+G)$. 
    Then we have an $\RS_H$-proof of $H[\varnothing] \sststile{\ttOmega+\omega}{\lag\rag}\Gamma$, where $\lag\rag$ is understood as a nullary term of $G = \Lin(\pi)\circ\varphi_2$.
\end{lemma}
\begin{proof}
    Put $L=\Lin(\pi)$ and $K=\varphi_2$,
    
    Recall that each $p\in T_\pi$ corresponds to a sequent $[\ttX_p]\vdash\Gamma_p$ and determines a term of $\Lin(\pi)$.
    For a template order $\ttY$ and an increasing
    map $f\colon\ttX_p\to \varphi_2(\ttY)$, we define $\tau_{p,f}$ as $f[p]\in G(\ttY)$ composed with the following map 
    \begin{equation*}
        G(\ttY)
        \longrightarrow
        \ttY+1+G(\ttY)
        \xrightarrow{x\mapsto \varphi_2(x)}
        H(\ttY).
    \end{equation*}
    In particular, we have $H(\ttY)\vDash \ttOmega\cdot\omega^m<\tau_{p,f}$ for every $m<\omega$.
    Again, let $f[\Gamma_p]$ denote the sequent obtained from $\Gamma_p$ by applying $\bfL(u)_f$ to its set terms.
    We construct, by coinduction on $\pi$, an $\RS_H$-preproof witnessing
    \begin{equation} \label{Formula: Embedding linearized proof invariant}
        H[\ttY]
        \sststile{\ttOmega+\omega}{\tau_{p,f}}
        f[\Gamma_p]
    \end{equation}
    for every $p\in T_\pi$, every template order
    $\ttY$, and every increasing
    $f\colon\ttX_p\to K(\ttY)$.
    Suppose first that the inference at $p$ is not the consequence of one of Pairing, Union, $\Delta_0$-Separation, or Infinity.
    For a finitary inference, we reproduce it in $\RS_H$.
    If $q$ is an immediate premise of $p$, then $\ttX_q=\ttX_p$, and the definition of $\Lin(\pi)$ gives $G(\ttY)\vDash f[q]<f[p]$, so $H(\ttY)\vDash \tau_{p,f}<\tau_{q,f}$, as required by the height condition of $\RS_H$.
    All set terms occurring in $f[\Gamma_p]$ have their indices below $\ttOmega$. Hence their ranks are below $\ttOmega$.
    More generally, every formula arising from the fixed finite syntax of $\pi$ has rank below $\ttOmega+\omega$. Thus the witness-rank conditions for existential inferences, the cut-rank conditions, and the height condition for $(\mathrm{Refl})$ are satisfied.

    Now consider the case that $[\ttX_p]\vdash \Gamma_p$ in $\pi$ is a consequence of $(\forall)$ at $p$, so an excerpt of $\pi$ at $p$ takes the following form:
    \begin{equation*}
    \begin{mathprooftree}
        \def\fCenter{\vdash}
        \AxiomC{$ \cdots \quad [\ttX_{p^\frown\lag s\rag}]\fCenter \Gamma'_p,\ B(s) \quad \cdots$ ($\heartsuit$)}
        \RightLabel{$(\forall)$}
        \UnaryInfC{$[\ttX_p]\fCenter \Gamma'_p,\ \forall x B(x)$}
    \end{mathprooftree}
    \end{equation*}
    Here $p^\frown\lag s\rag$ denotes the premise indexed by $(\ttX_{p^\frown\lag s\rag},s)$.
    Fix a template order $\ttY$ and an increasing map
    $f\colon\ttX_p\to\varphi_2(\ttY)$.
    We want to apply the $(\forall)$-rule of $\RS_H$
    to obtain
    \begin{equation*}
        H[\ttY]
        \sststile{\ttOmega+\omega}{\tau_{p,f}}
        \Gamma'_p[f],\ \forall x B(x)[f].
    \end{equation*}
    
    For this purpose, take an arbitrary pair
    $(\ttZ,r)$ satisfying $(\spadesuit)$ for the
    $(\forall)$-rule over $\ttY$.
    Since the ordinal indices of $r$ lie below
    $\ttOmega$, we may regard $r$ as a term of
    $\bfL(u)_{\varphi_2(\ttZ)}$.
    Now extend $f\colon \ttX_p\to\varphi_2(\ttY)$ to an increasing map $g\colon\ttX_{p^\frown\lag s\rag} \to\varphi_2(\ttZ)$ by introducing exactly the new parameters required by $r$, following the canonical
    convention in $(\heartsuit)$.
    More precisely, we choose the extension and $s\in\bfL(u)_{\ttX_{p^\frown\lag s\rag}}$ so that
    \begin{itemize}
        \item $g\restriction\ttX_p = f$,
        \item $\ran g=\ran f \cup\supp^{\bfL(u)}_{\varphi_2(\ttZ)}(r)$,
        \item $\bfL(u)_g(s)=r$.
    \end{itemize}
    The existence of such an extension follows from the support condition for $\bfL(u)$.
    In particular, $(\ttX_{p^\frown\lag s\rag},s)$ satisfies $(\heartsuit)$, so the corresponding hypothesis occurs in the displayed $(\forall)$-inference of $\pi$.
    By the coinductive construction, we have an $\RS_H$-preproof of
    \begin{equation*}
        H[\ttZ]
        \sststile{\ttOmega+\omega}
            {\tau_{p^\frown\lag s\rag,g}}
        g[\Gamma'_p],\ B(r).
    \end{equation*}
    Furthermore, since $p^\frown\lag s\rag$ is an immediate successor of $p$, the definition of the Kleene-Brouwer order gives $\Lin(\pi)\bigl(\varphi_2(\ttZ)\bigr) \vDash g[p^\frown\lag s\rag]<f[p]$.
    Consequently, $H(\ttZ)\vDash\tau_{p^\frown\lag s\rag,g}<\tau_{p,f}$.
    Thus the required height inequality holds. Since $(\ttZ,r)$ was arbitrary, we may apply the $(\forall)$-rule of $\RS_H$.    
    The cases $(\rmb\forall)$ and $(\notin)$ are analogous, with the additional condition $|s|<|t|$, which is preserved and reflected by the increasing maps involved.

    Finally, suppose that $[\ttX_p]\vdash\Gamma_p$ is the consequence of one of Pairing, Union, Infinity, or $\Delta_0$-Separation.
    By the embedding results established in \autoref{Section: Embedding}, there is an $m<\omega$, depending only on the fixed axiom instance, such that $H[\ttY]
        \sststile{\ttOmega+m}{\ttOmega\cdot\omega^m} f[\Gamma_p]$.
    Since $\ttOmega+m<\ttOmega+\omega$ and $\ttOmega\cdot\omega^m<\tau_{p,f}$, it also witnesses
    \eqref{Formula: Embedding linearized proof invariant}.
    This completes the coinductive construction.

    Taking $p=\langle\rangle$, $\ttY=\varnothing$, and $f=\varnothing$, we obtain $H[\varnothing]\sststile{\ttOmega+\omega}{\langle\rangle} \Gamma$, as desired.
\end{proof}


The next theorem is one of the important results in Girard's $\Pi^1_2$-logic. We provide an alternative proof of Girard's theorem by employing the functorial cut-elimination of $\KP$.
\begin{theorem}[Girard {\cite[Theorem 11.B.1]{Girard1982Logical2}}\footnote{See also \cite[Remark 11.B.3(i)]{Girard1982Logical2}.}]
    Suppose that $f\colon \omega_1^\CK\to \omega_1^\CK$ is a total function that is $\Sigma_1^{L_{\omega_1^\CK}}$-definable.
    Then there is a recursive dilator $F$ such that for every $\xi<\omega_1^\CK$, $f(\xi)< F(\xi)$.
\end{theorem}

Girard's theorem will follow from the following result:
\begin{theorem}
    Suppose that $f\colon L_{\omega_1^\CK}\to L_{\omega_1^\CK}$ is a $\Sigma_1^{L_{\omega_1^\CK}}$-definable total function.
    Then there is a recursive dilator $F$ such that for every $x\in L_{\omega_1^\CK}$, $f(x)\in L_{F(\rank_L (x))}$, where $\rank_L (x) = \min\{\alpha \mid x\in L_{\alpha+1}\}$ for $x\in L$.
\end{theorem}
\begin{proof}
    Let $\sigma$ be the $\Sigma$-assertion stating `There is an admissible ordinal greater than $\omega$' and $\phi(x,y,z)$ is a $\Sigma_1$-formula such that $\phi(x,y,p)$ for some $p\in L_{\omega_1^\CK}$ defines $f$ over $L_{\omega_1^\CK}$. That is, we have $L_{\omega_1^\CK}\vDash \forall x \exists! y \phi(x,y,p)$ and $L_{\omega_1^\CK}\vDash \phi(x,f(x),p)$ for all $x\in L_{\omega_1^\CK}$. 
    By \cite[Theorem II.5.14]{Barwise1975}, $p\in L_{\omega_1^\CK}$ implies that there is a $\Sigma_1$-formula $\psi(x)$ such that $\psi(x)$ defines $p$ over every admissible set.
    
    Then for every admissible $\alpha$, we have
    \begin{equation*}
        L_\alpha \vDash \sigma \lor \forall x\exists p\bigl(\psi(p)\land \exists y \phi(x,y,p)\bigr).
    \end{equation*}
    By the completeness of $\sfu\RS$, we have a recursive $\rmbeta$-proof $\pi$ of the sequent $[\varnothing]\vdash \sigma, \forall x\exists p \bigl(\psi(p)\land \exists y \phi(x,y,p)\bigr)$. 
    By \autoref{Lemma: Embedding a linearized uRS proof}, we can find a recursive dilator $G$ such that the following holds.
    \begin{equation*}
        \varphi_2(\Id+1+G) [\varnothing] \sststile{\ttOmega+\omega}{\tau_\bot}\sigma\lor \forall x \bigl( \exists p \psi(p)\land \exists y  \phi(x,y,p)\bigr).
    \end{equation*}
    
    Now, by \autoref{Proposition: First Cut Elimination Case 1} and the Functorial Collapsing with the choice $\eta = 0$, we have 
    \begin{equation*}
        \varphi_2\bigl(\Id+\psi_\Omega \varphi_2\bigl(\Id+\Omega+1+G(\Id+\Omega)\bigr)+1\bigr) [\varnothing] \sststile{\psi \hat{\tau}_\bot}{\psi \hat{\tau}_\bot}\sigma\lor \forall x \in \ttL_{\psi(0)}\bigl( \exists p \psi(p)\land \exists y  \phi(x,y,p)\bigr).
    \end{equation*}
    We warn that, unlike its writing, $\psi(0)$ is not a fixed ordinal. $\psi(0)$ (More precisely, $\varphi_1\bigl(\psi(0)\bigr)$) is a nullary term of a dilator $\varphi_2\bigl(\Id+\psi_\Omega \varphi_2\bigl(\Id+\Omega+1+G(\Id+\Omega)\bigr)+1\bigr)$. Now we take
    \begin{equation*}
        F(X) = \varphi_2\bigl(X+\psi_\Omega \varphi_2\bigl(X+\Omega+1+G(X+\Omega)\bigr)+1\bigr).
    \end{equation*}
    $F$ is a recursive dilator, and moreover, the map $x\mapsto E_x$ maps $X$ to $F(X)$ below $\psi(0)\in F(X)$.
    We note that for an ordinal $\alpha$, if $f\colon F(\alpha)\cong \beta$ is an isomorphism to some ordinal $\beta$, then $f(\psi(0)) = \varphi_1(\alpha)$. 
    
    Now let $\alpha$ be an ordinal. By Soundness, we have that $L_{F(\alpha)}\vDash \sigma\lor \forall x\in L_{\varphi_1(\alpha)} \exists y \phi(x,y)$. In particular, we have that if $\alpha<\omega_1^\CK$, then 
    \begin{equation*}
        \forall x\in L_\alpha \bigl(\exists p\in L_{F(\alpha)} \psi^{L_{F(\alpha)}}(p)\land \exists y\in L_{F(\alpha)} \phi^{L_{F(\alpha)}}(x,y,p)\bigr).
    \end{equation*}
    Here, we note that $\Sigma$-formulas are upward absolute, so $\psi^{L_{F(\alpha)}}(p)$ implies $\psi(p)$, and such $p$ is unique in $L$. We also have $\forall x\in L_\alpha \exists y\in L_{F(\alpha)} \phi(x,y,p)$, so we get $f(x)\in L_{F(\alpha)}$.
\end{proof}

We can generalize the above result as Girard stated in \cite[Theorem 11.B.6]{Girard1982Logical2}. Let $s_0$ be the least parameter-free $\Sigma^1_1$-reflecting ordinal. 
We also recall the fact by Aczel-Richter \cite[Theorem 6.2]{AczelRichter1974IDReflAdm}: If $\phi(\vec{v})$ is a $\Pi^1_1$-formula in the language of set theory, then there is a $\Sigma_1$-formula $\phi^+(X,\vec{x})$ such that for every non-empty countable transitive set $A$ and every admissible $M\ni A$, if $\vec{a}\in A$ then
\begin{equation*}
    A \vDash \phi(\vec{a}) \iff M\vDash \phi^+(A,\vec{a}).
\end{equation*}
Hence, if $\alpha<s_0$, then we can find a $\Pi_1$-formula $\chi(x)$ satisfying the following: $\alpha$ is the least ordinal such that for every admissible $M\ni \alpha$, $M\vDash \chi(\alpha)$ holds. The argument presented in \cite[p. 189--190]{Gostanian1979NextAdmissible} shows that we can find a $\Pi_1$-formula $\chi(x)$ such that $\alpha$ is the least ordinal satisfying $L_{\alpha^+}\vDash \chi(L_\alpha)$.
Based on this fact, we state the following result:
\begin{theorem}
    Let $\alpha<s_0$ and $f\colon L_{\alpha^+}\to L_{\alpha^+}$ be a $\Sigma_1^{L_{\alpha^+}}$-definable total function with parameter $\alpha$. Then there is a recursive dilator $F$ such that for every $x\in L_{\alpha^+}\setminus L_\alpha$, we have $f(x)\in L_{F(\rank_L (x))}$.
\end{theorem}
\begin{proof}
    For $\alpha<s_0$, let us fix a $\Pi_1$-formula $\chi(x)$ such that $\alpha$ is the least ordinal satisfying $L_{\alpha^+}\vDash \chi(L_\alpha)$.
    Let
    \begin{itemize}
        \item $\sigma_0$ is the statement ``There is an ordinal $\gamma$ such that $\chi(L_\gamma)$, and every admissible ordinal $\beta$ satisfies $\beta\le\gamma$.''
        \item $\sigma_1$ is the statement ``For every admissible ordinal $\beta$ and $\gamma<\beta$, if $\gamma^+=\beta$, then $L_\beta\vDash \lnot \chi(L_\gamma)$.''
    \end{itemize}
    $\sigma_0$ is the $\Sigma_2$-assertion stating that we have $\alpha$, and the universe has height $\le\alpha^+$. $\sigma_1$ is the $\Pi_1$-assertion stating that no $\gamma$ satisfies $L_{\gamma^+}\vDash\chi(L_\gamma)$ if $\gamma^+$ exists. Clearly, $\alpha^+$ is the unique admissible ordinal $\kappa$ satisfying $\sigma_0\land \sigma_1$. 
    Now suppose that $\phi(x,y,p)$ is a $\Sigma_1$-formula such that $\phi(x,y,\alpha)$ over $L_{\alpha^+}$ defines $f$. Then the following holds for every admissible $L$-model:
    \begin{equation*}
        \lnot \sigma_0 \lor \lnot\sigma_1 \lor \forall \xi\in\Ord [\lnot \chi(L_\xi) \lor \forall x \exists y \phi(x,y,\xi)].
    \end{equation*}
    Hence the following sequent has a recursive $\rmbeta$-proof $\pi$ in $\sfu\RS$:
    \begin{equation*}
        \lnot \sigma_0, \lnot\sigma_1, \forall \xi\in\Ord [\lnot \chi(L_\xi) \lor \forall x \exists y \phi(x,y,\xi)].
    \end{equation*}
    By \autoref{Lemma: Embedding a linearized uRS proof}, we have a recursive dilator $G$ such that
    \begin{equation*}
        \varphi_2(\Id+1+G) [\varnothing] \sststile{\ttOmega+\omega}{\tau_\bot}\lnot\sigma_0,\lnot\sigma_1, \forall \xi \forall x \bigl[\xi\in\Ord\land [\lnot \chi(L_\xi) \lor \exists y  \phi(x,y,\xi)]\bigr].
    \end{equation*}
    Observe that every formula in the sequent takes the form $\forall \Sigma$, so we can apply \autoref{Proposition: First Cut Elimination Case 1} and the functorial cut elimination to get
    \begin{multline*}
        \varphi_2\bigl(\Id+\psi_\Omega \varphi_2\bigl(\Id+\Omega+1+G(\Id+\Omega)\bigr)+1\bigr)[\varnothing] \\ 
        \sststile{\psi \hat{\tau}_\bot}{\psi \hat{\tau}_\bot} (\lnot\sigma_0)^{\psi(0)},\lnot\sigma_1, \forall \xi \forall x\in \ttL_{\psi(0)} \bigl[\xi\in\Ord\land [(\lnot \chi)(L_\xi) \lor \exists y  \phi(x,y,\xi)]\bigr].
    \end{multline*}
    Let us take $F(X) = \varphi_2\bigl(X+\psi_\Omega\varphi_2\bigl(X+\Omega+1+G(X+\Omega)\bigr)\bigr)$. The topmost $+1$ does not appear in the proof, so we may drop it. Again, we note that if $f\colon F(\beta)\cong \gamma$ is an order isomorphism, then $f(\psi(0))=\varphi_1(\beta)$. Soundness shows that for every $\beta<\alpha^+$
    \begin{equation*}
        L_{F(\beta)}\vDash (\lnot\sigma_0)^{\varphi_1(\beta)} \lor \lnot\sigma_1 \lor \forall\xi\in \Ord \forall x\in L_{\varphi_1(\beta)} [\lnot \chi(L_\xi) \lor \exists y \phi(x,y,\xi)].
    \end{equation*}
    Since $F$ is recursive and $\beta<\alpha^+$, we have $F(\beta)<\alpha^+$. Now suppose that $\beta>\alpha$. 
    Observe that $(\lnot \sigma_0)^{\varphi_1(\beta)}$ is the assertion ``For every $\gamma<\varphi_1(\beta)$, if $\chi(L_\gamma)$, then there is an admissible $\delta>\gamma$.'' This is a $\Sigma$-assertion, so $L_{F(\beta)}\vDash (\lnot\sigma_0)^{\varphi_1(\beta)}$ implies $L_{\alpha^+}\vDash (\lnot\sigma_0)^{\varphi_1(\beta)}$. 
    However, if $\beta>\alpha$, then $\alpha$ is the only ordinal $\gamma<\varphi_1(\beta)$ satisfying $\chi(L_\gamma)$ over $L_{\alpha^+}$, so $L_{\alpha^+}\vDash \lnot(\lnot\sigma_0)^{\varphi_1(\beta)}$.
    This shows that $L_{F(\beta)}\vDash \lnot(\lnot\sigma_0)^{\varphi_1(\beta)}$.
    Similarly, we can see that $L_{F(\beta)}\vDash \sigma_1$. This shows that for every $\beta>\alpha$,
    \begin{equation*}
        L_{F(\beta)}\vDash \forall x\in L_{\varphi_1(\beta)} \exists y \phi(x,y,\alpha).
    \end{equation*}
    In particular, for every $\beta\ge\alpha$, we have $\forall x\in L_{\beta+1} \exists y\in L_{F(\beta+1)} \phi(x,y,\alpha)$. Hence, if $\rank_L(x)=\beta\ge\alpha$, then $f(x)\in L_{F(\beta+1)}$, so $F(\cdot+1)$ is a desired recursive dilator.
\end{proof}

\printbibliography

\end{document}